\documentclass[10pt]{article}
\usepackage[english,activeacute]{babel}
\usepackage{epsfig}

\usepackage{stmaryrd}
\usepackage{amsmath,amsfonts,amssymb,mathrsfs,amsthm}
\usepackage{xcolor}
\usepackage{dsfont}
\usepackage{graphicx}
\usepackage[mathscr]{eucal}
\usepackage{makeidx}
\usepackage{verbatim}
\usepackage{graphics,graphicx}
\usepackage{tikz}
\usetikzlibrary{patterns}
\usepackage{textcomp}
\usepackage{float}
\usepackage[colorlinks=true]{hyperref}
\newtheorem{lemma}{Lemma}[section]
\newtheorem{theorem}[lemma]{Theorem}
\newtheorem{proposition}[lemma]{Proposition}
\newtheorem{corollary}[lemma]{Corollary}

\theoremstyle{definition}

\newtheorem{definition}[lemma]{Definition}
\newtheorem{remark}[lemma]{Remark}

\makeatletter
\def\keywords{
    \vspace{1ex}
    \noindent
    \if@twocolumn
      \small{\bf  Keywords}\/---$\!$    \else
      \begin{center}\small\ {\bf Keywords}\end{center}\quotation\small
    \fi}
\def\endkeywords{\vspace{0.6em}\par\if@twocolumn\else\endquotation\fi
    \normalsize\rm}
\makeatother

\newcommand{\G}{\ensuremath{\mathcal G}}

\renewcommand{\L}{\ensuremath{\mathcal L}}

\newcommand{\ID}{\ensuremath{\mathbb D}}
\newcommand{\x}{\ensuremath{\bf x}}

\newcommand{\X}{\ensuremath{\mathcal X}}
\newcommand{\Y}{\ensuremath{\mathcal Y}}

\DeclareMathOperator{\loc}{loc}
\DeclareMathOperator{\comp}{comp}

\DeclareMathOperator{\Vol}{Vol}

\DeclareMathOperator{\Int}{Int}

\DeclareMathOperator{\Eucl}{Eucl}

\DeclareMathOperator{\Hess}{Hess}

\DeclareMathOperator{\Tr}{Tr}

\newcommand{\mb}[1]{\ensuremath{\mathbb{#1}}}
\newcommand{\N}{{\mb{N}}}

\newcommand{\R}{{\mb{R}}}
\newcommand{\C}{{\mb{C}}}

\newcommand{\w}{\ensuremath{w}}

\newcommand{\eps}{\varepsilon}

\newcommand{\M}{\ensuremath{\mathcal M}}

\newcommand{\B}{\ensuremath{\mathcal B}}

\let \Re \relax
\DeclareMathOperator{\Re}{Re}
\let \Im \relax
\DeclareMathOperator{\Im}{Im}

\newcommand{\ovl}[1]{\overline{#1}}

\DeclareMathOperator{\supp}{supp}

\DeclareMathOperator{\dist}{dist}

\DeclareMathOperator{\op}{op}

\DeclareMathOperator{\id}{Id}

\renewcommand{\d}{\ensuremath{\partial}}
\newcommand{\wt}{\ensuremath{\widetilde}}

\newcommand{\dsp}{\displaystyle}

\let \div \relax
\DeclareMathOperator{\div}{div}

\def\x{x}

\newcommand{\E}{\mathscr E}

\def\e{{\mu}}

\newcommand\bna{\begin{eqnarray*}} 
	\newcommand\ena{\end{eqnarray*}}

\newcommand\bnan{\begin{eqnarray}} 
	\newcommand\enan{\end{eqnarray}}

\newcommand\bnp{\begin{proof}} 
	\newcommand\enp{\end{proof}}

\newcommand\bneq{\begin{eqnarray*}\left\lbrace \begin{array}{rcl}}
		\newcommand\eneq{\end{array} \right.\end{eqnarray*}}
\newcommand\bneqn{\begin{eqnarray}\left\lbrace \begin{array}{rcl}}
		\newcommand\eneqn{\end{array} \right.\end{eqnarray}}

\newcommand{\Ht}[1]{H^{#1}_{\tau}}

\newcommand\nor[2]{\left\|#1\right\|_{#2}}

\newcommand\ubar{\overline{u}}

\newcommand\chit{\widetilde{\chi}}

\renewcommand{\kappa}{\nu}

\newcommand{\tf}{\tilde{f}}

\newcommand{\dzb}{\partial_{\bar{z}}}
\newcommand{\zb}{\bar{z}}

\newcommand{\bc}{\check{b}_s}
\newcommand{\ban}{\mathcal{B}}

\newcommand{\pot}{\mathsf{q}}
\newcommand{\pg}{P_{\mathsf{b},\pot}}

\newcommand{\dt}{\langle D_t \rangle}

\newcommand{\kerone}{\mathcal{K}_{1,h}}
\newcommand{\kertwo}{\mathcal{K}_{2,h}}

\newcommand{\xg}{{\ensuremath{\bf x}}}

\newcommand\gl[2]{\left\langle #1, #2\right\rangle_g}
\newcommand\gln[1]{\left| #1 \right|_g^{2}}
\newcommand\Lap{\Delta_{g}}
\newcommand\nablag{\nabla_{g}}

\def\x{{\bf x}}
\def\q{{\mathsf q}}

\numberwithin{equation}{section}

\author{Spyridon Filippas\footnote{Department of Mathematics and Statistics, University of Helsinki, Helsinki, Finland, email: spyridon.filippas@helsinki.fi},
	Camille Laurent\footnote{CNRS UMR 7598 and Sorbonne Universit\'es UPMC Univ Paris 06, Laboratoire Jacques-Louis Lions, F-75005, Paris, France, email: camille.laurent@sorbonne-universite.fr}, and Matthieu L\'eautaud\footnote{Laboratoire de Math\'ematiques d'Orsay, UMR 8628, Universit\'e Paris-Saclay, CNRS, B\^atiment 307, 91405 Orsay Cedex France, email: matthieu.leautaud@universite-paris-saclay.fr.} \footnote{Institut Universitaire de France, Paris, France.}
}

\makeatletter
\def\keywords{
	\vspace{1ex}
	\noindent
	\if@twocolumn
	\small{\bf  Keywords}\/---$\!$    \else
	\begin{center}\small\ {\bf Keywords}\end{center}\quotation\small
	\fi}
\def\endkeywords{\vspace{0.6em}\par\if@twocolumn\else\endquotation\fi
	\normalsize\rm}
\makeatother

\begin{document}
	\title{Unique continuation for Schr{\"o}dinger operators with partially Gevrey coefficients}
	\maketitle
	
	\begin{abstract}
		We prove a local unique continuation result for Schr{\"o}dinger operators with time independent Lipschitz metrics and lower order terms which are Gevrey $2$ in time and bounded in space. This implies global unique continuation from any open set in a connected Riemannian manifold. 
		These results relax in the same geometric setting the analyticity assumption in time of the Tataru-Robbiano-Zuily-H\"ormander theorem for these operators.
		The proof is based on $(i)$ a Tataru-Robbiano-Zuily-H\"ormander type Carleman estimate with a nonlocal weight adapted to the anisotropy of the Schr{\"o}dinger operator and $(ii)$ the description of the  conjugation of the Schr{\"o}dinger operator with Gevrey coefficients by this nonlocal weight. 
	\end{abstract}
	

	\tableofcontents
	\setcounter{tocdepth}{1}

	\section{Introduction and main results}
	
	\subsection{Background and results}
	In this article we are interested in the {\em unique continuation} problem for a large family of time-dependent Schr\"{o}dinger operators. For a general differential operator 
	\begin{align}
		\label{e:def-P}
		P = \sum_{|\alpha| \leq m} a_\alpha(\x)D_\x^\alpha , \quad \text{where } D_{\x_j}=  \frac{\d_{\x_j}}{i}, \quad m\in \N ,
	\end{align}
	on an open set $\Omega \subset \R^n$ the problem of {\em local} unique continuation is the following question: given $\x_0 \in \Omega \subset \R^n$ and $S\ni \x_0$ a smooth oriented hypersurface, do we have:
	\begin{equation}
		\label{e:local-UC}
		\big(   Pu=0 \textnormal{ in } \Omega, \quad u=0 \textnormal{ in } S^{-}\cap \Omega \big)  \Longrightarrow \x_0 \notin \supp (u),
	\end{equation}
	where we denote by $S^{-}$ one side of the oriented hypersurface $S$? If the local unique continuation property holds for a sufficiently large family of hypersurfaces, one can propagate it and deduce a {\em global} result. For $\omega$ a small subset of $\Omega$ such a result takes the form:
	\begin{equation}
		\label{e:global-UC}
		\big(   Pu=0 \textnormal{ in } \Omega, \quad u=0 \textnormal{ in } \omega\big)  \Longrightarrow u=0 \textnormal{ in }\Omega.
	\end{equation}
	A motivation arising from control theory is described in Section~\ref{s:application} below (see also~\cite{LL:23notes}).
	On the one hand, the Holmgren-John theorem~\cite[Theorem 5.3.1]{Hoermander:63} yields unique continuation assuming all coefficients of $P$ (i.e. all $a_\alpha$'s for all $|\alpha |\leq m$ in~\eqref{e:def-P}) are real-analytic and the hypersurface  $S$ is non-characteristic, that is to say
	\begin{align}
		\label{e:non-char}
		p_m(\x_0, d \Psi(\x_0)) \neq 0, \quad \text{ where } S = \{\Psi = 0\} , 
	\end{align} 
	and  \begin{align}
		\label{e:ppal-sym}
		p_m(\x, \xi) := \sum_{|\alpha| = m} a_\alpha(\x) \xi^\alpha
	\end{align}
	is the so-called principal symbol of the operator $P$.
	On the other hand, if one is interested in $C^\infty$ (or $C^k$) regularity, H\"{o}rmander's theorem~\cite[Theorem 28.3.4]{Hoermander:V4} yields unique continuation under a (rather strong, unless if $P$ is elliptic, a case which is not considered in the present article) so-called pseudoconvexity condition (that is to be checked on the whole cotangent space over $\x_0$). 
	The seminal result of Robbiano~\cite{Robbiano:91} for hyperbolic operators, subsequently improved in~\cite{Hormander:92}, paved the way to a more general theorem that would bridge the gap between the $C^\infty$ and the analytic case. Following another breakthrough by Tataru~\cite{Tataru:95}, this program was finally completed by Robbiano-Zuily, H\"{o}rmander and Tataru in the series of papers~\cite{RZ:98, Hor:97,Tataru:99}, proving a general unique continuation result for operators having partially analytic coefficients, containing as a particular cases both the Holmgren-John and the H\"{o}rmander theorems. We refer to~\cite{LL:15,LL:XEDP,LL:book,LL:23notes} for further discussions and comments on these results.
	
\medskip
In this article, motivated by applications to control theory (see Section~\ref{s:application}  below), we are interested in the particular case of Schr\"odinger operators 
\begin{align}
	\label{laplacien-perturbe-gevrey}
	\pg = i \d_t +\sum_{j,k=1}^d \d_{x_j} g^{jk}(x) \d_{x_k}+\sum_{j=1}^d \mathsf{b}^j(t,x)\d_{x_j}+\pot(t,x) ,
\end{align}
where $g^{jk}(x)$ is a symmetric elliptic matrix on an open set $V \subset \R^d$, that is to say $g^{jk}(x)=g^{kj}(x)$ and 
\begin{align}
\label{e:elliptic}
\text{there is }c_0>0 \text{ such that } \sum_{j,k=1}^d g^{jk}(x)\xi_j\xi_k \geq c_0 |\xi|^2, \quad \text{ for all }(x,\xi) \in V \times \R^d .
\end{align}
Compared to the general situation in~\eqref{e:def-P}--\eqref{e:ppal-sym}, we have here $n=1+d$, $\x = (t,x)$, $m=2$, and the ``principal symbol'' of $P$ is $p_2(\x,\xi) =p_2(t,x,\xi_t,\xi_x) =- \sum_{j,k} g^{jk}(x)\xi_{x_j} \xi_{x_k}$. The latter does not depend on $\xi_t$ (and, in particular, is the same as for the heat operator~\eqref{laplacien-perturbe-gevrey} in which $i\d_t$ is replaced by $-\d_t$). The formulation of $\pg$ in divergence form, as opposed to~\eqref{e:def-P}, is related to the low regularity of the coefficients in our results, see the discussion in Section~\ref{s:rem-div}.
The classical theorem of H\"{o}rmander is empty in this case. 
Taking advantage of the anisotropic (or quasi-homogeneous) nature of the Schr\"odinger operator, Lascar and Zuily proved in~\cite{LZ:82} that the results of H\"{o}rmander~\cite[Chapter~28]{Hoermander:V4} can be generalized to the anisotropic case with an appropriate modification of the symbol classes and Poisson bracket. See also~\cite{Dehman:84},~\cite{isakov1993carleman} and~\cite{Tataru:96ani} for later results in this direction. 
In the context of~\eqref{laplacien-perturbe-gevrey}, this result applies for coefficients $g^{jk}\in C^1$ and $\mathsf{b}^j,\pot \in L^\infty$, under a pseudoconvexity condition on the hypersurface. 
The latter is a very strong {\em local} geometric assumption on the surface for~\eqref{e:local-UC} to hold, which necessarily leads to a very strong {\em global} geometric assumption of the observation set $\omega$ in an associated global unique continuation statement of the form~\eqref{e:global-UC}.
For applications to control or inverse problems, related global Carleman estimates for Schr\"odinger operators have been proved for instance in~\cite{BP:02} (constant leading order coefficients) and in~\cite{TX:07,L:10} (Riemannian manifolds or varying coefficients). A weak pseudoconvexity condition has also been proved sufficient in~\cite{MOR:08} for a flat metric and in \cite{L:10} with varying metrics. Yet, in all of these references, a form of pseudoconvexity related to that of~\cite{LZ:82} is required and global statements hold under strong geometric assumptions. As proved in~\cite[Th\'eor\`emes~1.4 et~1.6]{LZ:82}, a pseudoconvexity condition is actually essentially {\em necessary} in the following sense: if it is ``strongly violated'', then there exists $\pot \in C^\infty(\Omega)$ such that \eqref{e:local-UC} does not hold for the operator $P=\pg$ in~\eqref{laplacien-perturbe-gevrey} with $\mathsf{b}^j=0$ (see Section~\ref{s:section-remarks-gevrey} below).

The Tataru-Robbiano-Zuily-H\"{o}rmander theorem also applies to the Schr\"odinger operator~\eqref{laplacien-perturbe-gevrey}. In that case, it implies local unique continuation~\eqref{e:local-UC} assuming 
\begin{enumerate}
	\item that the surface $S$ is {\em non-characteristic}, i.e.~\eqref{e:non-char};
	\item that the coefficients are {\em real-analytic} with respect to the time variable $t$. 
\end{enumerate}
In the setting of the Schr\"odinger operator~\eqref{laplacien-perturbe-gevrey} in $\R^{1+d}$, note that the non-characteristicity assumption~\eqref{e:non-char} rewrites equivalently
\begin{equation}
	\label{e:non-char-schro}
	\sum_{j,k=1}^d  g^{jk}(x_0)\d_{x_j} \Psi(t_0,x_0)\d_{x_k} \Psi(t_0,x_0) \neq 0 .
\end{equation}
From the geometric point of view, the non-characteristicity assumption is optimal: it excludes only surfaces tangent to $\{t=t_0\}$, for which we know that local unique continuation may fail (this would otherwise imply finite speed of propagation for Schr\"odinger equations). Applied iteratively to appropriate families of hypersurfaces (see e.g.~\cite[Section~6.2]{LL:15}), this result thus leads to a global unique continuation statement under an optimal geometric condition, still assuming analyticity in time of the coefficients. 
From the point of view of regularity requirements, however, analyticity in time is of course very demanding. 

Note finally that T'jo{\"e}n~\cite{tjoenAnistrop} proved a quasi-homogeneous variant of the Tataru-H\"{o}rmander-Robbiano-Zuily theorem in a general setting and Masuda~\cite{masuda} proved a global uniqueness result in the case of $C^2$ principal coefficients and time independent coefficients. 
A challenging problem is to understand to which extent the time-analyticity condition can be relaxed under optimal geometric conditions. For the wave operator, we refer to the discussion in~\cite{LL:23notes} and the counterexamples of Alinhac-Baouendi~\cite{AB:79,Alinhac:83,AB:95} and H\"ormander~\cite{Hormander:00} (see Section~\ref{s:section-remarks-gevrey} below).
In this direction, our results relax the time analyticity assumption of the Tataru-Robbiano-Zuily-H\"{o}rmander theorem for the Schr\"odinger operator~\eqref{laplacien-perturbe-gevrey} down to Gevrey regularity. 

\begin{definition}
	\label{d:gevrey}
	Given $\mathsf{d}\in \N^*$, $U\subset \R^{\mathsf{d}}$ an open set, $(\ban, \| \cdot \|_{\ban})$ a Banach space and $s>0$,  we say that $f$ is a $s$-Gevrey function valued in $\ban$, denoted $f \in \mathcal{G}^s(U;\ban)$, 
	if $f\in C^\infty(U;\ban)$ is such that  for every compact set $K\subset U$, there are constants $C,R>0$ such that for all $\alpha \in \N^{\mathsf{d}}$  
	\begin{equation*}
		\max_{t \in K} \nor{\d^\alpha f(t)}{\ban}\leq C R^{|\alpha|} \alpha!^s .
	\end{equation*}
\end{definition} 
These spaces were introduced by Gevrey~\cite{gevrey1918nature} to investigate regularity properties for solutions of the heat equation between real-analyticity and $C^\infty$ regularity.
Recall that for $s=1$, $\mathcal{G}^1(U;\ban) = C^\omega (U;\ban)$ is the space of real-analytic $\ban$-valued functions. However, for $s>1$, $\mathcal{G}^s(U;\ban)$ contains nontrivial compactly supported functions. See e.g.~\cite{Hoermander:V1} or~\cite{Rodino:93} for more properties of Gevrey functions.
In what follows, we mostly consider the case $\mathsf{d}=1$, $t$ being the time variable (but also consider $\mathsf{d}=2$ in Section~\ref{s:gevrey-fcts}).
Our main results may be summarized as follows.
\begin{theorem}[Local unique continuation for Schr\"odinger operators]
	\label{theorem local uniqueness intro}
	Assume $\Omega = I \times V$ where $I \subset \R$ is an open interval and $V\subset \R^d$ an open set, and let $(t_0,x_0) \in \Omega$. 
	Assume $g^{jk}\in W^{1,\infty}(V)$ satisfies~\eqref{e:elliptic}, that $\mathsf{b}^j, \pot \in \G^2(I; L^\infty(V;\C))$. Let $\Psi \in C^1(\Omega;\R)$ such that $\{\Psi=0\}$ is non characteristic for $P$ at $(t_0,x_0)$, in the sense of~\eqref{e:non-char-schro}.
	Then, there is a neighborhood $W$ of $(t_0,x_0)$ such that, for $\pg$ defined in~\eqref{laplacien-perturbe-gevrey}, 
	$$
	\left( \pg u = 0 \quad \text{ in }\Omega, \quad  u \in L^2(I;H^1(V)) , \quad u =0 \text{ in } \{\Psi >0\}  \right) \implies u = 0 \text{ in } W .
	$$
\end{theorem}

For applications, one may need to assume less regularity on the solution $u$. The latter can indeed be relaxed, if we assume some additional regularity of the coefficients.
\begin{theorem}[Local unique continuation for $L^2$ solutions]
	\label{theorem local L2}
	Under same assumptions as in Theorem~\ref{theorem local uniqueness intro}, and assuming in addition that $\sum _{j=1}^d\d_{x_j}\mathsf{b}^j \in  L^\infty(\Omega;\C)$,  there is a neighborhood $W$ of $(t_0,x_0)$ such that 
	$$
	\left( \pg u = 0 \quad \text{ in }\Omega, \quad  u \in L^2(\Omega) , \quad u =0 \text{ in } \{\Psi >0\}  \right) \implies u = 0 \text{ in } W .
	$$
\end{theorem}
Note that the divergence form of the principal part of $\pg$ together with the respective regularity assumptions on $g^{jk}, \mathsf{b},\pot$ and $u$ allow to make sense of $\pg u$ in $\mathcal{D}'(\Omega)$.
With respect to  the Tataru-Robbiano-Zuily-H\"{o}rmander theorem for the Schr\"odinger operator~\eqref{laplacien-perturbe-gevrey}, we relax the analyticity-in-time assumption for the lower order terms to a Gevrey $2$ condition. We also relax the regularity of the main coefficients (assumed either $C^\infty$ in~\cite{RZ:98, Hor:97,Tataru:99} or $C^1$ in~\cite{Tataru:95}), replaced here by Lipschitz regularity; in the elliptic context (and therefore in our context as well) this is essentially the minimal regularity in dimension $d \geq 3$ for local uniqueness to hold (see~\cite{Plis:63} and~\cite{Miller:74} for $C^{0, \alpha}$ counterexamples for all $\alpha<1$, for operators in divergence forms or not). 

\begin{remark}
\label{r:distributions}
	One can further lower the regularity of the solution $u$ by assuming additional regularity of the coefficients $g^{ij}, \mathsf{b}^j, \pot$. For instance, assuming (in addition to the assumptions of Theorem~\ref{theorem local uniqueness intro}) that $g^{ij}\in C^\infty(V) ,  \mathsf{b}^j, \pot \in C^\infty(\Omega;\C)$, then we have 
	$$
	\left( \pg u = 0 \quad \text{ in }\Omega, \quad  u \in \mathcal{D}'(\Omega) , \quad u =0 \text{ in } \{\Psi >0\}  \right) \implies u = 0 \text{ in } W .
	$$
\end{remark}
Successive applications of Theorem~\ref{theorem local uniqueness intro} or Theorem~\ref{theorem local L2} through a family of well-chosen non-characteristic hypersurfaces yield the following global result (see~\cite[Proof of Theorem~6.7 p.~100]{LL:15} and use that a connected manifold is path-connected).
\begin{theorem}
	\label{theorem global intro}
	Let $T>0$ and $\M = \Int(\M)\sqcup \d\M$ be a connected smooth manifold with or without boundary $\d\M$.
	Suppose that $g \in W^{1,\infty}_{\loc}(\Int(\M))$ is a Riemannian metric on $\Int(\M)$, 
	that $\pot \in \G^2((0,T); L^\infty_{\loc}(\Int(\M);\C))$, that $\mathsf{b}$ is a one form with all components belonging to $\G^2((0,T); L^\infty_{\loc}(\Int(\M);\C))$, 
	and consider the differential operator 
	$$
	\mathcal{P}_{\mathsf{b},\pot} : = i \d_t +\Delta_g+ \mathsf{b} \cdot \nabla_g+\pot(t,x) ,
	$$
	where $\Delta_g$ is the Laplace-Beltrami operator on $\Int(\M)$, $\nabla_g$ the Riemannian gradient.

	Then given  $\omega$ a nonempty open set of $\M$,  we have 
	$$
	\begin{cases}
		\mathcal{P}_{\mathsf{b},\pot}  u = 0  \text{ in }(0,T)\times \Int(\M) \\  u \in L^2_{\loc}(0,T;H^1_{\loc}(\Int(\M)))  \\ u =0 \text{ in } (0,T)\times \omega    
	\end{cases} \implies u = 0 \text{ in } (0,T)\times \Int(\M) .
	$$
	If in addition $\div_g(\mathsf{b})\in L^\infty_{\loc}((0,T)\times \Int(\M))$, then 
	$$
	\begin{cases}
		\mathcal{P}_{\mathsf{b},\pot}  u = 0  \text{ in }(0,T)\times \Int(\M) \\  u \in L^2_{\loc}((0,T)\times \Int(\M))  \\ u =0 \text{ in } (0,T)\times \omega    
	\end{cases}
	\implies u = 0 \text{ in } (0,T)\times \Int(\M) .
	$$
\end{theorem}
Note that by $\pot \in \G^2((0,T); L^\infty_{\loc}(\Int(\M);\C))$, we mean $\pot \in \G^2((0,T); L^\infty(K;\C))$ for all compact subsets $K$ of $\Int(\M)$.
Note also that under the assumptions of Theorem~\ref{theorem global intro}, the Cauchy problem $\mathcal{P}_{\mathsf{b},\pot}  u=0, u(0, \cdot)=u_0$ is not well-posed in general. 

As in~\cite{LL:23notes} (see Theorem~3.24 and the remark thereafter), this result (for $L^2(I; H^1(V))$ solutions) can also be translated into a global unique condition from an arbitrarily small nonempty open subset of the boundary $\d \M$ (in case $\d\M\neq \emptyset$); we do not state this result for the sake of concision.

We finally mention that other notions of global unique continuation have been extensively investigated for solutions of Schr\"{o}dinger equations during the last years. One such notion is the following: Assume that a solution $u=u(t,x)$ of the Schr\"{o}dinger equation on $\R_{t}\times \R_{x}$ vanishes in $|x|>R$ for some $R>0$ at two different times $t_0$ and $t_1$. Can we then conclude that $u$ vanishes everywhere? This question has been addressed for instance in~\cite{escauriaza2006uniqueness,ionescu2006uniqueness,dong2007unique}, see also the references therein. All of these results hold under stronger geometric assumptions in space (flat metric, nullity outside of a ball), weaker regularity assumptions on the lower order terms, and use as a key tool Carleman inequalities.

\subsection{Application to controllability and observability}
\label{s:application}
\subsubsection{Approximate controllability}
As already alluded, unique continuation properties for evolution equations are often equivalent to approximate controllability results for an appropriate dual problem, see e.g. the introduction of~\cite{LL:23notes} for the wave equation.
In particular, Theorem~\ref{theorem global intro} has an ``approximate controllability'' counterpart. For simplicity of the exposition, we only treat the internal control (the boundary control could be considered as well) of $L^2$ solutions (the case of $C^0H^{-1}$ solutions could be considered as well) with $\mathsf{b}=0$ (the case of general $\mathsf{b}$ could be considered as well, with regularity assumptions depending on the space in which the control problem is set; note that in any case, additional assumptions should be made so that to ensure well-posedness of the Cauchy problem).
Given $T>0$, $\M=\Int(\M)\sqcup\d\M$ a smooth manifold with (possibly empty) boundary, $g$ a locally Lipschitz continuous metric on $\M$, and $\omega \subset \M$ an open set, we consider the control problem
\begin{equation}
	\label{e:Schro-control}
	\left\{
	\begin{array}{rl}
		i \d_t v+\Delta_gv+ \q v = \mathds{1}_\omega f, &  \text{ in } (0,T) \times \Int(\M)  ,\\
		v = 0 ,& \text{ on } (0,T) \times \d\M  \quad \text{ if } \d\M \neq \emptyset , \\
		v(0,\cdot) = v_0 ,&   \text{ in } \Int(\M) .
	\end{array}
	\right.
\end{equation}
Here, $f$ is a control force acting on the system on the small open set $\omega$ and one would like to control the state $v$ of the equation. Concerning well-posedness of the Cauchy problem in~\eqref{e:Schro-control}, we first let  $H^1_0(\M)$ be the completion of $C^1_c(\Int(\M))$ for the norm
\begin{equation}
\label{e:def-H1-norm}
\|u\|_{H^1(\M)}^2 := \int_\M \left( \left| \nablag u\right|_g^2 + |u|^2 \right) d\Vol_g .
\end{equation}
Note that $C^1_c(\Int(\M))$ being dense in $L^2(\M)$, we have a continuous embedding $H^1_0(\M) \subset L^2(\M)$.
Second, we take the Friedrichs extension on $L^2(\M)$ of $-\Delta_g$ defined on $C^\infty_c(\Int(\M))$, which we denote by $-\Delta_{g,\text{F}}$. It is defined by 
\begin{align}
D(-\Delta_{g,\text{F}})& :=\Big\{u\in H^1_0(\M),\text{ there exists }f\in L^2(\M), \nonumber  \\
& \quad   \int_\M \langle \nablag u,\nablag \varphi \rangle_g^2 + u\varphi \ d\Vol_g= \int_\M f\varphi \ d\Vol_g \quad \text{for all } \varphi \in H^1_0(\M) \Big\} ,  \label{e:domaine-}
\end{align}
For $u \in D(-\Delta_{g,\text{F}})$, there is a unique $f$ satisfying~\eqref{e:domaine-}, and we set  $(-\Delta_{g,\text{F}} + \id)u  := f$.
Third, for $\q \in L^\infty((0,T)\times \M; \C)$, the solution to~\eqref{e:Schro-control} is defined via the Duhamel formula for the unitary group $\big(e^{i t\Delta_{g,\text{F}}}\big)_{t\in \R}$ and is a solution of the first equation of~\eqref{e:Schro-control} in the sense of distributions on $(0,T)\times\Int(\M)$.
Note that if we assume that $\M$ is (topologically) complete and that $\d\M$ is compact, then 
$
H^1_0(\M) = \{ u \in H^1(\M) , \Tr(u) = 0\} ,
$
where $H^1(\M)$ is defined as the completion of $C^1(\M)$ functions with finite $H^1$ norm for this norm (definition~\eqref{e:def-H1-norm}) and $\Tr :H^1(\M)\to L^2(\d\M)$ is the trace operator.
This remark justifies the formal writing of the Cauchy problem in~\eqref{e:Schro-control}.

The (second) unique continuation result of Theorem~\ref{theorem global intro} combined with a classical duality argument (see~\cite{DR:77,Lio:88} or~\cite{LL:23notes}) yields the following corollary.
\begin{corollary}
	Assume $\M$ is a complete connected Riemannian manifold with or without compact boundary, $g$ is a locally Lipschitz continuous metric on $\M$, and $\q \in L^\infty((0,T)\times \M; \C) \cap \G^2((0,T);L^\infty_{\loc}(\M;\C))$. For any nonempty open set $\omega\subset \M$, for all $v_0,v_1 \in L^2(\M;\C)$ and for all precision $\eps>0$, there is $f \in L^2((0,T)\times \omega)$ such that the solution to~\eqref{e:Schro-control} satisfies $\|v(T,\cdot)-v_1\|_{L^2(\M)}\leq \eps$.
\end{corollary}
Note that we actually only need to assume $\q \in \G^2(I;L^\infty_{\loc}(\M;\C))$ for some nonempty open set $I\subset (0,T)$.

\subsubsection{Observability, exact controllability}
Unique continuation also plays a key role in proofs of exact controllability results, or equivalently, observability estimates. For wave-type and Sch\"odinger equations, the proof of the latter often decomposes into a high frequency and a low-frequency analysis. We refer to the introduction of~\cite{LL:16} for a detailed account in the case of the wave equation. 
The low-frequency part of the analysis amounts to a unique continuation like Theorem~\ref{theorem global intro}.
The observation system is the following free Schr\"odinger equation:
\begin{equation}
	\label{e:Schro-obs}
	\left\{
	\begin{array}{rl}
		i \d_t u+\Delta_g u+ q u = 0 &  \text{ in } (0,T) \times \Int(\M)  ,\\
		u = 0 ,& \text{ on } (0,T) \times \d\M  \quad \text{ if } \d\M \neq \emptyset , \\
		u(0,\cdot) = u_0 ,&   \text{ in } \Int(\M)  ,
	\end{array}
	\right.
\end{equation}
dual to the control problem~\eqref{e:Schro-control} if $\q=\ovl{q}$.
As in the preceding section, for simplicity of the exposition, we only discuss the internal observability/control of $L^2$ solutions with $\mathsf{b}=0$ to illustrate some applications of our results, and provide with a single geometric example of application.
\begin{theorem}
	\label{t:observability-gevrey}
	Assume that $(\M,g) = (\ID, \Eucl)$ is the Euclidean (closed) unit disk and that $q \in C^\infty([0,T]\times \ID; \R) \cap \G^2((0,T);L^\infty_{\loc}(\Int(\ID);\R))$ is real valued and $\omega$ is any nonempty open set such that $\omega \cap \d\ID\neq \emptyset$. Then for any $T>0$, there is $C>0$ such that for all $u_0 \in L^2(\M)$, the solution $u$ to~\eqref{e:Schro-obs} satisfies
	\begin{align}
		\label{observSchrod}
		\|u_0\|_{L^2(\M)}^2 \leq C \int_0^T \int_\omega |u(t,x)|^2 dx dt .
	\end{align}
\end{theorem}
Our contribution in Theorem~\ref{t:observability-gevrey} is to include more general time-dependent potentials $q$, using Theorem~\ref{theorem global intro} for the ``low frequency'' part of the proof.  Theorem~\ref{t:observability-gevrey} is a direct combination of~\cite[Theorem~1.2]{ALM:16} and Theorem~\ref{theorem global intro}. Note that the $C^\infty$ regularity can be relaxed, see~\cite[Remark~1.6]{ALM:16}.

By a classical compactness-uniqueness argument~\cite{BLR:92}, observability estimates like~\eqref{observSchrod} can be deduced from the unique continuation result of Theorem~\ref{theorem global intro} together with a weakened (or high-frequency) observability estimate (i.e. of the form~\eqref{observSchrod} with an additional relatively compact remainder term on the right-hand side).
The geometry discussed in Theorem~\ref{t:observability-gevrey} is only an example for which the high frequency result may be applied as a black box. One may hope to generalize Theorem~\ref{t:observability-gevrey} to many other geometric situations where the high frequency observability is well-understood, for instance in general geometries under the Geometric Control Condition~\cite{Leb:92}, on tori~\cite{AnantharamanMaciaTore,AnantharamanKMacia,BBZ14}, on negatively curved manifolds~\cite{AnRiv,NAQUE,DJ:17,DJN:20}, in unbounded geometries~\cite{Prouff:23} (see also the references therein).
This requires additional work and we plan to study this question elsewhere.

 \medskip
As a direct corollary of the observability statement of Theorem~\ref{t:observability-gevrey}, we deduce an exact controllability statement for System~\eqref{e:Schro-control}  (see~\cite{DR:77,Lio:88}).
\begin{corollary}
	Assume that the assumptions of Theorem~\ref{t:observability-gevrey} are satisfied. Then, for all $v_0,v_1 \in L^2(\M;\C)$, there is $f \in L^2((0,T)\times \omega)$ such that the solution to~\eqref{e:Schro-control} satisfies $v(T,\cdot) = v_1$.
\end{corollary}

\subsection{Remarks}
\subsubsection{Remarks on Gevrey regularity}
\label{s:section-remarks-gevrey}
One may question the role of the Gevrey $2$ regularity. 
Gevrey regularity already appears in the study of strong unique continuation for elliptic operators, see e.g.~\cite{Lerner:81,CGT:06,IK:12,KNS:19} and the references therein.
In these references, the authors consider elliptic operators with complex coefficients and characterize a critical Gevrey index for strong unique continuation to hold, in relation to the geometry of the image-cone of the principal symbol. 

Gevrey spaces also appeared recently in the related context of control of $1D$ evolution equations in the so-called flatness method. For an operator of the form $\partial_t+ a\partial_x^{\alpha}$, with $a\in \R$ and $\alpha\in \N$ the idea of this method is to solve the ill-posed problem $\partial_x^{\alpha}=-a^{-1}\partial_t u $, seeing $x$ as a new evolution variable. It turns out that the correct regularity in time to be able to solve this evolution problem and the associated control problem is Gevrey $\alpha$, see \cite{MRR:16} for the particular examples of the heat operator, \cite{MRRR:19} for the KdV operator. It corresponds to the index $\alpha=2$ in the case of the Schr\"odinger equation. For an anisotropic operator of the form $P=\partial_t^N+Q$ with $Q$ a differential operator in the space variable of order $M>N$, it is likely that an analog of our result holds assuming that the coefficients of the operator $Q$ are Gevrey $s$ in $t$ with $s=M/N$. 

\medskip
Also, as already alluded, it is proved in~\cite[Th\'eor\`emes~1.4 et~1.6]{LZ:82} that a quasihomogeneous version of pseudoconvexity is actually needed for unique continuation to hold for general $C^\infty$ lower order terms. As an illustration,~\cite[Th\'eor\`eme~1.6]{LZ:82} proves that if $d\geq 2$, there exist $u , \pot \in C^\infty(B_{\R^{1+d}}(0,1);\C)$ such that 
\begin{align*}
	P_{0,\pot} u=0, \quad \textnormal{ in }B_{\R^{1+d}}(0,1) , 
	 \quad \quad u=0 \text{ on } x_1>0, \quad  \text{and} \quad
	0 \in \supp(u) ,   
\end{align*}
whence unique continuation {\em does not hold} across the non-characteristic surface $\{x_1=0\}$. Hence the statements of Theorems~\ref{theorem local uniqueness intro} and~\ref{theorem local L2} are false without the Gevrey-in-time regularity assumption of $\pot$.

As a comparison, in the case of the wave equation, the classical counterexamples of Alinhac-Bahouendi~\cite{AB:79,Alinhac:83,AB:95}, as refined by H\"ormander~\cite{Hormander:00}, prove the following statement. For any $s>1$ and $d\geq 2$, there exist $u,\pot \in \mathcal{G}^s(B_{\R^{1+d}}(0,1);\C)$ so that 
\begin{align*}
	\partial_{t}^{2}u-\Delta u+\pot u=0, \quad \textnormal{ in }B_{\R^{1+d}}(0,1) ,  \quad \text{and} \quad 
	\supp(u)=\left\{(t,x_{1},\dots, x_{d})\in B_{\R^{1+d}}(0,1),  x_{1}\geq 0\right\}.
\end{align*}
This shows that, without any further assumptions, the analyticity in time of $\pot$ is essentially optimal (within the class of Gevrey spaces; note that H\"ormander's statement is even stronger) in geometrical situations where the strong pseudoconvexity of the hypersurface is not satisfied. Theorems~\ref{theorem local uniqueness intro} and~\ref{theorem local L2} show that in the context of the Schr\"odinger equation, Gevrey $1+\eps$ counterexamples do not exist. It would be interesting to know if an equivalent counterexample can be proved for Schr\"odinger type equations with Gevrey $2+\eps$ coefficients, that is to say, whether the Gevrey $2$ regularity in time is the critical one.

\subsubsection{Remarks on the divergence}
\label{s:rem-div}
In the local setting, we have written the elliptic operator in~\eqref{laplacien-perturbe-gevrey} in divergence form. Since we assume that $g^{jk}$ has Lipschitz (time-independent) regularity, and we have
$g^{jk}(x)\d_{x_j} \d_{x_k} = \d_{x_j} g^{jk}(x) \d_{x_k} -\d_{x_j} (g^{jk})(x) \d_{x_k}$, the operator $\d_{x_j} (g^{jk})(x) \d_{x_k}$ has time independent $L^\infty$ coefficients, i.e. the same regularity as $\mathsf{b}^j\d_{x_j}$ in Theorem~\ref{theorem local uniqueness intro}. Hence, the statement of Theorem~\ref{theorem local uniqueness intro} holds as well for $\d_{x_j} g^{jk}(x) \d_{x_k}$ replaced by $g^{jk}(x)\d_{x_j} \d_{x_k}$. That is to say, Theorem~\ref{theorem local uniqueness intro} does not care about the divergence form of the operator. The same remark holds for the first part of Theorem~\ref{theorem global intro}.

\medskip
In Theorem~\ref{theorem local L2} however (and in the second part of Theorem~\ref{theorem global intro}), for the unique continuation statement for $L^2$ solutions, it is important that the elliptic operator is in divergence form.
However, the principal term $\d_{x_j} g^{jk} \d_{x_k}$ or $\Lap$ in these two statements may be replaced by {\em any} operator of the form 
$$\Delta_{g,\varphi} :=\div_\varphi \nabla_g ,$$
 where $g$ is a Lipschitz continuous Riemannian metric, $\varphi$ is a Lipschitz continuous nowhere vanishing density and $\div_\varphi$ and $\nabla_g$ denote respectively the associated divergence (the Riemannian case corresponds to $\varphi = \sqrt{\det(g)}$ with $g=(g_{jk})=(g^{jk})^{-1}$, and the Euclidean case to $\varphi=1$) and gradient. In local coordinates, they write 
$$
\div_{\varphi}(X) = \sum_{j=1}^d \frac{1}{\varphi}\d_{x_j} \left(\varphi X_j\right) ,\qquad \nabla_{g} u = \sum_{j,k=1}^d g^{jk}(\d_{x_j}u)\, \frac{\d}{\d_{x_k}} .
$$
The results of Theorem~\ref{theorem local L2} and the second part of Theorem~\ref{theorem global intro} (for $L^2$ solutions) actually depend on the density chosen (i.e. the result for one density cannot be deduced from that for another density).  They are however valid for any locally Lipschitz nonvanishing density and the proof of Theorems~\ref{theorem local L2} is actually written in the general context of the operator $\Delta_{g,\varphi}$.

\medskip
As far as first order terms are concerned, for the unique continuation statement for $L^2$ solutions, it is crucial that $\sum _{j=1}^d\d_{x_j}\mathsf{b}^j \in  L^\infty(\Omega ;\C)$ in Theorem~\ref{theorem local L2}  (and in the second part of Theorem~\ref{theorem global intro}).
Note that in Theorem~\ref{theorem local L2}, the divergence (form of the operator as well as the divergence condition for $\mathsf{b}$) is taken with respect to the Euclidean density in $\R^{d}$. 
In the global setting of Theorem~\ref{theorem global intro}, the divergence (form of the operator as well as the divergence condition for $\mathsf{b}$) is taken with respect to the Riemannian density in $(\M,g)$.
However, in both settings, given {\em any} nondegenerate locally Lipschitz density $\varphi$, we see that  $$\div_{\varphi}(X) = \div_{1}(X) +  \sum_{j=1}^d \frac{\d_{x_j} \varphi}{\varphi}X_j.$$
Hence, for any $L^\infty$ vector field $\mathsf{b}$, any Lipschitz metric $g$ and any nonvanishing Lipschitz density $\varphi$, we have (locally near a point)
$$
\div_g(\mathsf{b}) \in L^\infty \quad \Longleftrightarrow \quad \div_\varphi (\mathsf{b}) \in L^\infty \quad \Longleftrightarrow\quad \div_1(\mathsf{b}) \in L^\infty ,
$$
where $\div_g$ denotes the Riemannian gradient (and is defined by $\div_{\sqrt{\det(g)}}$).

\subsubsection{More general lower order terms}
\label{r:lot}
So far, all results are stated for linear Schr\"odinger operators. However, as one can check in the proof (see Section~\ref{section adding partially gevrey} where the perturbation argument is performed), 
$\C$--antilinear lower order terms can be included in the unique continuation statements. For instance, the statement of Theorem~\ref{theorem local uniqueness intro} remains valid for all solutions $u$ to 
$$
\pg u + \sum_{j=1}^d \wt{\mathsf{b}}^j (t,x)\d_{x_j}\ovl{u} + \wt{\pot}(t,x) \ovl{u} = 0 ,
$$
assuming (in addition to the assumptions of Theorem~\ref{theorem local uniqueness intro})  that $\wt{\mathsf{b}}^j, \wt{\pot} \in \G^2(I; L^\infty(V;\C))$.

\medskip
One may also want to lower the space regularity of the lower order terms. 
In the proof of Theorem~\ref{theorem local uniqueness intro}, an application of a rough Sobolev embedding shows that only $\pot \in \G^2(I; L^d(V;\C))$ is needed if $d\geq 3$ and $\pot \in \G^2(I; L^{2+\eps}(V;\C))$ for some $\eps>0$ if $d=2$. See Remark~\ref{r:lot-bis} below.
Note also that our result is of no interest in space dimension $d=1$, for unique continuation applies to $L^\infty(I\times V)$ coefficients (without any Gevrey assumption; the appropriate pseudoconvexity condition being satisfied in $1D$), see e.g.~\cite[Corollary 6.1.]{isakov1993carleman}.

\subsection{Idea of the proof, structure of the paper}
\label{s:plan-structure}
Since the pioneering work of~\cite{Carleman:39}, Carleman inequalities are one of the main tools for proving unique continuation results. Carleman estimates are weighted inequalities of the form 
\begin{equation}
\label{e:Carleman-0}
\nor{e^{\tau \phi} Pu }{L^2} \gtrsim  \nor{e^{\tau \phi} u}{L^2}, \quad \tau \geq \tau_0,
\end{equation}
which are uniform in the large parameter $\tau$ and are applied to compactly supported functions $u$. The weight $e^{\tau \phi}$ allows to propagate uniqueness from large to low level sets of $\phi$ by letting $\tau \to \infty$. The key additional idea in~\cite{Tataru:95} (following the introduction in this problem of the FBI transform in time in~\cite{Robbiano:91}) is to make use of the nonlocal Fourier multiplier in time $e^{-\frac{\varepsilon|D_t|^2}{2\tau}}$, and  replace~\eqref{e:Carleman-0} by 
\begin{equation}
\label{e:Carleman-tataru}
\nor{ e^{-\frac{\varepsilon|D_t|^2}{2\tau}} e^{\tau \phi} Pu }{L^2} + e^{-\mathsf{d}\tau}  \nor{e^{\tau \phi} u}{L^2} \gtrsim \nor{ e^{-\frac{\varepsilon|D_t|^2}{2\tau}} e^{\tau \phi} u}{L^2}, \quad \tau \geq \tau_0 .
\end{equation}
A key feature of this approach is that, although~\eqref{e:Carleman-tataru} carries less information on $e^{\tau \phi} u$, it is still enough to prove unique continuation (see Lemma~\ref{lemme d analyse harmonique} below).  And the advantage of~\eqref{e:Carleman-tataru} with respect to~\eqref{e:Carleman-0} is that the operator and the function are localized in a low frequency regime with respect to the variable $t$. Hence~\eqref{e:Carleman-tataru} holds if we only assume the classical pseudoconvexity assumption in a smaller subset of the phase space, namely where $\xi_t=0$ (here, $\xi_t$ is the dual variable to $t$). See~\cite{Tataru:95,RZ:98, Hor:97,Tataru:99} for the original proofs and~\cite{LL:23notes} for introductory lecture notes on this topics in the case of the wave operator.

In the setting of the wave operator $P=-\d_t^2 + \sum g^{jk}(x)\d_{x_j}\d_{x_k}$, the principal symbol $p_2 = \xi_t^2 -  \sum g^{jk}(x)\xi_{x_j}\xi_{x_k}$ is homogeneous of degree two in all co-tangent variables in $(\xi_t,\xi_x)$. When proving Carleman esimates like~\eqref{e:Carleman-0} or~\eqref{e:Carleman-tataru}, the large parameter $\tau$ plays the role of a derivative, which, naturally results in $D_t \sim D_x \sim \tau$.
In this scaling, the Fourier multiplier $\frac{\varepsilon|D_t|^2}{2\tau}$ appearing in~\eqref{e:Carleman-tataru} is ``of order one'', and large frequencies $|D_t| \geq c_0 \tau$ only contribute to admissible remainders of size $e^{-\eps \frac{c_0^2}{2} \tau}$.
 
 \medskip
The first main idea for the proof of Theorems~\ref{theorem local uniqueness intro}--\ref{theorem local L2} is to prove a Carleman estimate adapted to the anisotropy of the  Schr{\"o}dinger operator~\eqref{laplacien-perturbe-gevrey} in case $\mathsf{b}=0,\pot=0$. 
In this setting, we want to consider that $D_t$ is homogeneous to $D^2_x \sim \tau^2$. With this new definition of homogeneity/order/scaling, the natural ``first order'' Fourier multiplier in time is $\frac{|D_t|^2}{\tau^3}$.
Therefore, the first step of the proof of Theorems~\ref{theorem local uniqueness intro}--\ref{theorem local L2} is the Carleman estimate of the form
\begin{equation}
\label{e:Carleman-tataru-nous}
\nor{ e^{-\frac{\mu|D_t|^2}{2\tau^3}} e^{\tau \phi} Pu }{L^2} + e^{-\mathsf{d}\tau}  \nor{e^{\tau \phi} u}{L^2} \gtrsim \nor{ e^{-\frac{\mu|D_t|^2}{2\tau^3}} e^{\tau \phi} u}{L^2}, \quad \tau \geq \tau_0  ,
\end{equation}
 for the unperturbed Schr\"odinger operator $P= i \d_t +\sum g^{jk}(x)\d_{x_j} \d_{x_k}$. This is achieved in Section~\ref{section the carleman estimate} (see Theorem~\ref{th:carlemanschrod}).
 Note that as compared to~\eqref{e:Carleman-tataru}, frequencies $|D_t| \geq c_0 \tau^2$ contribute to admissible remainders of size $e^{-\mu \frac{c_0^2}{2} \tau}$.
 In other words,~\eqref{e:Carleman-tataru-nous} carries information on time-frequencies $|D_t| \lesssim \tau^2$ of the function $e^{\tau \phi} u$ whereas the usual estimate~\eqref{e:Carleman-tataru} only contains information on time-frequencies $|D_t| \lesssim \tau$.
 This is also clearly seen in the proof of~\cite{LL:15} of the optimal quantitative version of the Tataru-H\"{o}rmander-Robbiano-Zuily theorem. In~\cite{LL:15}, the Carleman estimate~\eqref{e:Carleman-tataru-nous} allows to propagate low frequency information of the solution in the sense $|D_t| \lesssim \tau$; whereas the Carleman estimate~\eqref{e:Carleman-tataru-nous} will allow to propagate low frequency information of order $|D_t|\lesssim \tau^2$. This indicates that the new weight allows to ``propagate more information''.
 
The key step in the proof of the Carleman inequality~\eqref{e:Carleman-tataru-nous} (in Theorem~\ref{th:carlemanschrod} below) is a subelliptic estimate (Proposition~\ref{p:subellipti-xit=0} below) for the conjugated operator $P_{\phi,\mu}$ defined by 
\begin{equation}
\label{e:conj-P}
e^{-\frac{\mu|D_t|^2}{2\tau^3}} e^{\tau \phi} P = P_{\phi,\e} e^{-\frac{\mu|D_t|^2}{2\tau^3}} e^{\tau \phi} ,
\end{equation}
where the time independence of the coefficients of $P$ is crucial for the computation of $P_{\phi,\e}$.
The latter takes the form (for appropriate norms)
\begin{align}
\label{e:subell-estim}
 \nor{P_{\phi,\e}v}{L^2}+ \nor{D_t v}{}\gtrsim  \nor{v}{} .
\end{align}
That the subelliptic estimate~\eqref{e:subell-estim}, applied to $v= e^{-\frac{\mu|D_t|^2}{2\tau^3}} e^{\tau \phi} u$, implies the Carleman inequality~\eqref{e:Carleman-tataru-nous} follows from the fact that $e^{-\frac{\mu|D_t|^2}{2\tau^3}}$ localizes exponentially close to $D_t=0$. Hence the term $\nor{D_t v}{}$ mostly contributes to the exponentially small remainder in~\eqref{e:Carleman-tataru-nous} plus a small term that one can absorb in the right hand-side of~\eqref{e:Carleman-tataru-nous} .
The proof of~\eqref{e:subell-estim} relies on two steps. We first perform the computations in the case $\e=0$, that is to say, as for a traditional Carleman estimate of the form~\eqref{e:Carleman-0}, with the difference that all terms involving $\nor{D_t v}{}$ can be considered as remainder terms. This essentially reduces this step to a usual Carleman estimate for elliptic operators with only Lipschitz regularity (plus remainder terms involving time derivatives), for which we rely on~\cite[Appendix A]{LL:18}.
Then the second step consists in considering the general case $\e>0$ as a perturbation of the previous step plus admissible remainder terms.
A related (although different) perturbation argument is used in the proofs of~\cite{Tataru:95,Hor:97,RZ:98,Tataru:99}, see e.g.~\cite[Section~3.3]{LL:23notes}. 
A remarkable difference is that we prove~\eqref{e:Carleman-tataru-nous} for all $\e>0$, whereas~\eqref{e:Carleman-tataru} only holds for {\em small} $\eps>0$. 

\medskip
The second main step for the proof of Theorems~\ref{theorem local uniqueness intro}--\ref{theorem local L2} is to prove that~\eqref{e:Carleman-tataru-nous} still holds for general $\mathsf{b},\pot$ having Gevrey $2$ time-regularity.
To this aim, we perform again a perturbation argument and essentially need to prove that 
\begin{align}
\label{e:estimate-perturb-lot}
\nor{ e^{-\frac{\mu|D_t|^2}{2\tau^3}} (\pot w) }{L^2} \lesssim \nor{ e^{-\frac{\mu|D_t|^2}{2\tau^3}} w }{L^2} + e^{-\mathsf{d}\tau}  \nor{w}{L^2} ,
\end{align}
which becomes an admissible remainder in the sharp version of~\eqref{e:Carleman-tataru-nous} (i.e. with the appropriate norms and powers of the large parameter $\tau$). The proof of~\eqref{e:estimate-perturb-lot} relies on  a conjugation result of the form~\eqref{e:conj-P} but for the multiplication by a function, say $\pot$, depending on $t$. 
The issue of defining a conjugate of a multiplication operator $\pot$ by $e^{-\frac{\eps|D_t|^2}{2\tau}}$ is one of the main difficulties in~\cite{Tataru:95,RZ:98,Hor:97,Tataru:99}. Even if the function $\pot$ is real analytic with respect to $t$ a conjugate operator with respect $e^{-\frac{|D_t|^2}{2\tau}}$ does not necessarily exist. However, one can define an approximate conjugate up to an error of the form $e^{-d \tau }\nor{u}{}$, which is admissible in view of~\eqref{e:estimate-perturb-lot} and~\eqref{e:Carleman-tataru}. 
In the present setting and if typically $\pot \in \mathcal{G}^2(\R;\C)$ depends only on $t$, our conjugation result writes as
\begin{align}
\label{e:conjugation-intro}
e^{-\frac{h}{2}|D_t|^2}  \pot  =  \op^w\big(\tilde{\pot}^r(t+ih\xi_t) \big) e^{-\frac{h}{2}|D_t|^2}  + O\left(e^{-\frac{\delta}{h^{1/3}}}\right)_{\L(L^2(\R))} , \quad h\to 0^+ .
\end{align}
Taken for $h = \frac{\e}{\tau^3}$, owing to the fact that $\op^w\big(\tilde{\pot}^r(t+ih\xi_t) \big)$ is uniformly bounded on $L^2(\R)$, this provides a proof of~\eqref{e:estimate-perturb-lot} and  eventually of ~\eqref{e:Carleman-tataru-nous} for the perturbed  operator $\pg$.
In this expression, the conjugated operator $\op^w\big(\tilde{\pot}^r(t+ih\xi_t) \big)$ is cooked up from $\pot$ as follows: 
\begin{enumerate}
\item First we define $\tilde{\pot}$ an almost analytic extension of $\pot$ to $\C$, well-suited to the $\mathcal{G}^2$ regularity of $\pot$. That is to say a function $\tilde{\pot}\in \mathcal{G}^2(\C;\C)$ such that $\d_{\zb}\tilde{\pot}$ vanishes at any order on the real line.
\item Second we set $\tilde{\pot}^r(t+ih\xi_t) = \eta(h^{2/3}\xi_t) \tilde{\pot}(t+ih\xi_t)$ for $(t,\xi_t) \in \R \times \R$ (with the second variable being the dual variable to $t$, that is to say such that $\op^w(\xi_t)=D_t$), where $\eta$ cuts-off high frequencies, $|D_t |\gtrsim h^{-2/3}\simeq \tau^2$, which, as already mentioned, is the right scale in the present setting.
\end{enumerate}
Our proof of this conjugation result relies on the strategy of Tataru~\cite{Tataru:99}, and proceeds with a deformation of contour. Working with non-analytic functions raises certain non trivial technical difficulties.

\medskip
The plan of this article is as follows. Section~\ref{section the carleman estimate} is devoted to the proof of the Carleman estimate~\eqref{e:Carleman-tataru-nous} in the unperturbed case $\mathsf{b}=0,\pot=0$.
We use some notation from Riemannian geometry which we recall in Section~\ref{s:Riemtool}. We discuss the conjugated operator in this setting in Section~\ref{subsection carleman} and state the  Carleman estimate~\eqref{e:Carleman-tataru-nous}  in Theorem~\ref{th:carlemanschrod}.
We then state the subelliptic estimate~\eqref{e:subell-estim} in Proposition~\ref{p:subellipti-xit=0}, prove that the subelliptic estimate implies the Carleman estimate in Section~\ref{s:sub-to-carl}, and prove the subelliptic estimate in Section~\ref{s:sub}. As already mentioned, this proposition proceeds in two steps: the case $\mu=0$ is first treated in Section~\ref{section proof for epsilon zero} and then the case $\mu>0$ in Section~\ref{e:subell-in} in a perturbation argument. The usual convexification step is performed in Section~\ref{s:convexification}, allowing to transform the function $\Psi$ defining the hypersurface into a weight function $\phi$ satisfying the assumptions of the subelliptic and the Carleman estimate.

Section~\ref{section conjugation with gevrey} is devoted to the study of the conjugated operator and a proof of a conjugation statement like~\eqref{e:conjugation-intro} (namely Proposition~\ref{good conjugagte with function}). In Section~\ref{s:gevrey-fcts} we start with the construction of almost analytic extensions of Gevrey functions adapted to our needs. We then state the conjugation result in Proposition~\ref{good conjugagte with function} and proceed to the proof in Section~\ref{s:conj-operator}. 

The unique continuation Theorems~\ref{theorem local uniqueness intro}--\ref{theorem local L2}  are finally proved in Section~\ref{section the uniqueness theorem}.
Combining the results of Section~\ref{section the carleman estimate} and Section~\ref{section conjugation with gevrey} yields a Carleman estimate with Gevrey lower order terms, studied in Section~\ref{section adding partially gevrey}. Then an appropriate weight function for the unique continuation results is constructed in Section~\ref{section using the carleman estimate} and we conclude the proof of Theorem~\ref{theorem local uniqueness intro}. 
In Section~\ref{section reducing the regularity} we explain how one can exploit the time-regularization of the Fourier multiplier $e^{-\e\frac{|D_t|^2}{\tau^3}}$ combined with the ellipticity of $\pg$ in space, in order to reduce the regularity of the solution in the unique continuation result. This step, actually relying also on a refined estimate proved in Section~\ref{section the carleman estimate} and Section~\ref{section conjugation with gevrey} (where remainder terms involve only $H^{-1}$ regularity of the solution in time), allows to prove Theorem~\ref{theorem local L2}.

The article concludes with Appendix~\ref{a:tools} where we collect several technical estimates and lemmata.

\bigskip
\noindent
{\textbf{\em Acknowledgements.}}
The authors would like to thank Nicolas Burq for having pointed to them the reference~\cite{masuda} and Luc Robbiano for having drawn their attention to the articles~\cite{LZ:82} and~\cite{Dehman:84}. They also thank Didier Smets for helpful comments about elliptic estimates and in particular for Lemma~\ref{lm:ellipticH1H-1}. Most of the work for this project was done when the first author was in the Laboratoire de MathÃ©matiques d'Orsay. He would like to thank the institution for its kind hospitality. The third author is partially supported by the Institut Universitaire de France and the Agence Nationale de la Recherche under grants SALVE (ANR-19-CE40-0004) and ADYCT (ANR-20-CE40-0017).

\section{The Carleman estimate}
\label{section the carleman estimate}

\subsection{Toolbox of Riemannian geometry}
\label{s:Riemtool}
The proof of the Carleman estimate below (as many proofs of Carleman inequalities for operators with low-regularity coefficients) relies on an integration by parts. 
Although we work here in a local setting, it is still convenient to formulate our integration by parts formula in a Riemannian geometric framework following~\cite[Appendix~A]{LL:18}, which we recall now (see~\cite{GallotHulinLaf}).

We  work in a relatively compact open set $V\subset \R^d$.  
We denote by $g = (g_{jk})_{1\leq j,k\leq d}$ a Lipschitz metric on $V$, (that is, $x \mapsto g_x(\cdot , \cdot)$ is a Lipschitz family of symmetric bilinear forms on $TV$ that is uniformly bounded from below, which is equivalent to~\eqref{e:elliptic}). 
We denote by $\gl{\cdot}{\cdot} = g(\cdot , \cdot )$ the inner product in $TV=V \times \R^d$. Remark that this notation omits to mention the point $x \in V$ at which the inner products takes place: this allows to write $\gl{X}{Y}$ as a function on $V$ (the dependence on $x$ is omitted here as well) when $X$ and $Y$ are two vector fields on $V$.
We also denote for a vector field $X$, $\gln{X}=\gl{X}{X}$.
	In $V$, for $f$ a smooth function and $X=\sum_i X^i\frac{\partial }{\partial x_i}$, $Y=\sum_i Y^i\frac{\partial }{\partial x_i}$ smooth vector fields on $V$, we write
	\begin{equation*}
		\begin{array}{ll}
			\dsp \gl{X}{Y}=\sum_{i,j=1}^d g_{ij} X^iY^j , & 
			\nablag f= \sum_{i,j=1}^d g^{ij}(\partial_j f)\frac{\partial }{\partial x_i} , \\
			\dsp  \div_g(X) =\sum_{i=1}^d \frac{1}{\sqrt{\det g}}\partial_i \left(\sqrt{\det g} X_i\right) , & 
			\Lap f = \div_g  \nablag f= \sum_{i,j=1}^d \frac{1}{\sqrt{\det g}}\partial_i \left(\sqrt{\det g}g^{ij}\partial_j f\right) , \\ 
			\dsp D_{X}Y=\sum_{i=1}^d\left(\sum_{j=1}^dX^j\frac{\partial Y^i}{\partial x_j}+\sum_{j,k=1}^d\Gamma_{j,k}^iX^jY^k\right)\frac{\partial }{\partial x_i} ,
		\end{array}
	\end{equation*}
	where $(g^{-1})_{ij}=g^{ij}$ and the Chritoffel symbols are defined by $$\Gamma_{j,k}^i=\frac{1}{2}\sum_{l=1}^d g^{il}\left(\partial_j g_{kl}+\partial_k g_{lj}-\partial_lg_{jk}\right) ,$$ (see for instance \cite[p71]{GallotHulinLaf}).
	Note in particular that the Lipschitz regularity of $g$ writes $g_{ij} \in W^{1,\infty} (V)$, and implies $g^{ij} \in W^{1,\infty} (V)$.
	This entails, if $f, X, Y$ are smooth, that $\gl{X}{Y} \in W^{1,\infty} (V)$, $\nablag f$ is a Lipschitz vector field, $ \Lap f \in L^\infty(V)$ and $D_X Y$ is an $L^\infty$ vector field on $V$, since the definitions of $\Lap$ and $D_X$ involve one derivative of the coefficients of $g$.
 Note that we have chosen to use the Riemannian density $\varphi = \sqrt{\det g}$ in the definition of the divergence for simplicity. Any non-vanishing Lipschitz density $\varphi$ would do the same. The results for one density may anyways be deduced from those with another density, see the discussion in Section~\ref{s:rem-div} as well as Remark~\ref{r:carleman-density} below.
	Let us now collect some properties of these objects, that we shall use below.
	For $f,g$ two smooth functions on $V$ and $X=\sum_i X^i\frac{\partial }{\partial x_i}$, $Y=\sum_i Y^i\frac{\partial }{\partial x_i}$ two smooth vector fields on $V$, we have
	\begin{align*}
		\div_g (f X)&= \gl {\nablag f}{X}+f \div_g(X), \\
		D_X (fY) &= (Xf)Y+ f D_X Y, \\  
		D_X (\gl{Y}{Z})&= \gl {D_X Y}{Z}+ \gl {Y}{D_X Z}.
	\end{align*}
	We define (see~\cite[Exercice~2.65]{GallotHulinLaf} or~\cite{LL:18} for more on the the Hessian) 
	\begin{align*}
		\Hess(f)(X,Y)=(D_X df)(Y) = \sum_{i,j}X^iY^j\left[ \partial_{ij}^2f-\Gamma_{ij}^k\partial_k f \right] ,
	\end{align*}
	which again is in $L^\infty(V)$ for a Lipschitz metric $g$ and $L^\infty$ vector fields $X,Y$. Note also that the Hessian of $f$ is  symmetric, that is $\Hess(f)(X,Y)=\Hess(f)(Y,X)$ and for any function $f$ and any vector field $X$ and $Y$, we have (see e.g.~\cite[Lemma~A.1]{LL:18})
	$\Hess(f)(X,Y) =\gl{ D_{X}\nablag f}{ Y}$.
	Concerning integrals, we write in this section $$\int f  = \int_V f(x) \sqrt{\det g(x)} dx ,$$ 
	where $\sqrt{\det g(x)} dx$ is the Riemannian density. With this notation, a useful integration by parts formula writes as follows:
	For all $f\in H^2(V)$ and $h \in H^1(V)$ one of which having compact support in $V$, we have 
	$$
	\int (\Lap f) h  = - \int \gl{ \nablag f}{\nablag h} .
	$$
	As we are interested in complex-valued functions, we set
	$(f,g)=(f,g)_{L^2(V)}=\int f\overline{h}$ for the $L^{2}$ hermitian product. 
	We are moreover interested in time-dependent functions, and in the context of spacetime integration, we write $\iint f  = \int_{\R_{t}}\int_{V} f(t,x) \sqrt{\det g(x)} dx dt$ and similarly $(f,g) =\iint f\overline{h}$.
	
	\subsection{The Carleman weight}
	\label{subsection carleman}
	We denote by $\Omega=I \times V$ where $I$ is a bounded open  interval of $\R$ and $V$ is a relatively compact open subset of $\R^d$ equipped with a Lipschitz metric $g$. In this section, we set $P:=i \d_t +\Delta_g$ where $\Delta_g$ is defined in Section~\ref{s:Riemtool}.

	For a smooth real-valued weight function $\phi$ (later on, we will assume that it is polynomial of order $2$), the Carleman estimate below will make use of the operator, as explained in Section~\ref{s:plan-structure}.
	\bna
	Q_{\e,\tau}^\phi u: =e^{-\e\frac{|D_t|^2}{2\tau^3}}e^{\tau \phi} u.
	\ena
	
	In all the rest of the proof, $\mu$ does not have any role and could be any constant. We have chosen to keep it along the proof since we believe it helps to follow the perturbation of the pseudodifferential weight.
	We now describe the conjugation by $e^{-\e\frac{|D_t|^2}{2\tau^3}}$.
	\begin{lemma}[Lemma 3.12 in~\cite{LL:23notes}]
		\label{lmcommuteps}
		Let $u\in\mathcal{S}(\R^{1+d})$ and $\varsigma>0$, then
		\bna
		e^{-\frac{|D_t|^2}{2\varsigma}}(tu)=\left(t+i\frac{D_t}{\varsigma} \right)e^{-\frac{|D_t|^2}{2\varsigma}}u.
		\ena
	\end{lemma}
	This implies the following conjugation of monomials.
	\begin{lemma}[Lemma 3.14 in~\cite{LL:23notes}]
		\label{l:conjugation-D}
		Assume $\phi$ is a real polynomial of degree two in the variable $t$. For all $k \in \{0,\cdots ,d\}$ (with the convention $t=\x_0$, $D_0=D_t$) we have
		$$
		Q_{\e,\tau}^{\phi} D_k =(D_k)_{\phi,\e}Q_{\e,\tau}^{\phi} ,
		$$
		where (denoting $\phi''_{t,\x_k} = \d_t \d_{\x_k}\phi$) 
		$$
		(D_k)_{\phi,\e}= D_k +i\tau \d_{\x_k}\phi(\x)-\e \phi''_{t,\x_k}\frac{D_t}{\tau^2}.
		$$
	\end{lemma}
	The goal of Section~\ref{section the carleman estimate} is to prove a Carleman estimate for the ``unperturbed'' operator 
	 \begin{equation}
	 \label{e:def-P-0}
	 P=i \partial_t+\Delta_g=-D_t-\sum_{j,k=1}^d \frac{1}{\sqrt{\det g}}D_j \sqrt{\det g}g^{jk}D_k , 
	 \end{equation} with {\em all coefficients independent of $t$}.
 The following corollary is a direct consequence of Lemma~\ref{l:conjugation-D}.
	\begin{corollary}[The ``conjugated operator'']
		\label{lmsympphiwave}
		Let $\phi$ be a real-valued function being {\em quadratic} in $t$ and $P$ defined in~\eqref{e:def-P-0}. 
		Then, for any $\e>0$,  
		\begin{align*}
			Q_{\e,\tau}^{\phi}P & =P_{\phi,\e}Q_{\e,\tau}^{\phi} , \quad \text{ with } \\ 
			P_{\phi,\e} =&-\left( D_t +i\tau \d_t\phi(\x)-\e \phi''_{t,t}\frac{D_t}{\tau^2}\right)\\&-\sum_{j,k=1}^d \frac{1}{\sqrt{\det g}}\left( D_{j} +i\tau \d_{j}\phi(\x)-\e \phi''_{t,j}\frac{D_t}{\tau^2}\right) \sqrt{\det g}g^{jk}\left( D_{k} +i\tau \d_{k}\phi(\x)-\e \phi''_{t,k}\frac{D_t}{\tau^2}\right) .
		\end{align*}
	\end{corollary}
	We define the anisotropic norm
	\begin{equation}
		\label{def of anisotropic norm}
		\nor{v}{\Ht{1}}^2 := \tau^2 \nor{v}{L^2}^2 +\nor{D_x v}{L^2}^2 +\tau^{-2} \nor{D_t v}{L^2}^2, 
	\end{equation}
	adapted to the homogeneity of the operator $P$ in~\eqref{e:def-P-0} (see the discussion in Section~\ref{s:plan-structure})
	and its spatial part
	\begin{equation}
		\label{def of space norm}
		\nor{v}{H^{1}_{\tau, x}}^2 := \tau^2 \nor{v}{L^2}^2 +\nor{D_x v}{L^2}^2 .    
	\end{equation}
	
	Before stating our main Carleman estimate we need to define the following two important quantities, see~\cite[Theorem~A.5]{LL:18}. Given $\phi \in W^{2, \infty}(\Omega; \R), f\in W^{1,\infty}(\Omega; \R)$, $X$ a smooth {\em complex valued} vector field on $V$ we set
	\begin{align}
		\B_{g, \phi, f}(X)&  := 2 \Hess(\phi)(X,\ovl{X}) - (\Lap \phi )\gln{X}+f \gln{X} , \label{e:def-B}\\ 
		\E_{g, \phi, f} &:= 2 \Hess(\phi)(\nablag \phi,\nablag \phi) + (\Lap  \phi )\gln{\nablag \phi}-f\gln{\nablag \phi},\label{e:def-E}
	\end{align}
	where the Hessian is with respect to the $x$ variable only, see Section~\ref{s:Riemtool}, and where we have written $\gln{X}=\gl{X}{\ovl{X}}$.
	Note that these are two real quantities (since $\Hess(\phi)$ is a real symmetric bilinear form).
	Note that the only difference with~\cite[Theorem~A.5]{LL:18} is that the vectorfield $X$ was assumed real-valued (in applications, $X=\nablag u$).
	Note that for a Lipschitz metric $g$ on $V$, we have $\E_{g, \phi, f} \in L^\infty(\Omega; \R)$ and $\B_{g, \phi, f}(X) \in  L^\infty(\Omega; \R)$ for any bounded vector field $X$ on $V$ and we stress the fact that these two quantities are time-dependent (they are defined on $\Omega=I\times V$).

	\begin{remark}
		In what follows we use the notation $C$ for a constant  whose value may change from one line to another. It may depend on the norms $\nor{\phi}{  W^{2,\infty}}$ and $\nor{f}{W^{1,\infty}}$ where $f \in W^{1,\infty}$ is an auxiliary function, and on the metric $g$ only {\em via} the quantities $\nor{g^{jk}}{W^{1,\infty}(V)}$ and the ellipticity constant $c_0$ of the metric $g^{jk}$ (only Lipschitz regularity of $g$ is assumed). 
	\end{remark}
	
	Let us now state the main result of this section, which is a Carleman estimate in the spirit of~\cite{Tataru:95,Hor:97,RZ:98,Tataru:99} but with two main differences:
	\begin{enumerate}
		\item The Fourier multiplier is now $e^{-\frac{\e |D_t|^2}{2 \tau^3}}$ instead of $e^{-\frac{\e |D_t|^2}{2 \tau}}$ 
		
		\item We use the anisotropic norm defined in~\eqref{def of anisotropic norm}.
	\end{enumerate}
	In Section~\ref{section adding partially gevrey} we show that this estimate remains valid for lower order perturbations of the operator $P$ in~\eqref{e:def-P-0}.

	\begin{theorem}[Carleman estimate]
		\label{th:carlemanschrod}
		Let $\xg_0=(t_0,x_0) \in\Omega =I\times V\subset \R^{1+d}$. Assume that $\phi$ and $f$ satisfy the following: $\phi$ is a {\em quadratic real-valued polynomial}, $f\in W^{1,\infty}(\Omega; \R)$, there exist $r>0$ such that $\gln{\nablag \phi} >0$ on $\overline{B}(\xg_0,r)$, and $C_0>0$ such that for any vector field $X$, we have almost everywhere on $B(\xg_0,r)$:
		\begin{align}
			\label{e:sub-ellipticity-EB}
			\B_{g, \phi, f}(X) \geq C_0 \gln{X},  \quad \text{ and } \quad 
			\E_{g, \phi, f}  \geq C_0  \gln{\nablag \phi} .
		\end{align}
		Then, for all $\mu >0$ and $k \in \N$ there exist $ \mathsf{d}, C, \tau_0>0$ such that for all $\tau \geq \tau_0$ and $w \in C^\infty_c(B(\xg_0, \frac{r}{8}))$, for $P$ defined in~\eqref{e:def-P-0}, we have 
		\bnan
		\label{Carlemanschrod}
		C\nor{Q_{\e,\tau}^{\phi}P w}{L^2}^2+ Ce^{-\mathsf{d}\tau}\nor{e^{\tau\phi}w}{H^{-k}_tH^1_x}^2 \geq  \tau \|Q_{\e,\tau}^{\phi}w\|_{\Ht{1}}^2 .
		\enan
	\end{theorem}
	In~\eqref{Carlemanschrod}, $H^{-k}_tH^1_x= H^{-k}(\R;H^1(V))$, that is to say
	\begin{align}
		\label{e:norm-esp-tps}
		\nor{v}{H^{-k}_tH^1_x}= \nor{\langle D_t\rangle^{-k} v}{L^2(\R;H^1(V))} .
	\end{align}
	Theorem~\ref{th:carlemanschrod} states a precise version of~\eqref{e:Carleman-tataru-nous}.
		\begin{remark}[Lower order perturbations]
		\label{r:carleman-density}
	Note that in Theorem~\ref{th:carlemanschrod} we have stated the result for the operator $P$ defined in~\eqref{e:def-P-0}. As usual for Carleman estimates, the statement still holds for $P$ replaced by any lower order time-independent perturbation with $L^\infty(V)$ coefficients (using that the latter commutes with $Q_{\e,\tau}^{\phi}$ and the corresponding additional term in~\eqref{Carlemanschrod} can thus be absorbed in the right-hand side for $\tau$ sufficiently large). According to the discussion of Section~\ref{s:rem-div}, this proves that $P$ can be equivalently replaced by $i \partial_t+\Delta_{g,\varphi}$ for any Lipschitz nonvanishing density $\varphi$ in Theorem~\ref{th:carlemanschrod}.
	\end{remark}
	\begin{remark}
	\label{r:rem-k}
	The $H^{-k}_tH^1_x$ norm on the error term in the left-hand side of~\eqref{Carlemanschrod} is obtained as a consequence of the regularization properties of the operator $e^{-\frac{\e |D_t|^2}{2 \tau^3}}$.
	The unique continuation result of Theorem~\eqref{theorem local uniqueness intro} concerning $L^2(I;H^1(V))$ solutions only uses the case $k=0$ (for which the proof of Theorem~\ref{th:carlemanschrod} is simpler).
	The unique continuation result of Theorem~\ref{theorem local L2} concerning $L^2(I\times V)$ relies on the case $k=1$, combined with an ellipticity argument (to gain derivative in space). See  Section~\ref{section reducing the regularity} below.
	Finally, the unique continuation statement of Remark~\ref{r:distributions} concerning distribution solutions uses the full range of $k \in \N$ (together with an ellipticity argument). 
	\end{remark}

	The main step for the proof of Theorem~\ref{th:carlemanschrod} is the following subelliptic estimate.

	\begin{proposition}[Subelliptic estimate]
		\label{p:subellipti-xit=0}
		Let $\xg_0=(t_0,x_0) \in\Omega =I\times V\subset \R^{1+d}$. Assume that $\phi$ and $f$ satisfy the assumptions of Theorem~\ref{th:carlemanschrod}.
		Then, for all $\mu>0$ there exist $C, \tau_0>0$ such that for all $\tau \geq \tau_0$ and $v \in C^\infty_c(B(\xg_0,r))$, we have 
		\bnan
		\label{Carlpropwaveinterm}
		C \nor{P_{\phi,\e}v}{L^2}^2+C\tau^{-1} \nor{D_t v}{L^2}^2\geq \tau \nor{v}{\Ht{1}}^2 .
		\enan
	\end{proposition}
	
	\begin{remark}[Perturbations of~\eqref{Carlpropwaveinterm} by lower order terms]
		\label{carleman insensitive lot}
		In the setting of Proposition~\ref{p:subellipti-xit=0}, we consider
		\begin{align}
			\label{e:admiss-remainder}
			R= A\cdot D_x+\tau a+\frac{b}{\tau^2}D_t+\frac{c}{\tau}D_t, \quad \text{ with }\quad a,b,c \in L^\infty(\Omega;\C), A \in L^\infty(\Omega;\C^d) .
		\end{align}
		Recalling~\eqref{def of anisotropic norm}, we have for $\tau \geq1$,
		$$
		\nor{Rv}{L^2} \lesssim \tau \nor{v}{L^2}+\nor{D_{x}u }{L^2}+\frac{\nor{D_t v}{L^2}}{\tau^2}+\frac{\nor{D_t v}{L^2}}{\tau} \lesssim \nor{v}{\Ht{1}} .
		$$
		As a consequence, estimate~\eqref{Carlpropwaveinterm} holds for $P_{\phi,\e}$ if and only if it holds for $P_{\phi,\e}+R$ in place of $P_{\phi,\e}$, up to changing the values of $\tau_0$ and $C$.
	Let us now define
	\begin{equation}
		\label{le fameux spyrossien}
		\mathsf{P}:=\sum_{j,k} g^{jk}(x)\d_j\d_k=-\sum_{j,k} g^{jk}(x)D_j D_k.
	\end{equation}
	As in Corollary~\ref{lmsympphiwave}, we have $Q_{\e,\tau}^{\phi}\mathsf{P} = \mathsf{P}_{\phi,\e}Q_{\e,\tau}^{\phi}$ with 
	\begin{align}
		\label{e:def-PPhi-mu-2}
		\mathsf{P}_{\phi,\e} =&-\left( D_t +i\tau \d_t\phi(\x)-\e \phi''_{t,t}\frac{D_t}{\tau^2}\right) \nonumber\\&-\sum_{j,k=1}^d g^{jk}(x) \left( D_{j} +i\tau \d_{j}\phi(\x)-\e \phi''_{t,j}\frac{D_t}{\tau^2}\right)\left( D_{k} +i\tau \d_{k}\phi(\x)-\e \phi''_{t,k}\frac{D_t}{\tau^2}\right).
	\end{align}
	Remark now that since the metric $g$ is Lipschitz and time independent, the commutator
	$$
	\left[\left( D_{j} +i\tau \d_{j}\phi-\e \phi''_{t,j}\frac{D_t}{\tau^2}\right), \sqrt{\det g}g^{jk}\right]=\left[ D_{j}, \sqrt{\det g}g^{jk}\right],
	$$
	is a differential operator of order zero, with $L^\infty$ coefficients. It follows that 
	$$
	P_{\phi,\e}  = \mathsf{P}_{\phi,\e} -R, \quad \text{ with }\quad  R = \sum_{j,k=1}^d \frac{1}{\sqrt{\det g}}\left[ D_{j}, \sqrt{\det g}g^{jk}\right]\left( D_{k} +i\tau \d_{k}\phi-\e \phi''_{t,k}\frac{D_t}{\tau^2}\right) ,
	$$
	and, according to the above discussion, estimate~\eqref{Carlpropwaveinterm} for $\mathsf{P}_{\phi,\e}$ implies the same estimate for $P_{\phi,\e}$ (and {\em vice versa}). 
\end{remark}
Remark~\ref{carleman insensitive lot} allows to transfer estimates from $P_{\phi,\e}$ to $\mathsf{P}_{\phi,\e}$ and vice versa.
In Section~\ref{s:sub-to-carl}, we first show how the subelliptic estimate of Proposition~\ref{p:subellipti-xit=0} implies the Carleman estimate of Theorem~\ref{th:carlemanschrod}. Then in Section~\ref{s:sub} we prove the subelliptic estimate of Proposition~\ref{p:subellipti-xit=0}.

\subsection{From the subelliptic estimate to the Carleman estimate}
\label{s:sub-to-carl}
\begin{proof}[Proof of Theorem~\ref{th:carlemanschrod} from Proposition~\ref{p:subellipti-xit=0}]
	Suppose for simplicity that $t_0=0$ and let $r_0:=r/2$ with $r$ given by the assumptions of Theorem~\ref{th:carlemanschrod} and Proposition~\ref{p:subellipti-xit=0}. 
	Consider $w \in C^\infty_c(B(\xg_0,r_0/4);[0,1])$ and $\chi \in C^\infty_c((-r_0,r_0);[0,1])$ with $\chi=1$ on $(-r_0/2,r_0/2)$. We notice that
	\bnan
	\label{first basic for carleman }
	\tau \nor{Q_{\e,\tau }^{\phi}w}{\Ht{1}}^2\leq 2 \tau \nor{\chi Q_{\e,\tau }^{\phi}w}{\Ht{1}}^2+ 2 \tau \nor{(1-\chi)Q_{\e,\tau }^{\phi}w}{\Ht{1}}^2.
	\enan
	Consider $\chit \in C^\infty_c((-r_0/3,r_0/3);[0,1])$ with $\chit=1$ in a neighborhood of $[-r_0/4,r_0/4]$, so that $w=\chit w$. 
	Recalling the norms~\eqref{def of anisotropic norm}--\eqref{def of space norm}, the support properties of $\chi, \chit$ and $w$ together with Lemma~\ref{lemma 2.4 from ll} we estimate the second term in \eqref{first basic for carleman } as  
	\begin{align}
		\label{estim for 1-chi}
		\tau \nor{(1-\chi)Q_{\e,\tau }^{\phi}w}{\Ht{1}}^2&\leq C\tau \nor{(1-\chi)Q_{\e,\tau }^{\phi}w}{H^1_{\tau,x}}^2+C\tau^{-1}\nor{D_t(1-\chi)Q_{\e,\tau }^{\phi}w}{L^2}^2 \nonumber  \\
		&= C \tau\nor{(1-\chi)e^{-\e\frac{|D_t|^2}{2\tau^3}}e^{\tau \phi}\chit w}{H^1_{\tau,x}}^2+C\tau^{-1}\nor{D_t(1-\chi)Q_{\e,\tau }^{\phi}w}{L^2}^2\nonumber \\
		&\leq C \tau e^{-2c\frac{\tau^3}{\e}} \nor{e^{\tau \phi} w}{H^{-k}_t H^1_x}^2+C\tau^{-1}\nor{[D_t,(1-\chi)]e^{-\e\frac{|D_t|^2}{2\tau^3}}e^{\tau \phi} \chit w}{L^2}^2 \nonumber \\ &
		\hspace{4mm}+C\tau^{-1}\nor{(1-\chi)D_t Q_{\e,\tau }^{\phi}w}{L^2}^2\nonumber  \\
		&\leq C_\mu e^{-c\frac{\tau^3}{\e}} \nor{e^{\tau \phi} w}{H^{-k}_t H^1_x}^2+C\tau^{-1}\nor{D_t Q_{\e,\tau }^{\phi}w}{L^2}^2,
	\end{align}
	where $H^{-k}_t H^1_x=H^{-k}(\R; H^1(V))$, see~\eqref{e:norm-esp-tps}.
	We estimate now the second term in \eqref{estim for 1-chi}. To do so, we consider $\sigma>0$ a small constant to be chosen later on and we distinguish between frequencies smaller or larger than $\sigma \tau^2$. We also assume $\sigma \tau^{2}\geq 1$ and obtain
	\begin{align*}
		\nor{D_t Q_{\e,\tau }^{\phi}w}{L^2}&\leq \nor{D_t \mathds{1}_{|D_t|\leq \sigma \tau^2} Q_{\e,\tau }^{\phi} w}{L^2}+\nor{D_t^{k+1} \mathds{1}_{|D_t|\geq \sigma \tau^2}\langle D_t\rangle^{-k}e^{-\e\frac{|D_t|^2}{2\tau^3}}e^{\tau \phi} w}{L^2} \\
		&\leq  \sigma \tau^2  \nor{ Q_{\e,\tau }^{\phi} w}{L^2}+\underset{\xi_t\geq \sigma \tau^2}{\max}(\xi_t^{k+1} e^{-\e\frac{|\xi_t|^2}{2\tau^3}})\nor{e^{\tau \phi}w}{H^{-k}_t L^2_x}.
	\end{align*}
	Now the function $\R^+  \ni  s \mapsto s^\mathsf{k} e^{-\e\frac{|s|^2}{2\tau^3}}$ reaches its maximum at $s=\sqrt{\frac{\mathsf{k} \tau^3}{\e}}$ and is decreasing on $[\sqrt{\frac{\mathsf{k} \tau^3}{\e}}, \infty)$. As a consequence, if $\sigma \tau^2 \geq \sqrt{\frac{\mathsf{k} \tau^3}{\e}} 
	$ which translates to $\tau \geq \frac{\mathsf{k}}{\sigma^2 \e}$, one has $\underset{\xi_t \geq \sigma \tau^2}{\max}(\xi_t^\mathsf{k} e^{-\e\frac{|\xi_t|^2}{2\tau^3}})=\sigma^\mathsf{k} \tau^{2\mathsf{k}} e^{- \e \frac{\sigma^2 \tau^4}{2\tau^3}}=\sigma^\mathsf{k} \tau^{2\mathsf{k}} e^{- \e \frac{\sigma^2 \tau}{2}}$.
	We obtain therefore, for $\tau \geq \tau_0 \geq \max \left( 1, \sigma^{-1/2} , \frac{k+1}{\sigma^2\e} \right)$,
	\bnan
	\label{estimate for Dtgaussienne}
	\nor{D_t Q_{\e,\tau }^{\phi}w}{L^2} \leq \sigma \tau^2 \nor{ Q_{\e,\tau }^{\phi} w}{L^2}+\sigma^{k+1} \tau^{2k+2} e^{- \e \frac{\sigma^2 \tau}{2}}\nor{e^{\tau \phi}w}{H^{-k}_t L^2_x}.
	\enan
	We now estimate the term $\tau \nor{\chi Q_{\e,\tau }^{\phi}w}{\Ht{1}}^2$ appearing in \eqref{first basic for carleman }. Thanks to the support properties of $\chi$ and $w$ we can apply the subelliptic estimate of Proposition~\ref{p:subellipti-xit=0} to $v:=\chi Q_{\e,\tau }^{\phi}w \in C^\infty_c([-r/2,r/2]\times B(0,r/8))$. We obtain
	\begin{align}
		\label{subellipic applied to chi v }
		\tau \nor{\chi Q_{\e,\tau }^{\phi}w}{\Ht{1}}^2  &\leq C \nor{P_{\phi, \e} \chi Q_{\e,\tau }^{\phi}w}{L^2}^2+C\tau^{-1} \nor{D_t \chi Q_{\e,\tau }^{\phi}w }{L^2}^2 \nonumber \\
		&\leq 
		C \nor{P_{\phi, \e} \chi Q_{\e,\tau }^{\phi}w}{L^2}^2+C\tau^{-1} \nor{Q_{\e,\tau }^{\phi}w }{L^2}^2+C\tau^{-1} \nor{D_t Q_{\e,\tau }^{\phi}w }{L^2}^2 \nonumber \\
		&\leq C \nor{P_{\phi, \e} \chi Q_{\e,\tau }^{\phi}w}{L^2}^2+C \sigma^2 \tau^3 \nor{ Q_{\e,\tau }^{\phi} w}{L^2}^2+C\sigma^{2k+2} \tau^{4k+3} e^{-\e \sigma^2 \tau}\nor{e^{\tau \phi}w}{H^{-k}_t L^2_x}^2,
	\end{align}
	where for the last inequality, we used $\sigma \tau^{2}\geq 1$ and \eqref{estimate for Dtgaussienne}. Recalling Corollary~\ref{lmsympphiwave}, $Q_{\e,\tau}^{\phi}P =P_{\phi,\e}Q_{\e,\tau}^{\phi}$ and thus
	\begin{align*}
		\nor{P_{\phi, \e} \chi Q_{\e,\tau }^{\phi}w}{L^2} &\leq \nor{\chi P_{\phi, \e}  Q_{\e,\tau }^{\phi}w}{L^2}+\nor{[P_{\phi, \e},\chi] Q_{\e,\tau }^{\phi}w}{L^2}    \\ 
		&\leq \nor{Q_{\e,\tau }^{\phi}P w}{L^2}+\nor{[P_{\phi, \e},\chi] e^{-\e\frac{|D_t|^2}{2\tau^3}}e^{\tau \phi}\chit w}{L^2} .
	\end{align*}
	Recalling that $\chi=\chi(t)$, together with the expression of $P_{\phi, \e}$ in Corollary~\ref{lmsympphiwave}, we have 
	$$
	[P_{\phi, \e},\chi]=i\chi'+R , \quad \text{ with } \quad R = \frac{1}{\tau^2}\left( F(\x)\cdot D_x + f_0(\x) \frac{D_t}{\tau^2} + f_1(\x)\tau +f_2(\x) + \frac{1}{\tau^2}f_3(\x) \right),
	$$ where $F,f_0,f_1,f_2,f_3 \in L^\infty(I\times V)$ satisfy $ \supp(F,f_0,f_1,f_2,f_3) \subset \supp(\chi^\prime) \times V$. 
	Given the support properties of $\chi, \chit$, Lemma~\ref{lemma 2.4 from ll} yields for all $k \in \N$ the existence of $C,c>0$ such that 
	$$
	\nor{[P_{\phi, \e},\chi] e^{-\e\frac{|D_t|^2}{2\tau^3}}e^{\tau \phi}\chit w}{L^2} 
	\leq 
	\nor{\chi'  e^{-\e\frac{|D_t|^2}{2\tau^3}}\chit e^{\tau \phi} w}{L^2} +\nor{R e^{-\e\frac{|D_t|^2}{2\tau^3}}\chit e^{\tau \phi} w}{L^2} 
	\leq C_\mu e^{-c\frac{\tau^3}{\e}} \nor{e^{\tau \phi} w}{H^{-k}_t H^1_x}.
	$$
	Putting the two last inequalities together we obtain
	\begin{align}
		\label{estimate for term of subelliptic with comm}
		\nor{P_{\phi, \e} \chi Q_{\e,\tau }^{\phi}w}{L^2} &\leq  C\nor{Q_{\e,\tau }^{\phi}P w}{L^2}+C e^{-c\frac{\tau^3}{\e}} \nor{e^{\tau \phi} w}{H^{-k}_t H^1_x}.
	\end{align}
	Combining \eqref{first basic for carleman }, \eqref{estim for 1-chi}, \eqref{estimate for Dtgaussienne}, \eqref{subellipic applied to chi v } and \eqref{estimate for term of subelliptic with comm} we find that for any $\mu>0,k \in \N$, there are constants $C,c,\tau_0>0$ such that for any $\sigma>0$ and $\tau \geq \tau_0$ we have
	
	\begin{align*}
		\tau \|Q_{\e,\tau}^{\phi}w\|_{\Ht{1}}^2 \leq C\nor{Q_{\e,\tau }^{\phi}P w}{L^2}^2+ C \sigma^2 \tau^3\nor{Q^\phi_{\e,\tau}w}{L^2}^2+ C\left(  e^{-c\frac{\tau^3}{\e}} + \sigma^{2k+2}\tau^{4k+3}e^{-\e \sigma^2 \tau} \right) \nor{e^{\tau \phi}w}{H^{-k}_t H^1_x}^2.
	\end{align*}
	Choosing then $\sigma>0$ sufficiently small allows to absorb the term $\sigma^2 \tau^3\nor{Q^\phi_{\e,\tau}w}{L^2}^2$ in the left-hand side. Then taking $\tau\geq\tau_0$ with $\tau_0$ sufficiently large finishes the proof of Theorem~\ref{Carlemanschrod} from Proposition~\ref{p:subellipti-xit=0}.
\end{proof}

\subsection{Proof of the subelliptic estimate}
\label{s:sub}
This section is devoted to the proof of Proposition~\ref{p:subellipti-xit=0}.
Recall that the operator $P$ is defined in~\eqref{e:def-P-0} and let us consider the "classical" conjugated operator given by
$$
P_\phi:=e^{\tau \phi}P e^{- \tau \phi} = e^{\tau \phi}(i\partial_{t}+\Lap)e^{- \tau \phi},
$$
where we recall that $\Delta_g$ is defined in Section~\ref{s:Riemtool}.
Remark that $P_\phi =P_{\phi,0}$ where $P_{\phi,\mu}$ is defined in Corollary~\ref{lmsympphiwave}. We start by proving in Section~\ref{section proof for epsilon zero} the desired subelliptic estimate in the particular case $P_\phi =P_{\phi,0}$. We then prove in Section~\ref{e:subell-in} that the additional terms coming from the difference $\nor{(P_{\phi, \e}-P_\phi) u}{L^2}^2$ can be absorbed in the estimate.

\subsubsection{Case $\e=0$}
\label{section proof for epsilon zero}

We recall the definitions of $\B_{g, \phi, f}(X)$ and $\E_{g, \phi, f}$ in~\eqref{e:def-B} and~\eqref{e:def-E} respectively. We sometimes write $v_t:=\d_t v$.
\begin{proposition}
	\label{thmCarlemancalcul-epsilonzero}
	Let $\Omega \subset \R^{1+d}$. Assume that $\phi \in W^{2, \infty} (\Omega; \R)$ and $f \in W^{1,\infty}(\Omega; \R)$. Then, there exists $C>0$ such that for any $u\in C^\infty_c(\Omega)$ and $\tau\geq 0$, we have for any $\delta>0$ 
	\begin{align*}
		&3 \nor{P_\phi u}{L^2}^2 
		+\left(\nor{\Lap \phi}{L^{\infty}}^{2}+\nor{f}{L^{\infty}}\right)\frac{1}{\delta\tau} \nor{u_{t}}{L^{2}}^{2}
		+ R(u) \geq 
		2 \tau^3\iint \left[\E_{g, \phi, f} -\delta\right]|u|^2 + 2  \tau \iint \B_{g, \phi, f}(\nablag u),
	\end{align*}
	\begin{align}
		\label{e:def-R(u)}
		\text{ with } \quad 
		|R(u)| & \leq  C \tau^2 \nor{u}{L^{2}}^{2} 
		+ C \nor{\nablag u}{L^{2}}^{2}. 
	\end{align}
\end{proposition}

The proof of Proposition~\ref{thmCarlemancalcul-epsilonzero} is inspired by~\cite{L:10} for the Schr\"odinger operator and~\cite{LL:18} for elliptic operators. 
It relies on the Riemannian tools presented in Section~\ref{s:Riemtool}. In~\cite{L:10}, a positivity assumption on the (space) Hessian for the weight function is made (related to the pseudoconvexity assumption in~\cite{LZ:82,Dehman:84,isakov1993carleman}). Here, the possibility of having $\frac{1}{\tau} \nor{u_{t}}{L^{2}}^{2}$ as a remainder term and the introduction of the function $f$ allow to relax this convexity condition and stay closer to the elliptic case as presented in~\cite{LL:18}. 

\bnp[Proof of Proposition \ref{thmCarlemancalcul-epsilonzero}]
We start by computing 
\bna
P_{\phi} u= e^{\tau\phi}(i\partial_{t}+ \Lap) (e^{-\tau\phi}u) =iu_t-i\tau\phi_tu+ \Lap u - 2\tau\gl{\nablag \phi}{\nablag u}- \tau(\Lap \phi) u+\tau^2\gln{\nablag \phi}u .
\ena
We then decompose the conjugated operator $P_{\phi}$ as 
\begin{align*}
	P_{\phi}  & = i\d_t-i\tau\phi_t+Q_2 + Q_1,  \quad \text{ with }\\
	Q_1u & := - 2\tau\gl{\nablag \phi }{\nablag u} - \tau f  u , \\
	Q_2u & :=\Lap u+\tau^2\gln{\nablag \phi}u - \tau(\Lap \phi) u + \tau f u=\widetilde{Q}_2u+R_{2}u ,
\end{align*}
where $\widetilde{Q}_2$ is the principal part of $Q_2$, that is
$$
\widetilde{Q}_2u=\Lap u+\tau^2\gln{\nablag \phi}u , \quad \text{and} \quad  R_2 u =  \tau(- \Lap \phi + f ) u .
$$
Now, we write ($\left\| \cdot \right\|$ denotes the $L^2$ norm for short and $(\cdot,\cdot)$ the associated Hermitian inner product)
\begin{align}
	\label{CarlReIm}
	3\left\|P_{\phi}  u\right\|^2+3\left\|R_{2}u\right\|^2 +3\left\|\tau\phi_tu\right\|^2& \geq \left\|P_{\phi} u-R_{2}u+i\tau\phi_tu\right\|^2  = \left\|iu_t+Q_1u+\widetilde{Q}_2u\right\|^2, 
\end{align}
where we estimate the remainders as 
\begin{align}
	\label{e:estimR2}
	\left\|R_{2}u\right\|^2\leq \tau^{2}\nor{f-\Lap \phi }{L^{\infty}}^{2}  \nor{u}{L^{2}}^{2} , \quad \text{and} \quad   
	\left\|\tau\phi_tu\right\|^2\leq  \tau^{2}\nor{\phi_t }{L^{\infty}}^{2}  \nor{u}{L^{2}}^{2} .
\end{align}
Hence, we are left to produce a lower bound for
\begin{align}
	\left\| iu_t+Q_1u+\widetilde{Q}_2u\right\|^2  & =  \left\|Q_1u\right\|^2+\left\| iu_t+\widetilde{Q}_2u\right\|^2+2\Re\big( iu_t,Q_1u\big)+2\Re \big(Q_1u,\widetilde{Q}_2u\big) \nonumber \\
	& \geq  2\Re \big( iu_t,Q_1u\big)+2\Re \big(Q_1u,\widetilde{Q}_2u\big) . 
	\label{e:square-int}
\end{align}
The second term in the right hand-side of~\eqref{e:square-int} is described in Lemma~\ref{l:computeM1M2} below, and we now estimate the first term as a remainder.
Recalling the expression of $Q_1$, we decompose
\begin{align}
	& 2\Re \big( iu_t,Q_1u \big)  = 2I_1+I_{2} , \quad \text{ with }  \label{e:decomp-I1-I2}  \\
	&I_1 : =- 2\tau\Re \big( iu_t,\gl{\nablag \phi }{\nablag u}\big), \quad \text{ and } \quad I_2 := -2\tau\Re \big( iu_t,fu\big) . \nonumber
\end{align}
Expanding $2\Re a = a+\overline{a}$ for $I_1$ and performing an integration by parts in $t$ for the first term, we obtain 
\begin{align*}
	I_1&=\tau \iint i\gl{\nablag \phi }{\nablag u}\ubar_t -i \tau \iint \gl{\nablag \phi }{\nablag \ubar}u_t\\
	&=\tau \iint -i\left[\gl{\nablag \phi_{t} }{\nablag u} +\gl{\nablag \phi }{\nablag u_{t}}\right]\ubar-i \tau\iint \gl{\nablag \phi }{\nablag \ubar}u_t 
\end{align*}
Concerning the last term, an integration by parts in $x$ yields
$$
-i \iint \gl{\nablag \phi }{\nablag \ubar} u_t=  i\iint (\Lap \phi) \ubar u_t + i \iint \gl{\nablag \phi }{\nablag u_{t}}\ubar.
$$
As a consequence, we deduce
$$
I_1= \tau \iint -i \gl{\nablag \phi_{t} }{\nablag u} \ubar + i\tau \iint (\Lap \phi) \ubar u_t.
$$
The Cauchy-Schwarz inequality yields
\begin{align}
	\label{e:estim-I1}
	2 |I_1| & \leq 2\left| \tau \iint -i \gl{\nablag \phi_{t} }{\nablag u} \ubar \right| + 2\left|\tau \iint (\Lap \phi) \ubar u_t\right|  \nonumber \\
	& \leq \nor{\nablag \phi_t}{L^{\infty}}^2\tau^{2}\nor{u}{L^{2}}^{2}+\nor{\nablag u}{L^{2}}^{2} + \delta \tau^{3} \nor{u}{L^{2}}^{2} +\frac{\nor{\Lap \phi}{L^{\infty}}^{2}}{\delta}\frac{1}{\tau} \nor{u_{t}}{L^{2}}^{2} .
\end{align}
We obtain similarly 
\begin{align}
	\label{e:estim-I2}
	|I_{2}|\leq \delta \tau^{3} \nor{u}{L^{2}}^{2} +  \frac{\nor{f}{L^{\infty}}^{2}}{\delta}\frac{1}{\tau} \nor{u_{t}}{L^{2}}^{2} .
\end{align}
We now provide with a lower bound for the second term in the right hand-side of~\eqref{e:square-int}.
The following result is a version of~\cite[Lemma~A.7]{LL:18} for complex valued functions $u$ in the boundaryless case (recall the definitions of $\B_{g, \phi, f}(X)$ and $\E_{g, \phi, f}$ in~\eqref{e:def-B} and~\eqref{e:def-E}).

\begin{lemma}
	\label{l:computeM1M2}
	Given an open set $\Omega \subset \R^{1+d}$, for all functions $\phi \in W^{2,\infty}_{\loc}(\Omega; \R),f \in W^{1,\infty}_{\loc}(\Omega; \R)$ and $u \in H^2_{\comp}(\Omega;\C)$, we have 
	\begin{align*}
		\Re \big(Q_1u,\widetilde{Q}_2u\big)
		= \tau^3\iint \E_{g, \phi, f} |u|^2 
		+\tau \iint \B_{g, \phi, f}(\nablag u)
		+ \tau  \Re \iint u \gl{ \nablag f}{\nablag \bar{u}}.
	\end{align*}
\end{lemma}
Lemma~\ref{l:computeM1M2} is a consequence of~\cite[Lemma~A.7]{LL:18} applied to $\Re(u)$ and $\Im(u)$ (with vanishing boundary terms), using that $Q_1,\widetilde{Q}_2$ have real coefficients, hence are $\C-$linear (which follows from the fact that $\phi$ and $f$ are real-valued).

In the estimates of Lemma~\ref{l:computeM1M2}, the last term is estimated as a remainder as
\begin{align}
	\label{e:reste-R3}
	R_3 (u)  = - \Re \tau \iint u \gl{ \nablag f}{\nablag \bar{u}}, \qquad 
	|R_3 (u)|  \leq \frac{\nor{\nablag f}{L^{\infty}}}{2}\left(\nor{\nablag u}{L^{2}}^{2}+\tau^{2}\nor{u}{L^{2}}^{2}\right) ,
\end{align}
Now, combining~\eqref{e:square-int} with~\eqref{CarlReIm} and~\eqref{e:decomp-I1-I2} yields
\begin{align*}
	3\left\|P_{\phi}  u\right\|^2+3\left\|R_{2}u\right\|^2 +3\left\|\tau\phi_tu\right\|^2 +2|I_1| +|I_2| \geq 2\Re \big(Q_1u,\widetilde{Q}_2u\big)  .
\end{align*}
This  combined with~\eqref{e:estim-I1}-\eqref{e:estim-I2} and Lemma~\ref{l:computeM1M2} concludes the proof of the proposition with
\bna
R(u)=3\left\|R_{2}u\right\|^2 +3\left\|\tau\phi_tu\right\|^2+ |R_3 (u)| +C \tau^2 \nor{u}{L^{2}}^{2} 
+ C \nor{\nablag u}{L^{2}}^{2},
\ena
with the first two terms estimated in~\eqref{e:estimR2} and the third in~\eqref{e:reste-R3}.
\enp

\subsubsection{The case $\e>0$: end of the proof of Proposition~\ref{p:subellipti-xit=0}}
\label{e:subell-in}

The strategy of the proof of Proposition~\ref{p:subellipti-xit=0} is to follow step by step the proof of Proposition~\ref{thmCarlemancalcul-epsilonzero} and control the additional error terms. Therefore, we will make use of the different terms appearing in the proof of Proposition~\ref{thmCarlemancalcul-epsilonzero} like $\widetilde{Q}_2,Q_2,Q_1,R_2$.

Thanks to Remark~\ref{carleman insensitive lot} it suffices to prove the inequality of Proposition~\ref{p:subellipti-xit=0} for the operator $\mathsf{P}_{\phi,\e}$ defined in~\eqref{e:def-PPhi-mu-2}. We start by expressing it in terms of $P_\phi$. Recall that by assumption $\phi$ is a quadratic polynomial and therefore $\phi''_{t,j} = \d^2_{t,x_j}\phi$  are actually constants. We have

\begin{align*}
	\mathsf{P}_{\phi, \e}&=P_\phi -\sum_{j,k=1}^dg^{jk}(D_j+ i \tau \d_j \phi )\e \phi''_{t,k} \frac{D_t}{\tau^2}+ g^{jk}\e \phi''_{t,j} \frac{D_t}{\tau^2}(D_k+ i \tau \d_k \phi )\\
	&\quad +\e^2 \sum_{j,k=1}^d \phi''_{t,k} \cdot\phi''_{t,j} g^{jk} \frac{D^2_t}{\tau^4}+\widetilde{R}_1 \\
	&=P_\phi-2\e \sum_{jk} \phi''_{t,k} g^{jk} \frac{D_j D_t}{\tau^2}+\e^2 \sum_{jk} \phi''_{t,k} \cdot\phi''_{t,j} g^{jk} \frac{D^2_t}{\tau^4}+\widetilde{R}_2\\
	&=\widetilde{P}_{\phi,\e}+\widetilde{R}_2, 
\end{align*}
where the operators $\Tilde{R}_1,\Tilde{R}_2 $ belong to the class of admissible perturbations considered in Remark~\ref{carleman insensitive lot} and 
$$
\widetilde{P}_{\phi,\e}:=P_\phi+2\e \sum_{jk} \phi''_{t,k} g^{jk} \frac{\d_i \d_t}{\tau^2}-\e^2 \sum_{jk} \phi''_{t,k} \cdot\phi''_{t,j} g^{jk} \frac{\d^2_t}{\tau^4} .
$$
It suffices then to show the estimate of Proposition~\ref{p:subellipti-xit=0} for the operator $\widetilde{P}_{\phi,\e}$. We decompose
$$
\widetilde{P}_{\phi,\e}=i \d_t -i \tau \d_t \phi+\widetilde{Q}_{1,\e}+\widetilde {Q}_{2,\e},
$$
where, using the notation $\widetilde{Q}_2,Q_2,Q_1,R_2$ from the proof of Proposition~\ref{thmCarlemancalcul-epsilonzero} in Section~\ref{section proof for epsilon zero}, 
$
\widetilde{Q}_{1,\e}=Q_1
$
and
\begin{align*}
	\widetilde {Q}_{2,\e}=Q_2+2\e \sum_{jk} \phi''_{t,k} g^{jk} \frac{\d_i \d_t}{\tau^2}-\e^2 \sum_{jk} \phi''_{t,k} \cdot\phi''_{t,j} g^{jk} \frac{\d^2_t}{\tau^4} 
	=\widetilde{\mathcal{Q}}_{2,\e}+R_2
\end{align*}
with 
\begin{equation}
	\label{def of perturbation of Q2}
	\widetilde{\mathcal{Q}}_{2,\e}:=\widetilde{Q}_2+2\e \sum_{jk} \phi''_{t,k} g^{jk} \frac{\d_i \d_t}{\tau^2}-\e^2 \sum_{ij} \phi''_{t,k} \cdot\phi''_{t,j} g^{jk} \frac{\d^2_t}{\tau^4},
\end{equation}
As in the proof of Proposition~\ref{thmCarlemancalcul-epsilonzero}, the terms $R_2$ and $i\tau \partial_t \phi$ are admissible remainders. As above, we need to provide a lower bound for
\begin{align}
	\label{decomposition carre+ commutateur}
	\nor{i \d_t u + \widetilde{Q}_{1,\e}u+\widetilde{\mathcal{Q}}_{2,\e}u}{L^2}^2&=  \left\|Q_1u\right\|^2+\left\| i\d_t u+\widetilde{\mathcal{Q}}_{2,\e}u\right\|^2+2\Re ( iu_t,Q_1u)+2\Re (Q_1u,\widetilde{\mathcal{Q}}_{2,\e}u) \nonumber  \\
	&= \left\|Q_1u\right\|^2+\left\| i\d_t u+\widetilde{\mathcal{Q}}_{2,\e}u\right\|^2+2\Re ( iu_t,Q_1u)+2\Re (Q_1u,\widetilde{Q}_2u) \nonumber \\
	&\hspace{4mm}+2\Re (Q_1u,(\widetilde{\mathcal{Q}}_{2,\e}-\widetilde{Q}_2)u).
\end{align}
It follows that in order to finish the proof of Proposition~\ref{p:subellipti-xit=0} it suffices to show that the last term in~\eqref{decomposition carre+ commutateur} yields an admissible error in view of the estimate~\eqref{Carlpropwaveinterm}. This is the content of the following lemma.
\begin{lemma}
	\label{lemma perturbation of p phi}
	There exist $C, \tau_0 >0$ such that for all $u \in C^\infty_c(\Omega)$ one has
	\begin{equation*}
		\left|2\Re \left(Q_1u,(\widetilde{\mathcal{Q}}_{2,\e}-\widetilde{Q}_2)u\right)\right|\leq C\nor{u}{\Ht{1}}^2,\quad \textnormal{ for all }\tau \geq \tau_0.
			\end{equation*}
\end{lemma}
\begin{proof}[Proof of Lemma~\ref{lemma perturbation of p phi}]
	Recalling that $Q_1 u= - 2\tau\gl{\nablag \phi }{\nablag u} - \tau f u$ and writing $\widetilde{\mathcal{Q}}_{2,\e}-\widetilde{Q}_2=L_1+L_2$ with
	$$
	L_1:=2\e \sum_{jk} \phi''_{t,k} g^{jk} \frac{\d_j \d_t}{\tau^2},
	\quad \text{ and }\quad 
	L_2:=-\e^2 \sum_{jk} \phi''_{t,k} \cdot\phi''_{t,j} g^{jk} \frac{\d^2_t}{\tau^4},
	$$
	we may develop
	\begin{align}
		\label{temrs to contorl for perturbation Aj}
		&\Re \left(Q_1u,(\widetilde{\mathcal{Q}}_{2,\e}-\widetilde{Q}_2)u\right) = A_1+A_2+A_3+A_4, \quad \text{ with }\\
		&   A_1 :=-2\Re(\tau\gl{\nablag \phi }{\nablag u}, L_1u) ,\quad 
		A_2:= -2\Re(\tau\gl{\nablag \phi }{\nablag u}, L_2u) , \nonumber \\ 
		&   A_3:= -\Re(\tau f u,L_1u) , \quad     A_4:= -\Re(\tau f u,L_2 u) .\nonumber
	\end{align}
	We start by estimating the terms $A_3$ and $A_4$. Integrating by parts in $t$, we obtain 
	\begin{align*}
		A_3&=-2 \frac{\e}{\tau} \Re \bigg( fu,\d_t \sum_{jk}\phi''_{t,k} g^{jk}\d_j u \bigg)\\
		&=2  \frac{\e}{\tau} \sum_{jk} \left[ \Re \bigg( (\d_t f)u,\phi''_{t,k} g^{jk}\d_j u \bigg) + \Re \bigg( f\d_t u, \phi''_{t,k} g^{jk}\d_j u \bigg)\right] .
	\end{align*}
	Therefore, the Cauchy-Schwarz inequality implies, for a constant $C>0$ depending on $f, \phi $ and $g$, 
	\begin{align}
		\label{estimate for A3}
		|A_3|\leq C \bigg( \nor{u}{L^2}^2+\nor{\nabla_x u}{L^2}^2+\frac{\nor{D_t u}{L^2}^2}{\tau^2} \bigg),  \quad \tau \geq 1.
	\end{align}
	Similarly, integrating by parts in time yields
	\begin{align}
		\label{estimate for A4}
		|A_4|\leq \frac{C}{\tau^3}\left(\nor{u}{L^2}^2+\nor{D_t u}{L^2}^2\right).
	\end{align}
	We now turn our attention to $A_1$. Here one needs to use the real part in order to decrease the number of derivatives. We write $\gl{\nablag \phi }{\nablag u}=\sum_{jk}g^{jk}\d_j\phi \d_k u$ and $2\Re a =a +\bar{a}$ to obtain 
	\begin{align}
		\label{estim aux for A1}
		-A_1&=2\Re\bigg(\tau \sum_{jk}g^{jk}\d_j\phi \d_j u, 2\e \sum_{lm}\phi''_{t,m} g^{lm} \frac{\d_l \d_t}{\tau^2} u\bigg) \nonumber \\
		&=\frac{2\e}{\tau}\sum_{jklm}\left( g^{jk}\d_j\phi \d_k u, \phi''_{t,m} g^{lm} \d_l \d_t u\right)+  \left(  \phi''_{t,m} g^{lm} \d_l \d_t u ,  g^{jk}\d_j\phi \d_k u \right).
	\end{align}
	Integrating by parts in $t$ in the first term in the right-hand side~\eqref{estim aux for A1} yields
	\begin{align*}
		& \sum_{jklm}\left(g^{jk}\d_j\phi \d_k u, \phi''_{t,m} g^{lm} \d_l \d_t  u \right)\\
		&=- \sum_{jklm} \left(g^{jk}\phi''_{j,t} \d_k u,\phi''_{t,m} g^{lm} \d_l u   \right) +\left( g^{jk}\d_j \phi \d_{k}\d_{t}u,  \phi''_{t,m} g^{lm}\d_l u \right)\\
		&=-\sum_{jklm} \left(g^{jk}\phi''_{j,t} \d_k u, \phi''_{t,m} g^{lm} \d_l u   \right)+  \left( \phi''_{t,m} g^{lm} \d_l \d_t u , g^{jk}\d_j\phi \d_k u \right) .
	\end{align*}
	Together with~\eqref{estim aux for A1}, this implies 
	$$
	A_1=\frac{2\e}{\tau} \sum_{jklm} \left(g^{jk}\phi''_{j,t} \d_k u,\phi''_{t,m} g^{lm} \d_l u   \right),
	$$
	and thus 
	\begin{equation}
		\label{estimate for A1}
		|A_1|\leq \frac{C}{\tau} \nor{\nabla_x u}{L^2}^2.
	\end{equation}
	Finally, to estimate $A_2$ we proceed similarly by writing
	\begin{align}
		\label{estimate aux for A2}
		A_2&=2\Re\left(\tau \sum_{jk}g^{jk}\d_j\phi \d_k u, \e^2 \sum_{lm}\phi''_{t,m} \cdot\phi''_{t,l} g^{lm} \frac{\d^2_t}{\tau^4} u\right) \nonumber \\
		&=\frac{\e^2}{\tau^3}\sum_{jklm} \left( g^{jk}\d_j\phi \d_k u,\phi''_{t,m} \cdot\phi''_{t,l} g^{lm} \d^2_t u\right)+  \left( \phi''_{t,m} \cdot \phi''_{t,l} g^{lm} \d^2_t u, g^{jk}\d_j\phi \d_k u \right).
	\end{align}
	We integrate by parts in $t$ in the first term in the right-hand side of~\eqref{estimate aux for A2} to obtain
	\begin{align}
		\label{estim aux for A22 2}
		&\sum_{jklm}\left( g^{jk}\d_j\phi \d_k u,\phi''_{t,m} \cdot\phi''_{t,l} g^{lm} \d^2_t u\right)=A_{21}+A_{22} , \quad \text{ with }\\
		&A_{21}:=-\sum_{jklm}\left( g^{jk}\phi''_{j,t} \d_k u,\phi''_{t,m} \cdot\phi''_{t,l} g^{lm} \d_t u\right) , \nonumber \\
		&A_{22}:=-\sum_{jklm}\left( g^{jk}\d_{j}\phi \d^2_{t,k} u,\phi''_{t,m} \cdot\phi''_{t,l} g^{lm} \d_t u\right) . \nonumber
	\end{align}
	To facilitate the notation, we write in what follows $S_j$ for multiplication operators by $L^\infty$ functions that depend only $g, D_xg$, on $\phi$ and its derivatives. 
	We integrate by parts in $x$ and then in $t$ to find
	\begin{align}
		\label{estim for A22}
		A_{22} &=- \sum_{jklm} \left(\d_k \left(g^{jk}\d_{j}\phi \d_{t} u\right),\phi''_{t,m} \cdot\phi''_{t,l} g^{lm} \d_t u\right)+\sum_{lm}\left(S_1 \d_t u,  \phi''_{t,m} \cdot\phi''_{t,l} g^{lm} \d_t u \right) \nonumber\\
		&= \sum_{jklm}\left( g^{jk}\d_{j}\phi \d_{t} u, \phi''_{t,m} \cdot\phi''_{t,l} g^{lm} \d^2_{t,k} u\right)+\left(S_2 \d_t u, \d_t u \right)\nonumber \\
		&=- \sum_{jklm} \left(  g^{jk}\d_{j}\phi \d^2_{t} u, \phi''_{t,m} \cdot\phi''_{t,l} g^{lm} \d_{k} u\right) +\left(S_2 \d_t u,  \d_t u \right)+\sum_{j} \left(S_{3,j} \d_t u,  \d_j u \right) \nonumber \\
		&=- \sum_{jklm} \left(  \phi''_{t,m} \cdot\phi''_{t,l} g^{lm} \d^2_t u,g^{jk}\d_j\phi \d_k u \right)+\left(S_2 \d_t u,  \d_t u \right)+\sum_{j} \left(S_{3,j} \d_t u,  \d_j u \right).
	\end{align}
	Now putting together~\eqref{estimate aux for A2},\eqref{estim aux for A22 2} and~\eqref{estim for A22} implies
	$$
	A_2= \frac{\e^2}{\tau^3}\Big( A_{21}+\left(S_2 \d_t u,  \d_t u \right)+\sum_{j} \left(S_{3,j} \d_t u,  \d_j u \right) \Big) .
	$$
	We obtain therefore
	\begin{equation}
		\label{estim for A2}
		|A_2|\leq \frac{C}{\tau^3}\left( \nor{\nabla_x u}{L^2}^2+ \nor{D_t u}{L^2}^2\right).
	\end{equation}
	Plugging~\eqref{estimate for A3},~\eqref{estimate for A4},~\eqref{estimate for A1} and~\eqref{estim for A2} in~\eqref{temrs to contorl for perturbation Aj} finishes the proof of the lemma.
\end{proof}

With Lemma~\ref{lemma perturbation of p phi}, we can now conclude the proof of the subelliptic estimate of Proposition~\ref{p:subellipti-xit=0}. 
\begin{proof}[End of the proof of Proposition~\ref{p:subellipti-xit=0}]
   Recall now that it suffices to obtain a lower bound for
	$$
	\nor{i \d_t u + \widetilde{Q}_{1,\e}u+\widetilde{\mathcal{Q}}_{2,\e}u}{L^2}^2 \geq 2\Re ( iu_t,Q_1u)+2\Re (Q_1u,\widetilde{Q}_2u) 
	+2\Re (Q_1u,(\widetilde{\mathcal{Q}}_{2,\e}-\widetilde{Q}_2)u),
	$$
	where we used decomposition~\eqref{decomposition carre+ commutateur}. The first two terms on the right-hand side above are estimated in Section~\ref{section proof for epsilon zero}. The first one yieds an admissible error thanks to~\eqref{e:decomp-I1-I2},~\eqref{e:estim-I1},~\eqref{e:estim-I2} and the second one is calculated in Lemma~\ref{l:computeM1M2}. Combining those estimates with Lemma~\ref{lemma perturbation of p phi} which controls the third term above we obtain the existence of $C, \tau_0>0$ such that for all $\delta>0, \tau\geq \tau_0$ and $u \in C^\infty_c(\Omega)$ one has
	\begin{align*}
		\nor{\mathsf{P}_{\phi,\e} u}{L^2}^2
		+ \frac{C}{\delta \tau} \nor{u_{t}}{L^{2}}^{2}
		+ C \nor{u}{H^1_\tau}^2\geq \tau^3\iint \left[\E_{g, \phi, f} -\delta\right]|u|^2 + 2  \tau \iint \B_{g, \phi, f}(\nablag u), 
	\end{align*}
	Recalling Assumption~\eqref{e:sub-ellipticity-EB} in Proposition~\ref{p:subellipti-xit=0}, we may now fix $\delta := \frac{C_0}{2}$ to obtain 
	\begin{align*}
		\nor{\mathsf{P}_{\phi,\e} u}{L^2}^2
		+ \frac{2C}{C_0 \tau} \nor{u_{t}}{L^{2}}^{2}
		+ C \nor{u}{H^1_\tau}^2\geq \frac{C_0}{2} \tau \nor{u}{H^1_\tau}^2
	\end{align*}
	This concludes the proof of Proposition~\ref{p:subellipti-xit=0} when taking $\tau \geq \tau_0$ for $\tau_0$ sufficiently large.
\end{proof}

\subsection{Choice of weight function via convexification}
\label{s:convexification}
In this section, we explain how to construct weight functions $(\check{\phi}, f)$ that \textit{almost} satisfy the assumptions of Theorem~\ref{th:carlemanschrod}, via the usual convexification procedure. In the present context (as opposed to the usual situation), this also requires a smart choice of the function $f$, see~\cite{LL:18}.

The main difference with respect to the assumptions of Theorem~\ref{th:carlemanschrod} is that the function $\check{\phi}$ that we construct here is not a quadratic polynomial. In Section~\ref{section using the carleman estimate} we shall see however that since the positivity of the quantities $\B$ and $\E$ is a condition that only involves derivatives up to order $2$ one can replace $\check{\phi}$ by its Taylor expansion at order 2.
The following is~\cite[Lemma~A.9]{LL:18}. 
\begin{lemma}[Explicit convexification]
	\label{l:explicit-convexification-G}
	Let $\Psi \in W^{2,\infty}(\Omega;\R)$ and $G \in W^{2,\infty}(\R)$, and choose 
	\bnan
	\label{e:choice-phi-f}
	\check{\phi} = G(\Psi)\quad \text{ and } \quad  f = 2 G''(\Psi) \gln{\nablag \Psi}.
	\enan
	Then we have 
	\begin{align*}
		\B_{g, \check{\phi}, f}(X) & = 2 G'(\Psi) \Hess(\Psi)(X,\ovl{X}) + 2G''(\Psi) \left|\gl{\nablag \Psi }{ X}\right|^2
		+ \left(G''(\Psi) \gln{\nablag \Psi} -G'(\Psi) \Lap \Psi \right) \gln{X},\\
		\E_{g, \check{\phi}, f} & = G'(\Psi)^2\Big[2 G'(\Psi) \Hess(\Psi)(\nablag\Psi,\nablag\Psi)+ G''(\Psi)\left| \nablag \Psi\right|_g^4 
		+ G'(\Psi) \Lap \Psi \gln{\nablag \Psi}\Big].
	\end{align*}
\end{lemma}
To state the next corollary, for $B$ an $L^\infty_{\loc}$ section of bilinear forms on $T V$, we define $|B|_g (x) = \sup_{X \in T_x V\setminus 0}\frac{|B(x,X,\ovl{X})|}{\gln{X}}$ which yields a $L^\infty$ function on $V$. 
\begin{corollary}
	\label{c:explicit-convexification-exp}
	Let $\Psi \in W^{2,\infty}(\Omega;\R)$, $\lambda>0$ and define $\check{\phi}, f$ as in~\eqref{e:choice-phi-f} with $G(t)=e^{\lambda t}-1$. Then, for any $\lambda >0$ and any vector field $X$, we have almost everywhere on $U$
	\begin{align*}
		\B_{g, \check{\phi}, f}(X) &\geq
		\lambda e^{\lambda \Psi} \gln{X}  \left( \lambda \gln{\nablag \Psi} - 2 |\Hess(\Psi)|_g - \Lap \Psi \right) ,\\
		\E_{g, \check{\phi}, f}  & \geq
		\lambda e^{\lambda\Psi} \gln{\nablag \check{\phi}} \left(  \lambda \left| \nablag \Psi\right|_g^2 - 2 | \Hess(\Psi)|_g +  \Lap \Psi \right) .
	\end{align*}
\end{corollary}
See~\cite[Lemma~A.10]{LL:18} for a proof.

\section{Conjugation with a partially Gevrey function}
\label{section conjugation with gevrey}
In~\cite{Tataru:95,RZ:98,Hor:97,Tataru:99} part of the difficulty consists in defining an appropriate conjugated operator even in the case where the coefficients of $P$ depend analytically on the time variable. 
Here, we exploit the anisotropic nature of $P$ to allow conjugation with Gevrey $s$ in time functions, for an appropriate $s>1$ adapted to the scaling of the Schr\"odinger operator. Our strategy is based on the proof of Proposition 4.1 in~\cite{Tataru:99}. 

\subsection{Gevrey functions and Banach valued symbols}
\label{s:gevrey-fcts}
For notations, definitions and basic properties of Gevrey functions we essentially follow~\cite{bonthonneau2020fbi}. 
We recall Definition~\ref{d:gevrey} where the space $\mathcal{G}^s(\Omega;\ban)$ of Gevrey $s$ Banach valued is defined. We shall also make use of the following notion 

\begin{definition}
	Given $\mathsf{d}\in \N^*$, $U\subset \R^{\mathsf{d}}$ an open set, $(\ban, \| \cdot \|_{\ban})$ a Banach space, $s>0,R>0$ we say that $f \in \mathcal{G}^{s,R}_b(U;\ban)$, 
	if $f\in C^\infty_b(U;\ban)$ (smooth bounded functions, as well as all their derivatives) and there exists $C>0$ such that
	\begin{equation}
		\label{e:gevrey}
		\nor{\d^\alpha f(t)}{\ban}\leq C R^{|\alpha|} \alpha!^s, \quad \text{ for all } t \in U , \alpha \in \N^{\mathsf{d}} .
	\end{equation}
	and set 
	\begin{equation}
		\label{e:gevrey-norm}
		\nor{f}{s,R,U}: =\underset{\alpha \in \N^{\mathsf{d}} }{\sup} \sup_{t \in U} \frac{\nor{\d^\alpha f(t)}{\ban}}{R^{|\alpha|}\alpha!^s}. 
	\end{equation}
\end{definition}
In what follows, we only consider the case $\mathsf{d}=1$ ($t$ being the time variable) and $\mathsf{d}=2$ for extensions to $\C\simeq \R^2$ of such Gevrey functions.
Note that, given an open set $U$ and $s,R>0$ fixed, $\mathcal{G}^{s,R}_b(U;\ban)$ has the advantage of being a Banach space for the norm $\nor{\cdot}{s,R,U}$ in~\eqref{e:gevrey-norm}.
Note also that for any $R>0$, $\mathcal{G}^{s,R}_b(U;\ban) \subset \mathcal{G}^{s}(U;\ban)$.
Conversely, if $f \in \mathcal{G}^{s}(U;\ban)$, then for any bounded open set $W$ such that $\overline{W} \subset U$, there exists $R>0$ such that $f \in \mathcal{G}^{s,R}_b(W;\ban)$.

\bigskip
The following lemma contains the key properties which we will need concerning Gevrey functions. 
\begin{lemma}
	\label{extension of gevrey}
	Fix $s>1$. For any open set $U \subset \R$ and $\rho>0$, there exist $C_0, A>0$ such that for any $R>0$, there exist $C >0$ and a continuous linear map
	$$
	\begin{array}{rcl}
		\mathcal{G}^{s,R}_b(U;\ban) & \to & \mathcal{G}^{s,AR}_b(U+i\R ;\ban) \\
		f & \mapsto & \tilde{f},
	\end{array}
	$$
	such that for all $f \in \mathcal{G}^{s,R}_b(U;\ban)$, 
	\begin{align}
		\label{e:continuous-extension}
		&  \supp(\tf) \subset U + i [-\rho,\rho], \quad  \tf(t)=f(t) \text{ for } t \in U , \quad \nor{\tf}{s,AR, U+i\R} \leq C \nor{f}{s,R,U},  \\
		\label{estimate for dzb}
		& \nor{\dzb \tf(z)}{\ban}\leq C \nor{f}{s,R,U}\exp{\left(-\frac{1}{C_0(R|\Im(z)|)^{\frac{1}{s-1}}}\right)},\quad \text{ for } z\in U+i\R  ,\\
		\label{propriete de commutation}
		&  \d_{\Re z}^j \tilde{f}(z)=\widetilde{f^{(j)}}(z) \quad \text{ for all }j \in \N  \text{ and }z\in U+i\R . 
	\end{align}
\end{lemma}

Estimate~\eqref{estimate for dzb} translates the fact that $\tf$ is an \textit{almost analytic extension} of $f$ well-adapted to the Gevrey regularity $\mathcal{G}^{s}$.
Property~\ref{propriete de commutation} states that the operation of derivation w.r.t. the real part and taking the almost analytic extension commute.

If $\ban=\C$, Lemma~\ref{extension of gevrey} is essentially a consequence of Lemma 1.2 and Remark 1.7 in~\cite{bonthonneau2020fbi} (in a simpler $1D$ context). The proof in this reference does not seem to adapt straightforwardly to the case of Banach-valued functions, so we provide here with a short and different proof.

Our proof of Lemma~\ref{extension of gevrey} relies on the following classical result which is the key step (and mostly equivalent) for the Borel extension problem in Gevrey classes. 

\begin{lemma}
	\label{l:russe}
	For all $s>1$, there are constants $B,C>1$ and a family $(\zeta_{k,D}) \in C^\infty(\R)^{\N\times [1,+\infty)}$ such that for all $D\geq 1,k \in \N, j \in \N$,
	$$
	\zeta_{k,D}^{(j)}(0) =\delta_{jk} ,\quad |\zeta_{k,D}^{(j)}(x)| \leq C^{j+1}B^{k} D^{j-k} k^{-ks} \max(k,j)^{js} \text{ for all } x \in  \R .
	$$
\end{lemma}
An explicit construction of such functions $\zeta_{k,D}$ is given in~\cite{Dz:62}. Another less explicit construction but with improved estimates on the constants is provided in~\cite{MRR:16}. 
In both cases, the functions are constructed as $\zeta_{k,D}(t) := a_{k,D}(t) \frac{t^k}{k!}$ with an appropriate family $a_{k,D}(t)$ satisfying $a_{k,D}(0)=1, a_{k,D}^{(j)}(0)=0$ for all $j \geq 2$ together with $\supp(a_{k,D}) \subset [-(Dk^s)^{-1}, (Dk^s)^{-1}]$ (for $k \geq 1$) and appropriate estimates of Gevrey $s$ norm. In~\cite{Dz:62}, $a_{k,D}(t)$ is defined by an explicit expression on page 1 and the estimates are proved on page 4.

In~\cite{MRR:16},  the notation is $a_{k,D}(t) = \varphi_k(t), M_p=p^{ps}, h=D$ and $\varphi_k$ is defined on page 14 and $\zeta_{k,D}=\zeta_k$ on page 15, and the estimates are performed on page 16 and correspond to~(3.17) (in that reference) which is even better, namely $|\zeta_{k,D}^{(j)}(x)| \leq C^{j+1}B^{-k} D^{j-k} k^{-ks} j^{js}$, and is (essentially) equivalent to $\nor{\zeta_{k,D}}{s,CD,\R} \leq C(BD)^{-k}k^{-ks}$. This result of~\cite{MRR:16} is a refinement of~\cite[Theorem~2.2]{Petzsche:88} where the dependence in the parameter $D$ (called $h$ in these two references) is not made explicit.

\bnp[Proof of Lemma~\ref{extension of gevrey}]
From Lemma~\ref{l:russe} we first define 
\begin{align}
	\label{e:def-tilde-f}
	\check{f}(x+iy) : = \sum_{k \in \N}  \d^kf(x) i^k \zeta_{k,D}(y) , \quad (x,y) \in U\times \R .
\end{align}
We first check that for $D$ large enough (fixed later on in the proof), the series converge normally as well as all its derivatives, and prove the estimate in~\eqref{e:continuous-extension} at once. To this aim, we follow essentially~\cite[Proof of Lemma~3.1]{BP:09}. 
From~\eqref{e:gevrey-norm} we have $\nor{\d^k f(t)}{\ban} \leq R^{k}k!^s\nor{f}{s,R,U}$ for all $t \in U$ and thus, uniformly for $(x,y)\in  U\times \R$, 
\begin{align*}
	\nor{\d_x^j \d_y^\ell \check{f}(x+iy)}{\ban} & = \nor{\sum_{k \in \N} \d_x^{k+j}(f)(x)i^k \d_y^\ell \zeta_{k,D}(y)}{\ban} 
	\leq  \sum_{k \in \N} \nor{\d_x^{k+j}(f)(x)}{\ban}  |\d_y^\ell \zeta_{k,D}(y)|  \\
	& \leq  \nor{f}{s,R,U} \sum_{k \in \N} R^{k+j}(k+j)!^s  C^{\ell+1}B^{k} D^{\ell-k} k^{-ks} \max(k,\ell)^{\ell s},
\end{align*}
where we used Lemma~\ref{l:russe} in the last inequality.
We recall the classical inequalities (see e.g. \cite[p10-11]{Rodino:93}): $(k+j)! \leq 2^{k+j}k! j!$, $N! \leq N^N$ and $N! \geq (N/e)^N$.
We deduce
\begin{align}
	\label{e:estim-extension}
	\nor{\d_x^j \d_y^\ell \check{f}(x+iy)}{\ban} &\leq  \nor{f}{s,R,U} \sum_{k \in \N} R^{k+j}2^{sk+sj}k!^s j!^s  C^{\ell+1}B^{k} D^{\ell-k} k^{-ks} \max(k,\ell)^{\ell s} \nonumber  \\
	&\leq  \nor{f}{s,R,U} C(R 2^s)^j j^{js} (CD)^\ell \sum_{k \in \N} (R2^{s}B)^k  D^{-k}  \max(k,\ell)^{\ell s}  .
\end{align}
Then we split the sum as
\begin{align*}
	\sum_{k \in \N} (R2^{s}B)^k D^{-k}  \max(k,\ell)^{\ell s} &=  \sum_{k \leq \ell } (R2^{s}B)^k  D^{-k} \ell^{\ell s}+ \sum_{k > \ell } (R2^{s}B)^k  D^{-k} k^{\ell s} .
\end{align*}
In the last sum we use $k^{\ell s} \leq e^k \left(\frac{\ell s}{e}\right)^{\ell s}$, which is a consequence of $x \geq \log(ex)$ taken for $x=\frac{k}{\ell s}>0$ (applied if $k>\ell>0$, and also true in case $\ell=0$). We obtain
\begin{align*}
	\sum_{k \in \N} (R2^{s}B)^k D^{-k}  \max(k,\ell)^{\ell s} &\leq  \ell^{\ell s} \sum_{k \leq \ell } \left(\frac{R2^{s}B}{D} \right)^{k} + \sum_{k > \ell } (R2^{s}B)^k  D^{-k} e^k \left(\frac{\ell s}{e}\right)^{\ell s} \\
	&\leq \ell^{\ell s} \sum_{k \leq \ell } \left(\frac{R2^{s}B}{D} \right)^{k} +  \left(\frac{\ell s}{e}\right)^{\ell s}  \sum_{k > \ell }  \left(\frac{R2^{s}B e }{D} \right)^{k} 
	\leq (\ell s)^{\ell s}  \sum_{k\in \N}  \left(\frac{R2^{s}B e }{D} \right)^{k} .
\end{align*}
We now fix $D := 2 \times R2^{s}B e$ and, coming back to~\eqref{e:estim-extension}, we obtain
\begin{align*}
	\nor{\d_x^j \d_y^\ell \check{f}(x+iy)}{\ban} &\leq  \nor{f}{s,R,U} C(R 2^s)^j j^{js} (C R2^{s+1}B e)^\ell  2  (\ell s)^{\ell s}=  
	2 C \nor{f}{s,R,U}  (R 2^s)^j (C R2^{s+1}B e s^s )^\ell  \ell^{\ell s}j^{js}   .
\end{align*}
Noticing that $\ell^{\ell s}j^{js}\leq e^{s(j+\ell)}j!\ell !$, we have obtained, uniformly for $(x,y)\in  U\times \R$, 
\begin{align}
	\label{e:estimate-derivatives}
	\nor{\d_x^j \d_y^\ell \check{f}(x+iy)}{\ban} \leq \mathsf{C} \nor{f}{s,R,U}   (AR)^{j+\ell}\ell!^s  j!^s ,\quad \text{ with }  \mathsf{C} = 2C , \quad A = C R2^{s+1}B s^s e^{s+1} .
\end{align}

Now, we take $1<\sigma<s$ and let $g \in \mathcal{G}^{\sigma}(\R;\R)$ be such that $\supp(g) \subset(-\rho,\rho)$ and $g = 1$ in a neighborhood of $0$ and we set 
$$
\tilde{f}(x+iy) := g(y)\check{f}(x+iy) , \quad (x,y) \in U\times \R ,
$$
so that $\tilde{f}$ has the sought support properties in~\eqref{e:def-tilde-f}.
That $\tf(x)=f(x)$ for $x\in U $ is a direct consequence of the definition~\eqref{e:def-tilde-f}, the properties of $g$ together with $\zeta_{k,D}(0) =\delta_{0k}$.
Property~\eqref{propriete de commutation}  is a direct consequence of the definition~\eqref{e:def-tilde-f} and derivation under the sum.

To deduce~\eqref{e:continuous-extension} from~\eqref{e:estimate-derivatives}, we write $\d_x^j \d_y^\ell \tilde{f}(x+iy)  = \d_y^\ell \left( g(y) \d_x^j \check{f}(x+iy) \right)$
and apply~\cite[Lemma~3.7]{MRR:16} with $g=g$ and $f = \d_x^j \check{f}(x+iy)$ (referring to the notation of this reference) for fixed $j$ (that the function is Banach-valued plays no role in the proof of~\cite[Lemma~3.7]{MRR:16}). 
This reference, combined with~\eqref{e:estimate-derivatives} for fixed $j$, implies the existence of a constant $C_{g,s}$ depending only on $g$ (and in particular on $\rho$ and $\sigma$) and $s$ such that for all $(x,y) \in U\times \R$,
\begin{align}
	\label{e:estimate-derivatives-bis}
	\nor{\d_x^j \d_y^\ell \tilde{f}(x+iy)}{\ban} =\nor{\d_y^\ell \left( g(y) \d_x^j \check{f}(x+iy) \right)}{\ban} \leq C_{g,s} \mathsf{C} \nor{f}{s,R,U}   (AR)^{j+\ell}\ell!^s  j!^s .
\end{align}
Noticing that $j!\ell ! \leq (j+\ell)!$, we have obtained the continuity statement in~\eqref{e:continuous-extension} with continuity constant $C_{g,s} \mathsf{C}$ (and $\mathsf{C},A$ given by~\eqref{e:estimate-derivatives}).

Finally, in order to prove~\eqref{estimate for dzb}, we notice that $\dzb\tf\in \mathcal{G}^{s,AR}_b(U+i\R ;\ban)$ since $\tf\in \mathcal{G}^{s,AR}_b(U+i\R ;\ban)$, and check that $\dzb \tf$ vanishes at infinite order on the real axis. Indeed, we have 
\begin{align*}
	\d_x^j \d_y^\ell (\d_x + i \d_y) \check{f}(x+iy) & =  \sum_{k \in \N} \d_x^{k+j+1}(f)(x)i^k \d_y^\ell \zeta_{k,D}(y) +\d_x^{k+j}(f)(x)i^{k+1} \d_y^{\ell+1} \zeta_{k,D}(y)  .
\end{align*}
Using that $\zeta_{k,D}^{(\ell)}(0) =\delta_{\ell k}$ and that $g=1$ in a neighborhood of $0$, this implies
\begin{align}
	\label{e:infinite-order-vanish}
	\d_x^j \d_y^\ell (\d_x + i \d_y) \tilde{f}(x+iy)\Big|_{y=0} & =  \sum_{k \in \N} \d_x^{k+j+1}(f)(x)i^k  \delta_{\ell k} +\d_x^{k+j}(f)(x)i^{k+1} \delta_{\ell+1,k} \nonumber  \\
	& =  \d_x^{\ell+j+1}(f)(x)i^\ell +\d_x^{\ell+1+j}(f)(x)i^{\ell+1+1} = 0 .
\end{align}
Applying the ``sommation au plus petit terme'' in~\cite[Lemma~1.3]{bonthonneau2020fbi} (which holds with the same proof in the Banach-valued case), there exist constants $C,C_0>0$ such that for all $F\in \mathcal{G}^{s,AR}_b(U+i\R ;\ban)$ and all $x+iy  \in U +i\R$
\begin{align*}
	\nor{ F(x+iy)- \sum_{\ell \leq C_0^{-1}(AR|y|)^{-\frac{1}{s-1}}} \frac{1}{\ell!}  (\d_y^\ell F ) (x+iy)\big|_{y=0}}{\ban}
	& \leq  C\nor{F}{s,AR, U+i\R}\exp \left( - \frac{1}{C (AR |y|)^{\frac{1}{s-1}}}\right) .
\end{align*}
We may apply this estimate to $F = \dzb\tf \in \mathcal{G}^{s,AR}_b(U+i\R ;\ban)$ according to the following consequence of~\eqref{e:estimate-derivatives-bis}
\begin{align*}
	\nor{\d_x^j \d_y^\ell(\d_x+i\d_y) \tilde{f}(x+iy)}{\ban} \leq 2 \mathsf{C}C_{g,s} \nor{f}{s,R,U}   (AR)^{j+\ell+1} (j+\ell+1)!^s  .
\end{align*}
Recalling the infinite order of vanishing~\eqref{e:infinite-order-vanish} finally yields~\eqref{estimate for dzb}, and concludes the proof of the lemma.
\enp

\bigskip
Consider now $\X,\Y$ two separable Hilbert spaces and denote by $\mathcal{L}(\X,\Y)$ the space of bounded operators from $\X$ to $\Y$, which is a Banach space as well for $\|\cdot\|_{\L(\X,\Y)}$. 
We recall some facts of semiclassical analysis in dimension $1$ with values in $\mathcal{L}(\X,\Y)$. We consider a family of symbols depending on a (small) parameter   $h \in (0,1)$.
We say that $a \in S^m(\R\times \R; \mathcal{L}(\X,\Y))$ if $a \in C^{\infty}(\R\times \R; \mathcal{L}(\X,\Y))$ depends implicitly on $h \in (0,1)$ and satisfies: for all $ \alpha , \beta \in \N^n$ there is $C_{\alpha\beta}>0$ such that 
$$ \nor{\d_x^\alpha\d_\xi^\beta a(t,\xi, h)}{\mathcal{L}(\X,\Y)} \leq C_{\alpha\beta} \langle \xi \rangle^{m-\beta}, \quad \text{for all }(t,\xi, h) \in \R\times \R \times (0,1) .$$
Note that for readability, in this section, we write $\xi=\xi_t$ for the dual variable to the time variable $t$. 
We then quantify (using the Weyl quantization) such a symbol as 
\begin{align}
	\label{e:weyl-quantization}
	\left(   \op^w (a) u  \right)(t)& : = \frac{1}{2\pi}\int_{\R\times\R} e^{i(t-s)\xi}a\left(\frac{t+s}{2}, \xi\right)u(s)ds d\xi .
\end{align}
	According to~\cite[Paragraph~18.1 Remark~2 p~117]{Hoermander:V3}, 
	\begin{itemize}
		\item 
		for all $a \in S^m(\R\times \R; \mathcal{L}(\X,\Y))$,  $\op^w (a)$ maps continuously $\mathcal{S}(\R;\X)$ into $\mathcal{S}(\R;\Y)$ uniformly in $h \in (0,1)$;
		\item 
		for all $a \in S^0(\R\times \R; \mathcal{L}(\X,\Y))$,  $\op^w (a)$ maps continuously $L^2(\R;\X)$ into $L^2(\R;\Y)$ uniformly in $h \in (0,1)$.
	\end{itemize}
	If $a \in S^0(\R\times \R; \mathcal{L}(\X,\Y))$ has compact support in $\R\times\R$ (with support possibly depending on the parameter $h \in (0,1)$), then 
	$$
	(\op^w(a)u)(t)=\int_{\R} \mathcal{K}(t,s)u(s)ds ,\quad \mathcal{K}(t,s) = \frac{1}{2\pi}\int_{\R} e^{i(t-s)\xi}a\left(\frac{t+s}{2}, \xi\right) d\xi  ,
	$$
	where  the Schwartz kernel $\mathcal{K}$ of the operator $\op^w(a)$ satisfies $\mathcal{K} \in C^\infty(\R \times \R; \mathcal{L}(\X,\Y))$. 
	Note that such functions $a$ do not necessarily belong to $S^{-\delta}$ for some $\delta>0$ (since the support may depend on $h$).
	
\begin{remark}
\label{r:lot-bis}
	Note that in the application we have in mind, for a domain $V \subset \R^d$, we choose $\X = \Y = L^2(V)$ and $\ban = L^\infty(V)$
	and observe the embedding  $L^\infty(V) = \ban \hookrightarrow \L(\X,\Y) = \L(L^2(V))$ (via the application that maps to a bounded function $f$ the multiplication operator by $f$) with $\|\cdot \|_{ \L(\X,\Y)} \leq \|\cdot\|_{\ban}$.
	
	Another application is $\X = H^1(V),  \Y = L^2(V)$ and $\ban = L^d(V)$ if $d\geq 3$ (resp.  $\ban = L^{2+\eps}(V)$ for all $\eps>0$ if $d=2$) 
	and observe the embedding  $L^d(V) = \ban \hookrightarrow \L(\X,\Y) = \L(H^1(V), L^2(V))$ (a function $\pot$ acting by multiplication) according to the Sobolev embedding: $\|\pot u\|_{L^2(V)} \leq \|\pot\|_{L^d(V)} \|u\|_{L^{\frac{2d}{d-2}}(V)} \leq  \|\pot\|_{L^d(V)} \|u\|_{H^1(V)}$ if $d\geq 3$ (resp. $\|\pot u\|_{L^2(V)} \leq \|\pot\|_{L^{2+\eps}(V)} \|u\|_{H^1(V)}$ for all $\eps>0$ if $d=2$). 
	\end{remark}
	
	\subsection{The conjugated operator}
	\label{s:conj-operator}
	In this section we define for $t_0 \in \R$ and $r_0>0$ the open intervals $I:=(t_0-2r_0,t_0+2r_0)$ and $U:=(t_0-r_0,t_0+r_0)$. Given now $f \in \mathcal{G}^{s}(I;\L(\X,\Y))$ there exists $R>0$ such that $f \in \mathcal{G}^{s,R}_b(U;\L(\X,\Y))$. The intervals $I,U$ and the radius $R$, used in definition \eqref{e:gevrey-norm} will be fixed for the rest of this section. For $\rho>0$ we denote by $\tf(z) $ the almost analytic extension of $f$ in $U+i\R$ given by Lemma~\ref{extension of gevrey} which is supported on $U_\rho:=U+i[-\rho,\rho]$.
	
	Along this section, we will need some cut-off functions satisfying the following properties: $\chi^0 \in C^\infty_c((-4,4);[0,1])$ with $\chi=1$ in a neighborhood of $[-3,3]$, $\theta^0 \in C^\infty_c((-1,1);[0,1])$
	and  $\eta^0 \in C^\infty_c((-3,3);[0,1])$ with $\eta=1$ in a neighborhood of $[-2,2].$
	
	Take now $r$ with $0<r< \min(\frac{r_0}{4},\frac{\rho}{3})$. We will define $\chi(t)=\chi^0((t-t_0)/r)$, $\theta(t)=\theta^0((t-t_0)/r)$ and $\eta(\xi)=\eta^0(\xi/r)$. In particular, they satisfy
	\begin{itemize}
		\item $\chi \in C^\infty_c((t_0-4r,t_0+4r);[0,1])$ with $\chi=1$ in a neighborhood of $[t_0-3r,t_0+3r]$
		\item $\theta \in C^\infty_c((t_0-r,t_0+r);[0,1])$
		\item $\eta \in C^\infty_c((-3r,3r);[0,1])$ with $\eta=1$ in a neighborhood of $[-2r,2r].$
	\end{itemize}
	The functions $\chi$, $\theta$ and $\eta$ depend implicitly on $r$ and $t_0$, but we will not write anymore this dependence for better readability. 
	
	With $h \in (0,1)$, we set
	\begin{align}
		\tf^r(z)& :=\chi(\Re z) \eta(h^{-1/3}\Im z)\tf(z), \quad z \in \C \quad \text{ and hence } \label{e:def-tf-r} \\
		\tf^r(t+ih\xi) & = \chi(t) \eta(h^{2/3}\xi) \tf(t+ih\xi) , \quad (t,\xi) \in \R \times \R. \nonumber
	\end{align}
	Observe that the function $(t,\xi) \mapsto \tf^r(t+ih\xi)$ is smooth, compactly supported in $\R\times \R$, and belongs to $S^0(\R\times \R; \mathcal{L}(\X,\Y))$. According to the above discussion, we define the operator
	\begin{equation}
		\label{def of Fh}
		F_h:=\textnormal{op}^w(\tf^r(t+ih\xi)).
	\end{equation}
	It maps continuously $\mathcal{S}(\R;\X)$ into $\mathcal{S}(\R;\Y)$ uniformly in $h \in (0,1)$ and  
	\begin{equation}	
	\label{e:bornitude-L2-Fh}
	F_h \in \mathcal{L} \left(L^2(\R;\X) ; L^2(\R;\Y) \right) , \quad \text{ uniformly in } h \in(0,1) .
	\end{equation}
	We are now ready to state the following result, which guarantees that we have a reasonable conjugate for the operator $e^{-\frac{h}{2}|D_t|^2}f$.
	
	\begin{proposition}
		\label{good conjugagte with function}
		Let $\rho,r_0>0$ and $0<r< \min(\frac{r_0}{4},\frac{\rho}{3})$. Then there exists $c>0$ such that for all $R>0$ and all $k \in \N$ there exist $C_k >0$ and $h_0>0$ such that for all 
		 $f \in \mathcal{G}^{2,R}_b(U;\L(\X,\Y))$ and $u\in \mathcal{S}(\R;\X)$ one has
		$$
		\nor{\chi F_h e^{-\frac{h}{2}|D_t|^2} \theta u-e^{-\frac{h}{2}|D_t|^2} f \theta u }{L^2(\R; \Y)}\leq  C_{k} h^{-k}\left(\sum_{j \leq k}\nor{f^{(j)}}{2,R,U}\right)e^{-\frac{ch^{-1/3}}{R}}\nor{u}{H^{-k}(\R;\X)},
		$$
		for all $0<h\leq h_0$, where $F_h$ is defined by~\eqref{def of Fh}.
	\end{proposition}
	We refer to Remark~\ref{r:rem-k} for the interest of the index $k$. Here again, the proof of Proposition~\ref{good conjugagte with function} is simpler for $k=0$.
		Note also that  for all $R>0,j \in \N$ and $\eps>0$, there is a constant $C>0$ such that $(R-\eps)^m (m+j)^2 \leq C R^m$ for all $m \in \N$, whence 
		$$
		 \sum_{j \leq k}\nor{f^{(j)}}{2,R-\eps,U}  \leq C(R,k,\eps)\nor{f}{2,R,U} , \quad \text{ for all }f \in \mathcal{G}^{2,R}_b(U;\L(\X,\Y)).
		$$
		As a consequence, the result of the lemma reformulates in a simpler way as
		\begin{align*}
		\nor{\chi F_h e^{-\frac{h}{2}|D_t|^2} \theta u-e^{-\frac{h}{2}|D_t|^2} f \theta u }{L^2(\R; \Y)}\leq  C_{k,\eps} \nor{f}{2,R,U} e^{-\frac{c}{R-\eps}h^{-1/3}}\nor{u}{H^{-k}(\R;\X)},
		\end{align*}
		for all $h \in (0,h_0)$ where $h_0=h_0(k,\eps)$.

	\begin{remark}
		Taking $h=\mu/\tau^3$ one sees that 
		$$
		e^{-\frac{\mu}{2\tau^3}|D_t|^2} \theta f  =F_{\frac{\mu}{\tau^3}} e^{-\frac{\mu}{2\tau^3}|D_t|^2} \theta  , \qquad F_{\frac{\mu}{\tau^3}} = \op^w \left( \chi(t) \eta(\frac{\mu^{2/3}}{\tau^2}\xi) \tf(t+i\frac{\mu}{\tau^3}\xi) \right) ,
		$$
modulo an exponentially small error of order $e^{-c \tau}$ (in well-adapted norms), which is an admissible error in the Carleman estimate~\eqref{Carlemanschrod} (in view of its application ot unique continuation in Section~\ref{section the uniqueness theorem}).
 Notice that with this scaling, the cut-off $\eta$ localizes in frequencies $|\xi_t| \lesssim \tau^2$. This is consistent with the sketch of proof in Section~\ref{s:plan-structure}.
	\end{remark}
	
	\begin{remark}
		Proposition~\ref{good conjugagte with function} provides with a substitute of Lemma~\ref{lmcommuteps} in the case where $f(t)=t$ is replaced by an arbitrary Gevrey $2$ function.
	\end{remark}

	\begin{lemma}
		\label{l:identify-kernel}
		Setting 
		\begin{align}
			\label{e:def-Rh}
			R_h := \chi F_h e^{-\frac{h}{2}|D_t|^2} \theta  -\chi e^{- \frac{h}{2}|D_t|^2} f \theta  \quad \in \L\big(L^2(\R,\X) , L^2(\R,\Y)\big)  , 
		\end{align} 
		we have
		\begin{align}
			\label{difference noyau}
			(R_h u )(t)& =\int_\R \mathcal{K}_h(t,s)u(s)ds , \quad u \in \mathcal{S}(\R,\X) \text{ with }  \\
			\label{e:K-K1+K2}
			\mathcal{K}_h(t,s)& =-\frac{1}{2 \pi}\kerone+ C_h \kertwo(t,s) \quad C_h:= \frac{1}{2\pi}    \left(\frac{1}{2\pi h}\right)^{1/2}, \quad  \text{ and }  \\
			\label{def of I1}
			\kerone(t,s)& :=\chi(t)  \theta(s) f(s)  \int_\R e^{-i(s-t)\xi}(1-\eta(h^{2/3}\xi))e^{-\frac{h|\xi|^2}{2}}   d\xi , \\
			\label{def of Ixy}
			\kertwo(t,s)& :=\chi(t) \theta(s) \int_{\R\times\R} \left(\tf^r \left(\frac{t+w}{2}+i h \xi\right)-\eta(h^{2/3}\xi)f(s)\right) e^{i(t-w)\xi}e^{-\frac{|w-s|^2}{2h}}dw d \xi.
		\end{align}
	\end{lemma}
	
	\bnp[Proof of Lemma~\ref{l:identify-kernel}]
	Recalling the definition of the Weyl quantization in~\eqref{e:weyl-quantization} and that of $F_h$ in~\eqref{def of Fh}, we have
	\begin{equation*}
		(F_hu)(t)=\frac{1}{2\pi} \int_{\R\times \R} e^{i(t-w)\xi}\tf^r\left(\frac{t+w}{2}+i h \xi\right)u(w)dw d\xi.
	\end{equation*}
	Combined with formula~\eqref{e:gaussian-convol}, this implies
	\begin{align}
		\label{kernel first term}
		(\chi F_h e^{-\frac{h}{2}|D_t|^2} \theta u)(t)
		&= \frac{1}{2\pi} \cdot \left(\frac{1}{2\pi h}\right)^{1/2} \chi(t) \int_{\R \times \R \times \R}  e^{i(t-w)\xi}\tf^r\left(\frac{t+w}{2}+i h \xi\right) \theta(s) u(s)e^{-\frac{|w-s|^2}{2h}}dwd\xi ds.
	\end{align}
	Using again formula~\eqref{e:gaussian-convol} as well as the formula for the Fourier transform of a Gaussian~\eqref{fourier of a gaussian} we find
	\begin{align*}
		(\chi e^{-\frac{h}{2}|D_t|^2} f \theta u)(t) &= \left(\frac{1}{2\pi h}\right)^{1/2} \chi(t)\int_{\R} f(s) \theta(s) u(s)e^{-\frac{|t-s|^2}{2h}}ds\\
		&=\left(\frac{1}{2\pi h}\right)^{1/2} \chi(t) \int_{\R} f(s) \theta(s) u(s)\left(\frac{1}{2\pi} (2 \pi h)^{1/2} \int_{\R} e^{-i(s-t)\xi}e^{-\frac{h|\xi|^2}{2}}   d\xi \right)  ds\\
		&=\left(\frac{1}{2\pi h}\right)^{1/2} \chi(t) \int_{\R} f(s) \theta(s) u(s)\left(\frac{1}{2\pi} (2 \pi h)^{1/2} \int_{\R} e^{-i(s-t)\xi}\eta(h^{2/3} \xi)e^{-\frac{h|\xi|^2}{2}}   d\xi \right)  ds\\
		&\hspace{4mm}+\left(\frac{1}{2\pi h}\right)^{1/2} \chi(t) \int_{\R} f(s) \theta(s) u(s)\left(\frac{1}{2\pi} (2 \pi h)^{1/2} \int_{\R} e^{-i(s-t)\xi}(1-\eta(h^{2/3}\xi))e^{-\frac{h|\xi|^2}{2}}   d\xi \right)  ds.
	\end{align*}
	We now use once more~\eqref{fourier of a gaussian} in order to replace $e^{-\frac{h |\xi|^2}{2}}$ by $(\frac{1}{2 \pi h})^{1/2}\int_{\R} e^{-i w \xi} e^{-\frac{|w|^2}{2h}}dw$ in the first term of the sum above. We find then:
	\begin{align}
		\label{kernel second term}
		(\chi e^{-\frac{h}{2}|D_t|^2} f \theta u)(t)  &=\frac{1}{2 \pi}\left(\frac{1}{2\pi h}\right)^{1/2} \chi(t) \int_{\R} f(s) \theta(s) u(s)\left( \int_{\R} e^{-i(s-t)\xi}\eta(h^{2/3} \xi) \int_{\R} e^{-i w \xi} e^{-\frac{|w|^2}{2h}}dw  d\xi \right)  ds \nonumber \\
		&\hspace{4mm} + \frac{1}{2 \pi}\chi(t) \int_{\R} f(s) \theta(s) u(s)\left( \int_{\R} e^{-i(s-t)\xi}(1-\eta(h^{2/3}\xi))e^{-\frac{h|\xi|^2}{2}}   d\xi \right)  ds.
	\end{align}
	We finally perform the change of variable $w\to w-s$ in the integral with respect to $w$ to express the first term in~\eqref{kernel second term} in the following way:
	\begin{align}
		\label{alternative expr for kernel}
		&\frac{1}{2 \pi}\left(\frac{1}{2\pi h}\right)^{1/2} \chi(t) \int_{\R} f(s) \theta(s) u(s)\left( \int_{\R} e^{-i(s-t)\xi}\eta(h^{2/3} \xi) \int_{\R} e^{-i w \xi} e^{-\frac{|w|^2}{2h}}dw  d\xi \right)  ds \nonumber \\
		&\hspace{4mm}=\frac{1}{2 \pi}\left(\frac{1}{2\pi h}\right)^{1/2} \chi(t) \int_{\R \times \R \times \R} e^{i(t-w)\xi}\eta(h^{2/3} \xi)f(s)\theta(s)u(s)e^{-\frac{|w-s|^2}{2h}}dw d\xi ds.
	\end{align}
	The result is then a consequence of~\eqref{kernel first term}, \eqref{kernel second term} and~\eqref{alternative expr for kernel}.
	\enp

	The key step for the proof of Proposition~\ref{good conjugagte with function} consists in controlling the terms $\mathcal{K}_{j,h}$ in~\eqref{def of I1}--\eqref{def of Ixy}. For later applications, we consider a slightly more general family of kernels (useful when) defined for functions $\chi_1, \theta_1 \in C^\infty_c(\R)$ and $f \in \mathcal{G}_b^{2,R}(\R;\L(\X,\Y))$ and $m\in \N$, by 
		\begin{align*}
			\mathcal{I}_{1,h}(t,s)&:=\chi_1(t)  \theta_1(s) f(s)  \int_\R e^{-i(s-t)\xi}(1-\eta(h^{2/3}\xi))e^{-\frac{h|\xi|^2}{2}} \xi^m   d\xi,\\
			\mathcal{I}_{2,h}(t,s)&:=\chi_1(t) \theta_1(s) \int_{\R\times\R} \left(\chi\left(\frac{t+w}{2}\right)\eta(h^{2/3}\xi)\tf \left(\frac{t+w}{2}+i h \xi\right)-\eta(h^{2/3}\xi)f(s)\right) \nonumber \\  & \hspace{25mm}\cdot \left(\frac{t+w}{2}+i h \xi-s\right)^m e^{i(t-w)\xi}e^{-\frac{|w-s|^2}{2h}}dw d \xi .
		\end{align*}
		Later in the proofs, we shall write $\mathcal{I}_{2,h}(t,s)= \mathcal{I}_{2,h}[\chi_1,\theta_1,f,m](t,s)$ to stress the dependence on the functions and parameters involved in the definition of $\mathcal{I}_{2,h}$.
Note that $\mathcal{K}_{2,h}=\mathcal{I}_{2,h}[\chi,\theta,f,0]$, where $\chi,\theta$ are defined (once and for all) at the beginning of Section~\ref{s:conj-operator}.

	\begin{lemma}
		\label{estimate for the kernel I}
		Let $\rho,r>0$ as in~Proposition~\ref{good conjugagte with function} and $\chi,\theta$ defined accordingly at the beginning of Section~\ref{s:conj-operator}.
		Then, for any $m\in \N$, any $\chi_1 \in C^\infty_c(\R)$ with $\supp(\chi_1)\subset \supp(\chi)$ and $\supp(\chi_1')\subset \supp(\chi')$, for any $\theta_1\in C^\infty_c(\R)$ with $\supp(\theta_1)\subset \supp(\theta)$, there exist $C,c, h_0>0$ such that for all $f \in \mathcal{G}_b^{2,R}(U;\L(\X,\Y))$, 
		$$
		\nor{\mathcal{I}_{j,h}}{L^\infty(\R\times \R ; \mathcal{L}(\X;\Y))}\leq C \nor{f}{2,R,U} e^{-\frac{ch^{-1/3}}{R}}, \quad \text{ for all }h \in(0,h_0).
		$$
	\end{lemma}
	Note that this lemma will be only used with $\chi_1 = \chi^{(k)}$ and $\theta_1 = \theta^{(k)}$ for some $k\in \N$, which satisfy the support assumptions.
	\begin{proof}[Proof of Lemma~\ref{estimate for the kernel I}]
			We start with the proof for $j=1$ i.e. study $\mathcal{I}_{1,h}$. We remark that in the support of $1-\eta(h^{2/3}\xi)$ one has $h^{2/3} |\xi| \geq 2r $ which implies that $h|\xi|^2 \geq ch^{-1/3}$ in the support of $1-\eta(h^{2/3}\xi)$. We estimate then, for $h\leq h_0$ with $h_0$ sufficiently small:
		\begin{align}
			\label{estimate for I1}
			\nor{\mathcal{I}_{1,h}(t,s)}{\mathcal{L}(\X;\Y)} 
			&\leq \nor{ \mathds{1}_{\supp \theta}f(s)}{\mathcal{L}(\X;\Y)}  \int_\R \left|(1-\eta(h^{2/3}\xi))e^{-\frac{h|\xi|^2}{2}} \xi^m \right|   d\xi \nonumber \\
			& \leq C \nor{f}{L^\infty(\supp(\theta) ; \mathcal{L}(\X;\Y))} \int_\R \left|(1-\eta(h^{2/3}\xi))e^{-\frac{h|\xi|^2}{4}} e^{-\frac{h|\xi|^2}{4}} \xi^m  \right|   d\xi \nonumber \\
			& \leq Ce^{-ch^{-1/3}} \int_\R  \left| e^{-\frac{h|\xi|^2}{4}} \xi^m \right| d\xi \, \nor{f}{L^\infty(\supp(\theta) ; \mathcal{L}(\X;\Y))}\nonumber  \\
			&   \leq Ce^{-ch^{-1/3}} \nor{f}{L^\infty(\supp(\theta) ; \mathcal{L}(\X;\Y))} \leq  Ce^{-ch^{-1/3}}\nor{f}{2,R,U},
		\end{align}
		where we used the fact that $f$ is Gevrey (and hence continuous) and $\theta$ is compactly supported in $U$.

		\bigskip
		We now turn our attention to $\mathcal{I}_{2,h}(t,s)$. In the definition of $\mathcal{I}_{2,h}(t,s)$ we change variable by writing $(\w,\xi)\in \R^2 \to z\in \C$ with 
		\begin{align}
			\label{e:change-variable}
			z=\frac{t+w}{2}+i h \xi, \quad \text{whence}\quad w=2\Re(z)-t, h\xi=\Im(z) , \quad \text{and} \quad  dw\wedge d\xi = \frac{i}{h} dz\wedge d\zb .
		\end{align}
		The factor $e^{i(t-w)\xi}e^{-\frac{|w-s|^2}{2h}}$ rewrites as $e^{i(t-w)\xi}e^{-\frac{|w-s|^2}{2h}} =e^{\frac1h \Phi(t,s,z)}$, with 
		\begin{align}
			\Phi(t,s,z) & =i(t-w)h\xi-\frac{(w-s)^2}{2} = 2i(t-\Re(z))\Im(z) -\frac{(2\Re(z)-t-s)^2}{2}  			\label{e:phase}\\
			& = 2it\Im(z)-2i\Re(z)\Im(z) - 2 \Re(z)^2 - \frac{t^2+s^2}{2} + 2t\Re(z) +2s\Re(z) - ts \nonumber \\
			& = 2 tz + s(z+\zb-t)-(z+\zb)z- \frac{t^2+s^2}{2} \nonumber \\
			& = -\frac{(t-s)^2}{2}+ (z-s)(2t-z-\zb).			\label{calcul de phase cht de variable}
		\end{align}
		Then, we can write $\mathcal{I}_{2,h}$ as 
		\begin{align}
			\label{alternative form for kernel}
			\mathcal{I}_{2,h}(t,s)&=\frac{i}{h} \chi_1(t) \theta(s) \int \eta(h^{-1/3} \Im z) \left(\chi(\Re z)\tf(z)-f(s) \right)(z-s)^m e^{-\frac{|t-s|^2}{2h}}e^{\frac{1}{h} (z-s)(2t-z-\bar{z})}dz \wedge d \zb .
		\end{align}
		Defining  
		\begin{equation}
			\label{def of b check}
			\check{b}_s(z)= \theta(s)\frac{\chi(\Re z) \tf(z)-f(s)}{z-s},
		\end{equation}
		we may rewrite 
		\begin{align}
			\label{ipp minus a ball}
			\mathcal{I}_{2,h}(t,s)& =-i \chi_1(t) \int_{\C} (z-s)^m \eta(h^{-1/3} \Im z) \bc(z) \dzb\left(e^{-\frac{|t-s|^2}{2h}}e^{\frac{1}{h} (z-s)(2t-z-\bar{z})}\right) dz \wedge d \zb.
		\end{align}
		We will now check that we are in position to integrate by parts using Lemma~\ref{l:ipp with zbar at infinity}. 
		
		First, we prove that $\check{b}_s\in C^1(\C)$. It is smooth away from $s$, so we only need to check the regularity close to $z=s$. We decompose $\check{b}_s(z)= \theta(s)\chi(\Re z) \frac{\tf(z)-f(s)}{z-s}-\theta(s)(1-\chi)(\Re z)\frac{ f(s)}{z-s}$. The first term is $C^1(\C)$ thanks to Lemma \ref{l:regC1} applied to $\tf(\cdot-s)$. For the second term, we observe that for $s \in (t_0-r,t_0+r)$ in the support of $\theta$ and for $\Re(z)\notin (t_0-3r,t_0+3r)$ in the support of $1-\chi$, we have $|z-s|\geq |\Re(z)-s|\geq 2r$. This gives the regularity of the second term.

		According to~\eqref{e:phase} and  $(2\Re(z)-t-s)^2\geq \Re(z)^2-C_{t,s}$ for some $C_{t,s}>0$, we have
		\begin{align}
			\label{estimexp}
			\left|e^{-\frac{|t-s|^2}{2h}}e^{\frac{1}{h} (z-s)(2t-z-\bar{z})}\right|\leq e^{\frac{C_{t,s}}{h}}e^{-\frac{\Re(z)^2}{2h}},
		\end{align}
		as well as
		\begin{align}
			\label{estimdbarexp}
			\left|\dzb\left(e^{-\frac{|t-s|^2}{2h}}e^{\frac{1}{h} (z-s)(2t-z-\bar{z})}\right)\right|=|z-s|\left|e^{-\frac{|t-s|^2}{2h}}e^{\frac{1}{h} (z-s)(2t-z-\bar{z})}\right|\leq e^{\frac{C_{t,s}}{h}}\left(|\Im(z)|+|\Re(z)-s|\right)e^{-\frac{\Re(z)^2}{2h}}.
		\end{align}
		Since $\eta$ localizes the imaginary part in a compact set and now \eqref{estimexp} and \eqref{estimdbarexp} are obtained, we are left to prove $L^{\infty}$ estimates on $\check{b}_s(z)$ and $ \dzb \bc$.
		
		We have
		$$
		\nor{\check{b}_s(z)}{\mathcal{L}(\X;\Y)}\leq \nor{\tf}{W^{1,\infty}(U_\rho; \mathcal{L}(\X;\Y))}, \quad \text{ for } \Re (z) \in (t_0-3r,t_0+3r) ,
		$$
		since $\chi(\Re z)=1$ for such $z$. For $\Re(z) \notin (t_0-3r,t_0+3r)$ and $s \in \supp \theta$, we have
		$
		|z-s|\geq 2r, 
		$
		which implies 
		$$
		\nor{\check{b}_s(z)}{\mathcal{L}(\X;\Y)}\leq C \nor{\tf}{L^\infty(U_\rho; \mathcal{L}(\X;\Y))}, \quad \text{ for } \Re(z) \notin (t_0-3r,t_0+3r),
		$$
		with a constant $C$ depending only on $r$. Putting the two estimates above together we obtain that $\bc \in C^0_b(\C)$ and there is $C=C(r)>0$ such that 
		\begin{equation}
			\label{bound for b check}
			\nor{\check{b}_s(z)}{\mathcal{L}(\X;\Y)}\leq C \nor{\tf}{W^{1,\infty}(U_\rho; \mathcal{L}(\X;\Y))} , \quad z \in \C .
		\end{equation}
		Secondly, we compute
		\begin{align}
			\label{derivative of bchech}
			\dzb \bc (z)= \theta(s)\frac{\chi^\prime(\Re z)}{2(z-s)}\tf(z)+\theta(s) \frac{\chi(\Re z)\dzb \tf(z)}{z-s}, 
		\end{align}
		and notice that the first term is smooth and bounded given the relative support properties of $\theta$ and $\chi'$. For the second term, using~\eqref{estimate for dzb} for Gevrey 2 functions and the fact that $s\in\R$, we obtain, for $z \in U_\rho$ (the value of the constant $C$ may change from one line to another):
		\begin{align}
			\label{dzb estimate for bcech}
			\nor{\frac{\chi(\Re z)\dzb \tf(z)}{z-s}}{\mathcal{L}(\X;\Y)} &\leq \frac{1}{|z-s|} C \nor{f}{2,R,U}\exp{\left(-\frac{1}{C_0 R |\operatorname{Im} z| }\right)} \nonumber\\
			&\leq \frac{1}{|\Im z|} C \nor{f}{2,R,U}\exp{\left(-\frac{1}{C_0R |\operatorname{Im} z|}\right)} \nonumber \\
			&\leq  C \nor{f}{2,R,U}\exp{\left(-\frac{1}{2C_0 R |\operatorname{Im} z|}\right)}.
		\end{align}
		Combining the previous estimate and \eqref{derivative of bchech}, we get
		$$ 
		\nor{\dzb \bc(z)}{\mathcal{L}(\X;\Y)}\leq C\nor{\tf}{L^{\infty}(U_\rho; \mathcal{L}(\X;\Y))}+ C\nor{f}{2,R,U}, \quad z \in \C .
		$$
		As announced before, the $L^{\infty}$ bounds on $\dzb \bc$ and $\bc$, combined with the localization of $\eta$, \eqref{estimexp} and \eqref{estimdbarexp} give the integrability of all the terms involved in the integration by parts. All assumptions of  Lemma~\ref{ipp with zbar at infinity} are therefore satisfied and we may now integrate by parts in~\eqref{ipp minus a ball}, yielding
		\begin{align}
			\label{def of I2(t,s)}
			\mathcal{I}_{2,h}(t,s) &=i\chi_1(t)  \int_\C \dzb\left((z-s)^m \eta(h^{-1/3} \Im z) \bc(z) \right)e^{-\frac{|t-s|^2}{2h}}e^{\frac{1}{h} (z-s)(2t-z-\bar{z})}dz \wedge d \zb.
		\end{align}
		
		Recalling~\eqref{derivative of bchech}, we now decompose~\eqref{def of I2(t,s)} as
		\begin{align}
			\mathcal{I}_{2,h} & =     \mathcal{I}_{21,h}+    \mathcal{I}_{22,h}+    \mathcal{I}_{23,h} , \quad \text{ with }\nonumber  \\
			\label{def of I21}
			\mathcal{I}_{21,h}(t,s)&  : = i\chi_1(t)  \theta(s)\int_\C (z-s)^m\eta(h^{-1/3} \Im z) 
			\frac{\chi^\prime(\Re z)}{2(z-s)}\tf(z)  e^{-\frac{|t-s|^2}{2h}}e^{\frac{1}{h} (z-s)(2t-z-\bar{z})}dz \wedge d \zb , \\
			\label{def of I22}
			\mathcal{I}_{22,h}(t,s)&  : = i\chi_1(t)  \theta(s)\int_\C (z-s)^m \eta(h^{-1/3} \Im z) 
			\frac{\chi(\Re z)\dzb \tf(z)}{z-s}  e^{-\frac{|t-s|^2}{2h}}e^{\frac{1}{h} (z-s)(2t-z-\bar{z})}dz \wedge d \zb \\
			\label{def of I23}
			\mathcal{I}_{23,h}(t,s)&  : =   -\frac12 h^{-1/3}\chi_1(t)  \int_\C (z-s)^m \eta'(h^{-1/3} \Im z) \bc(z) e^{-\frac{|t-s|^2}{2h}}e^{\frac{1}{h} (z-s)(2t-z-\bar{z})}dz \wedge d \zb.
		\end{align}
		We now estimate each term separately.
		We start with $\mathcal{I}_{21,h}$ and rewrite the integral in the original variables~\eqref{e:change-variable} as
		\begin{multline*}
			\mathcal{I}_{21,h}(t,s)
			\\=i h \chi_1(t) \theta(s)\int_{\R\times\R}\left(\frac{t+w}{2}-s+ih\xi \right)^m \eta(h^{2/3} \xi)\chi^\prime\left(\frac{t+w}{2}\right)\frac{\tf\big(\frac{t+w}{2}+i h \xi\big)}{2\big(\frac{t+w}{2}+i h \xi-s\big)} e^{i(t-w)\xi}e^{-\frac{|w-s|^2}{2h}}dw d \xi.
		\end{multline*}
		Observe now that $\supp(\chi')  \subset (t_0-4r,t_0-3r) \cup (t_0+3r,t_0+4r).$  
				Therefore the integrand above is supported in $ |\frac{t+w}{2}-t_0|\geq 3r $ (thanks to the support of $\chi^\prime$) and $|t-t_0|<4r$ (thanks to the support of $\chi$). This implies that $ |w -t_0|\geq 2r$ for otherwise one would have
		$$
		\left|\frac{t+w}{2}-t_0\right|\leq \left|\frac{t-t_0}{2}\right|+\left |\frac{w-t_0}{2}\right |<2r+r=3r.
		$$
		Since in the support of $\theta$ we have $|s-t_0|<r$ we find finally that $|w-s|\geq r$ in the support of the integral. Notice finally that, if $\chi^\prime\left(\frac{t+w}{2}\right)\neq 0$ and $\theta(s)\neq 0$ one has
		$$
		\left| \frac{t+w}{2}+i h \xi-s \right|\geq  \left| \frac{t+w}{2}-s \right|\geq \left| \frac{t+w}{2}-t_0 \right|-\left| t_0-s \right|\geq 2r
		$$
		and thanks to the supports of $\chi$, $\theta$ and $\eta$ have for a constant $C>0$ depending on $m$ and $r$ that
		$$
		\left| \frac{t+w}{2}+i h \xi-s \right|^m\leq C.
		$$
		We can then estimate as follows:
		\begin{align*}
			\nor{  \mathcal{I}_{21,h}}{\mathcal{L}(\X; \Y)} &\leq  \frac{C h}{4r} \nor{\tf}{L^\infty(U_\rho;\mathcal{L}(\X; \Y))} \int_{\R \times \R} |\eta(h^{2/3}\xi) |  e^{-\frac{|w-s|^2}{4h}} e^{\frac{r^2}{4h}} dw d \xi \\
			&\leq \frac{C h}{4r} \nor{\tf}{L^\infty(U_\rho;\mathcal{L}(\X; \Y))}e^{-\frac{r^2}{4}h}\int_{[-3rh^{-2/3},3rh^{-2/3}]} d\xi \int_{\R} e^{-\frac{|w-s|^2}{4h}} dw.
		\end{align*}
		This implies the stronger bound
		\begin{equation}
			\label{bound for I21}
			\nor{\mathcal{I}_{21,h}(t,s)}{\mathcal{L}(\X;\Y)}\leq C \nor{\tf}{L^\infty(U_\rho;\mathcal{L}(\X; \Y))}e^{-ch^{-1}} \leq C e^{-ch^{-1}}  \nor{f}{2,R,U},
		\end{equation}
		where the last inequality follows from~\eqref{e:continuous-extension}.
		
		\medskip
		
		We now study the integral $\mathcal{I}_{22,h}$ defined in~\eqref{def of I22}. Recall that $\supp\eta \subset [-3r,3r]$, so that the domain of integration is contained in $|\Im z|\leq 3r h^{1/3}$. Using \eqref{dzb estimate for bcech}, we can then estimate the corresponding integral as follows:
		\begin{align}
			\label{bound for I22}
			\nor{ \mathcal{I}_{22,h}(t,s)}{\mathcal{L}(\X;\Y)} &\leq C \int_\C \nor{ \eta(h^{-1/3} \Im z) \frac{\chi(\Re z)\dzb \tf(z)}{z-s}   e^{-\frac{|t-s|^2}{2h}}e^{\frac{1}{h} (z-s)(2t-z-\bar{z})} }{\mathcal{L}(\X;\Y)} \left| dz \wedge d \zb \right|  \nonumber \\
			&\leq C  \nor{f}{2,R,U}\exp{\left(-\frac{h^{-1/3}}{6 rC_0 R }\right)}\int_{ K_\rho'}  \left |   e^{-\frac{|t-s|^2}{2h}}e^{\frac{1}{h} (z-s)(2t-z-\bar{z})} \right | \left| dz \wedge d \zb \right|  \nonumber\\
			&\leq C \nor{f}{2,R,U}   e^{-\frac{ch^{-1/3}}{R}}.
		\end{align}
		In this last inequality of~\eqref{bound for I22}, we used the fact that
		\begin{align*}
			\int_{ K_\rho'}  \left |   e^{-\frac{|t-s|^2}{2h}}e^{\frac{1}{h} (z-s)(2t-z-\bar{z})} \right | |dz \wedge d \zb|
			\leq \int_{ K_\rho'}|dz \wedge d \zb| \leq C.
		\end{align*}
		which follows from~\eqref{e:phase}.

		The last term we need to control is the integral $\mathcal{I}_{23,h}$ in~\eqref{def of I23}. In the original coordinates~\eqref{e:change-variable}, we have
		\begin{align*}
			\mathcal{I}_{23,h}(t,s)= i\frac{h^{2/3}}{2}\chi_1(t)   \int_{\R\times \R} \left(\frac{t+w}{2}-s+ih\xi \right)^m\eta^\prime(h^{2/3}\xi)\bc\left(\frac{t+w}{2}+ih\xi\right)e^{i(t-w)\xi}e^{-\frac{|w-s|^2}{2h}} dw d \xi.
		\end{align*}
		We look at the integral in $w$ and treat $\xi$ as a parameter satisfying $2r h^{-2/3}\leq |\xi| \leq 3r h^{-2/3}$ thanks to the support of $\eta^\prime$. The change of variable $w \to w+s$ allows to rewrite the integral as follows:
		\begin{align}
			\label{e:use-g}
			& \int_{\R}  \bc\left(\frac{t+w}{2}+ih\xi\right)e^{i(t-w)\xi}e^{-\frac{|w-s|^2}{2h}} dw  
			=e^{-i(s-t) \xi}\int_{\R}g_{h\xi,t,s}(w)e^{-iw\xi}e^{-\frac{|w|^2}{2h}}dw, \\
			\label{def of g}
			&\text{with } \quad 
			g_{\widetilde{\xi},t,s}(w):= \bc\left(\frac{t+s+w}{2}+i\widetilde{\xi}\right)\left(\frac{t+w-s}{2}+i\widetilde{\xi}\right)^m. 
		\end{align}
		Using~\eqref{e:use-g}, we obtain 
		\begin{align*}
			\nor{ \mathcal{I}_{23,h}(t,s)}{\mathcal{L}(\X;\Y)} 
			&=\nor{\frac{h^{2/3}}{2}\chi_1(t)   \int_{\R\times\R} \left(\frac{t+w}{2}-s+ih\xi \right)^m \eta^\prime(h^{2/3}\xi)\bc\left(\frac{t+w}{2}+ih\xi\right)e^{i(t-w)\xi}e^{-\frac{|w-s|^2}{2h}} dw d \xi}{\mathcal{L}(\X;\Y)} \nonumber \\
			&= \frac12 \chi_1(t) h^{2/3}  \nor{ \int_{\R} \eta^\prime(h^{2/3}\xi)e^{-i(s-t) \xi} \left(\int_{\R} g_{h \xi,t,s} (w) e^{-i w\xi}e^{-\frac{|w|^2}{2h}} dw \right)d\xi}{\mathcal{L}(\X;\Y)} \nonumber \\
			&\leq  \frac12 \chi_1(t)h^{2/3}  \int_{\R}\left|\eta^\prime(h^{2/3} \xi)\right| \nor{\int_{\R } g_{h \xi,t,s} (w) e^{-i w\xi}e^{-\frac{|w|^2}{2h}} dw}{\mathcal{L}(\X;\Y)} d \xi .
		\end{align*}
		Recalling that $\supp\chi \subset (t_0-4r,t_0+4r)$ together with the definition of $g_{h \xi,t,s}$ in~\eqref{def of g}, of $\check b_s$ in~\eqref{def of b check}
		and $\supp \chi \subset (t_0-4r,t_0+4r)$, Lemma~\ref{lemma for estimate of integral in w} (below)  now implies
		\begin{align*}
			\chi_1(t)\int_{\R}|\eta^\prime(h^{2/3} \xi)| \nor{\int_{\R } g_{h \xi,t,s} (w) e^{-i w\xi}e^{-\frac{|w|^2}{2h}} dw}{\mathcal{L}(\X;\Y)} d \xi
			\leq C\int_\R |\eta^\prime(h^{2/3}\xi)| d \xi  e^{-\frac{ch^{-1/3}}{R}}  \nor{f}{2,R,U}.
		\end{align*}
		Combining the two estimates above and recalling the support of $\eta$ yields
		\begin{align}
			\label{estimate for I23(t,s)}
			\nor{ \mathcal{I}_{23,h}(t,s)}{\mathcal{L}(\X;\Y)} &\leq h^{2/3}  \int_{-3rh^{-2/3}}^{3rh^{-2/3}} d\xi   e^{-\frac{ch^{-1/3}}{R}}  \nor{f}{2,R,U}\leq C  e^{-\frac{ch^{-1/3}}{R}}  \nor{f}{2,R,U},
		\end{align}
		for $h \leq h_0$.

		Putting together~\eqref{bound for I21}, \eqref{bound for I22} and \eqref{estimate for I23(t,s)} yields for some constants $C$ and $c$ depending only on $I, \rho, r$:
		\begin{equation*}
			\nor{\mathcal{I}_{2,h}(t,s)}{\mathcal{L}(\X;\Y)}\leq C  e^{-\frac{ch^{-1/3}}{R}} \nor{f}{2,R,U},
		\end{equation*}
		which concludes the proof of Lemma~\ref{estimate for the kernel I}.
	\end{proof}

	In the proof of Lemma~\ref{estimate for the kernel I}, we have used the following result.
	\begin{lemma}
		\label{lemma for estimate of integral in w}
		Let $g_{h\xi,t,s}$ be as in~\eqref{def of g} and fix $c_2>c_1>0$. Then there exist $C>0$, $c>0$ and $h_{0}$ depending on $I, \rho, r, c_1,c_2$ such that for  $t \in (t_0-4r,t_0+4r)$, $s\in \R$, $h\in (0,h_{0})$ and $c_1 h^{-2/3}\leq |\xi|\leq c_2 h^{-2/3}$ one has:
		$$
		\nor{\int_{\R}g_{h\xi,t,s}(w)e^{-iw\xi}e^{-\frac{|w|^2}{2h}} dw }{\mathcal{L}(\X;\Y)} \leq C  e^{-\frac{ch^{-1/3}}{R}}  \nor{f}{2,R,U}.
		$$
	\end{lemma}
	\begin{proof}
		First, thanks to the definition of $\check{b}_s$ and the support of $\theta$, we can assume without loss of generality that $s\in (t_0-r,t_0+r)$, for otherwise the integral is zero. 
		We start by separating the integral in two terms:
		$$
		\int_{\R}g_{h\xi,t,s}(w)e^{-iw\xi}e^{-\frac{|w|^2}{2h}} dw=\int_{|w|\geq r}g_{h\xi,t,s}(w)e^{-iw\xi}e^{-\frac{|w|^2}{2h}} dw+\int_{|w|\leq r}g_{h\xi,t,s}(w)e^{-iw\xi}e^{-\frac{|w|^2}{2h}} dw.
		$$
		Observe that since $t,s, h\xi$ lie in a fixed compact set (which depends on $r$) we have that 
		$$
		\left|\frac{t+w-s}{2}+i h\xi\right|^m\leq C (|w|^m+1).
		$$
		For the integral in $|w|\geq r$ we can then proceed as in~\eqref{estimate for I1} to obtain the stronger bound
		\begin{align*}
			\nor{\int_{|w|\geq r}g_{h\xi,t,s}(w)e^{-iw\xi}e^{-\frac{|w|^2}{2h}} dw }{\mathcal{L}(\X;\Y)} &\leq C e^{-ch^{-1}}  \nor{\tf}{W^{1,\infty}(U_\rho; \mathcal{L}(\X;\Y))}
			\int_\R   e^{-\frac{h|w|^2}{4}} (|w|^m+1)  dw \\&\leq C e^{-ch^{-1}}  \nor{\tf}{W^{1,\infty}(U_\rho; \mathcal{L}(\X;\Y))} \leq C e^{-ch^{-1}} \nor{f}{2,R,U},
		\end{align*}
		thanks to \eqref{bound for b check}.

		We now work in the region $|w|\leq r$ and remark that for $t \in (t_0-4r,t_0+4r), s\in (t_0-r,t_0+r)$ and $|w|\leq r$ one has for $z=\frac{t+s+w}{2}+ih\xi $ that
		\begin{align*}
			\left|\Re(z)-t_0\right|\leq \left|\frac{t-t_0}{2}\right|+\left|\frac{s-t_0}{2}\right|+\left|\frac{w}{2}\right|\leq 3r,
		\end{align*}
		and $|\Im(z)|=h |\xi|$. Therefore in this region we have $\chi(\Re z)=1$ and consequently
		$$
		\bc\left(z\right)=\theta(s) \frac{\chi(\Re z)\tf(z)-f(s)}{z-s}=\theta(s) \frac{\tf(z)-f(s)}{z-s}.
		$$
		This implies as in~\eqref{dzb estimate for bcech} that, for $\Im(z)\leq h_0$:
		\begin{equation}
			\label{dzb estimate for bchech bis}
			\nor{\dzb \bc(z)}{\mathcal{L}(\X;\Y)}\leq  C \nor{f}{2,R,U}\exp{\left(-\frac{1}{2C_0 R |\operatorname{Im} z|}\right)}.
		\end{equation}

		To alleviate the notation we write $g$ for $g_{h\xi,t,s }$. We know thanks to~\eqref{def of g} that $g$ admits a complex extension in $[-r,r]+i[-\rho/2,\rho/2]$ for $h\leq h_0$ given by
		$$
		g(w+iv):=\check{b}_s\left(\frac{t+s+w}{2}+ih\xi+\frac{iv}{2}\right)\left(\frac{t+w-s}{2}+ih\xi+\frac{iv}{2}\right)^m,
		$$
		that is 
		$$
		g(z)=\check{b}_s\left(\frac{z}{2}+\frac{t+s}{2}+ih\xi\right)\left(\frac{z}{2}+\frac{t-s}{2}+ih\xi\right)^m,
		$$
		which implies
		\begin{equation}
			\label{dzb for g}
			\dzb g(z)=\frac{1}{2}\dzb \check{b}_s\left(\frac{z}{2}+\frac{t+s}{2}+ih\xi\right)\cdot \left(\frac{z}{2}+\frac{t-s}{2}+ih\xi\right)^m.
		\end{equation}
		Remark that for $|z|\leq r$ and $t,s, \xi$ as in the statement of the lemma we have 
		$$
		\left|\frac{z}{2}+\frac{t-s}{2}+ih\xi\right|^m\leq C,
		$$
		for a constant $C>0$ depending on $r$ and $m$.
		
		We now write the integral we want to control as
		\begin{align*}
			\int_{-r}^{r} g(z)e^{-iz\xi}e^{-\frac{z^2}{2h}}dz= \int_{-r}^{r} g(z)e^{\frac{-h\xi^2}{2}}e^{-\frac{(z+ih\xi)^2}{2h}}dz.
		\end{align*}
		We consider now $\sigma \in (0,\frac12)$ to be chosen later on. We let $\Omega = [-r,r]\times[-\sigma h \xi,0]$ in case $\xi \in [c_1 h^{-2/3} , c_2 h^{-2/3}]$ (see Figure~\ref{change of contour}), resp. $\Omega = [-r,r]\times[0, -\sigma h \xi]$ in case $\xi \in [-c_2 h^{-2/3} ,-c_1 h^{-2/3}]$. Stoke's theorem applies, see~\eqref{stokes for dzb}, and yields:
		\begin{align}
			\label{e:stokes-Gamma}
			\int_{-r}^{r} g(z)e^{\frac{-h\xi^2}{2}}e^{-\frac{(z+ih\xi)^2}{2h}}dz &=\int_{\Gamma_1}g(z)e^{\frac{-h\xi^2}{2}}e^{-\frac{(z+ih\xi)^2}{2h}}dz+\int_{\Gamma_2}g(z)e^{\frac{-h\xi^2}{2}}e^{-\frac{(z+ih\xi)^2}{2h}}dz \nonumber \\&\hspace{4mm}+\int_{\Gamma_3}g(z)e^{\frac{-h\xi^2}{2}}e^{-\frac{(z+ih\xi)^2}{2h}}dz+\int_{\Omega}\dzb(g(z))e^{-iz\xi}e^{-\frac{z^2}{2h}} dz \wedge d \zb,
		\end{align}
		\begin{figure}
			\centering
			\begin{tikzpicture}
				\draw (-5,0) -- (5,0) node[anchor=west]{$\R$};
				\draw (0,-2.5)--(0,2) node[above]{$i\R$};
				\filldraw[black] (-4,0) circle (1pt) node[above]{$-r$};
				\filldraw[black] (4,0) circle (1pt) node[above]{$r$};
				\filldraw[black] (0,-1.5) circle (1pt) node[anchor=north west]{$-i\sigma h \xi$};
				\draw [-stealth](-4,0) -- (-4,-0.75) node[left]{$\Gamma_1$};
				\draw (-4,-0.75)--(-4,-1.5);
				\draw [-stealth](-4,-1.5) -- (-1,-1.5) node[below]{$\Gamma_2$};
				\draw (-1,-1.5)--(4,-1.5);
				\draw (-1,-1.5)--(4,-1.5);
				\draw [-stealth](4,-1.5) -- (4,-0.75) node[right]{$\Gamma_3$};
				\draw (4,-0.75)--(4,0);
				\draw[pattern=north east lines, pattern color=blue, ,opacity=0.2] (-4,0) rectangle (4,-1.5);
				
				\node at (1, -1) {$\Omega$};
			\end{tikzpicture}
			\caption{The domain $\Omega$ where we apply Stokes' theorem in case $\xi >0$ (the picture in case $\xi <0$ is the symmetric about the real axis). Notice that $\partial \Omega=\Gamma_1\cup \Gamma_2 \cup \Gamma_3 \cup [-r,r]$. Recall as well that in this regime we have $\xi\sim h^{-2/3}$ and therefore $h \xi \sim h^{1/3}$. As $h$ goes to $0$ the domain $\Omega$ collapses to the segment $[-r,r]$. }
			\label{change of contour}
		\end{figure}
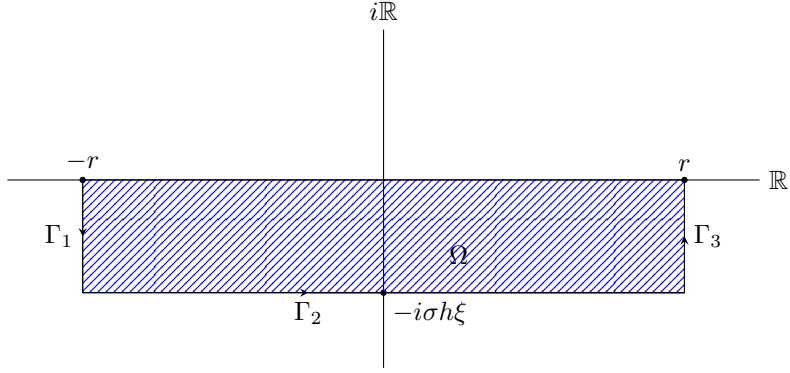
		where the contours (oriented counterclockwise, see Figure~\ref{change of contour} in the case $\xi>0$) are defined by 
		\begin{align*}
			\Gamma_1&=\{z \in \C , \Re z=-r,  \quad-\sigma h \xi \leq \Im z\leq 0 \}, \\
			\Gamma_2&=\{z \in \C , -r \leq \Re z \leq r, \quad \Im z = - \sigma h \xi \}, \\
			\Gamma_3&=\{z \in \C ,  \Re z=r, \quad -\sigma h \xi \leq \Im z\leq 0 \},
		\end{align*}
		if $\xi>0$ and 
		\begin{align*}
			\Gamma_1&=\{z \in \C , \Re z=-r, \quad 0 \leq \Im z \leq \sigma h \xi  \}, \\
			\Gamma_2&=\{z \in \C , -r \leq \Re z \leq r, \quad \Im z = \sigma h \xi \}, \\
			\Gamma_3&=\{z \in \C ,  \Re z=-r, \quad 0 \leq \Im z \leq \sigma h \xi  \},
		\end{align*}
		if $\xi<0$.
		We now estimate all terms in the right hand-side of~\eqref{e:stokes-Gamma}.
		
		We start with the last term in the right hand-side of~\eqref{e:stokes-Gamma}. Using~\eqref{dzb for g} and~\eqref{dzb estimate for bchech bis} together with the fact that $z\in \Omega$ in particular $|\Im z| \leq \sigma h|\xi| \leq  \frac12 h |\xi|$ (since $\sigma \leq \frac12$), we obtain 
		\begin{align}
			\label{estimate for gevrey term in stokes term}
			\nor{\dzb(g(z))}{\mathcal{L}(\X;\Y)} &\leq  C \nor{f}{2,R,U}\exp{\left(-\frac{1}{C_0 R |\operatorname{Im} (z/2)+h \xi|}\right)}\nonumber \\
			&  \leq C \nor{f}{2,R,U}\exp{\left(-\frac{1}{C_0 R( |\operatorname{Im} (z/2)|+h |\xi|)}\right)} \leq C \nor{f}{2,R,U}\exp{\left(-\frac{1}{2C_0 R | h \xi|}\right)} \nonumber \\
			&\leq C \nor{f}{2,R,U}\exp{\left(-\frac{1}{2C_0 c_2 R h^{1/3}}\right)}  =  C \nor{f}{2,R,U}e^{-\tilde{c} h^{-1/3}},
		\end{align}
		where $c_2$ is given by $|\xi|\leq c_2 h^{-2/3}$ and $\tilde{c}:=\frac{1}{2C_0 c_2 R}$. 
		We write $z=\alpha+i \beta$ with $\alpha, \beta \in \R$ and notice that for $z \in \Omega$ we have $|\beta| \leq \sigma h |\xi|$ and $\beta\xi<0$ (in both cases). As a consequence, we deduce
		\begin{align*}
			\left|e^{-iz\xi}e^{-\frac{z^2}{2h}} \right|&= e^{\beta\xi} e^{-\frac{\alpha^2-\beta^2}{2h}}\leq e^{\frac{\sigma^2 h^2 |\xi|^2}{2h}}\leq e^{\frac{\sigma^{2} h |\xi|^2}{2}}\leq e^{\frac{\sigma^{2} c_2^2}{2 h^{1/3}}}  . 		\end{align*}
		Together with~\eqref{estimate for gevrey term in stokes term} this yields 
		$$
		\int_{\Omega} \nor{\dzb(g(z))e^{-iz\xi}e^{-\frac{z^2}{2h}} }{\mathcal{L}(\X;\Y)} |dz \wedge d \zb|  \leq C \nor{f}{2,R,U}e^{- (\tilde{c}-\sigma^{2} c_2^2/2) h^{-1/3}} \leq C \nor{f}{2,R,U}e^{-\tilde{c}/2 h^{-1/3}},
		$$
		after having chosen $\sigma := \min (\frac{\tilde{c}^{1/2}}{ c_2},\frac{1}{2})$.  
		With $\sigma$ now fixed we control the other three terms in~\eqref{e:stokes-Gamma}.
		
		\begin{itemize}
			\item For $\alpha+i \beta=z \in \Gamma_1$ we have $\alpha^2=r^2$ and estimate the real part of the second exponential, using $(\beta + h\xi)^2 \leq (h\xi)^2$ (in both cases $-\sigma h \xi\leq\beta \leq 0$ if $\xi \geq 0$ and $0 \leq \beta \leq -\sigma h \xi$ if $\xi <0$), as 
			\begin{align*}
				\Re \left(\frac{(z+ih\xi)^2}{2h}\right)&=\frac{r^2-(\beta+h\xi)^2}{2h}\geq \frac{r^2-h^2\xi^2}{2h} \geq \frac{r^2-c_2^2h^{2/3}}{2h}\geq \frac{r^2}{4h}\geq 0,
			\end{align*}
			for $h$ sufficiently small.
			This implies
			\begin{align}
				\label{e:estim-gamma1}
				\int_{\Gamma_1} \nor{g(z)e^{\frac{-h\xi^2}{2}}e^{-\frac{(z+ih\xi)^2}{2h}}}{\mathcal{L}(\X;\Y)}dz\leq C \nor{\tf}{W^{1,\infty}(U_\rho; \mathcal{L}(\X;\Y))} e^{\frac{-h\xi^2}{2}}\leq  C  \nor{\tf}{W^{1,\infty}(U_\rho; \mathcal{L}(\X;\Y))} e^{-ch^{-1/3}},
			\end{align}
			thanks to~\eqref{bound for b check}.
			\item For the integral in $\Gamma_3$ we proceed exactly as for $\Gamma_1$.
			
			\item For $\alpha+i \beta = z\in \Gamma_2$ we have $\beta=- \sigma h \xi$ and $\alpha \in [-r,r]$, and we obtain
			\begin{align*}
				\Re \left( \frac{h|\xi|^2}{2}+\frac{(z+ih\xi)^2}{2h}\right)&  = \frac{h\xi^2}{2} + \frac{\alpha^2 - (\beta+h\xi)^2}{2h} \\
				& \geq \frac{h \xi^2}{2}-\frac{(\beta+h\xi)^2}{2h}=\frac{h \xi^2}{2}\left(1-(1-\sigma)^2\right)\geq \frac{\sigma h \xi^2}{2}  \geq \frac{\sigma c_1^2}{2} h^{-1/3},
			\end{align*}
			for $|\xi| \geq c_1h^{-2/3}$. The estimate of  $\int_{\Gamma_3}$ in~\eqref{e:stokes-Gamma} then proceeds as that of $\int_{\Gamma_1}$ in~\eqref{e:estim-gamma1}.
		\end{itemize}
		This concludes the proof of Lemma~\ref{lemma for estimate of integral in w}.
	\end{proof}
	\subsection{Proof of Proposition~\ref{good conjugagte with function}}
	We can now turn to the proof of Proposition~\ref{good conjugagte with function}.
	
	\begin{proof}[Proof of Proposition~\ref{good conjugagte with function}]
		For $u \in \mathcal{S}(\R;\X)$, we start by writing
		\begin{align}
			\label{difference facile}
			\chi F_h e^{-\frac{h}{2}|D_t|^2}\theta u-e^{-\frac{h}{2}|D_t|^2} f\theta u & =\left(\chi F_h e^{-\frac{h}{2}|D_t|^2}\theta u-\chi e^{-\frac{h}{2}|D_t|^2} f\theta u\right) -(1-\chi)e^{-\frac{h}{2}|D_t|^2}(f\theta u) \nonumber \\
			& = R_h u  -(1-\chi)e^{-\frac{h}{2}|D_t|^2}(f\theta u) ,
		\end{align}
		where $R_h$ is defined in~\eqref{e:def-Rh}.
		The second term in~\eqref{difference facile} is bounded using Lemma~\ref{lemma 2.4 from ll} by
		\begin{align}
			\label{l:1-chi theta}
			\nor{(1-\chi)e^{-\frac{h}{2}|D_t|^2}(f\theta u)}{L^2(\R;\Y)} \leq C e^{-c/h} \nor{f \theta u }{H^{-k}(\R;\Y)} \leq  C e^{-c/h} \nor{f}{W^{k,\infty}(\supp(\theta);\L(\X,\Y))}\nor{u}{H^{-k}(\R;\X)}
		\end{align}
		thanks to the supports of $(1-\chi)$ and $\theta$. Concerning the first term in~\eqref{difference facile}, the kernel of $R_h$ is $\mathcal{K}_h(t,s)$ given by~\eqref{difference noyau} according to Lemma~\ref{l:identify-kernel}.
		Since $\mathcal{K}_h(t,s)=-\frac{1}{2 \pi}\mathcal{K}_{1,h}(t,s)+ C_h \mathcal{K}_{2,h}(t,s)$, Lemma~\ref{estimate for the kernel I} applied in the particular case $m=0, \chi_1=\chi$ yields
		\begin{equation}
			\label{kernel l infini}
			\nor{\mathcal{K}_h(\cdot,\cdot)}{L^\infty(\R \times \R; \mathcal{L}(\X;\Y))}\leq C  e^{-\frac{ch^{-1/3}}{R}}  \nor{f}{2,R,U}.
		\end{equation}
		Combining Lemmata~\ref{l:identify-kernel} and~\ref{estimate for the kernel I} and recalling $\supp \mathcal{K}_h \subset (t_0-4r,t_0+4r) \times (t_0-r,t_0+r)$, the Cauchy-Schwarz inequality yields
		\begin{align*}
			\nor{R_h u }{L^2(\R;\Y)}&= \nor{\int \mathcal{K}_h(\cdot,s) u(s)ds}{L^2(\R;\Y)} 
			\leq  C e^{-\frac{ch^{-1/3}}{R}}  \nor{f}{2,R,U} \nor{ u}{L^2((t_0-r,t_0+r);\X)} .
		\end{align*}
		This, together with~\eqref{difference facile} and \eqref{l:1-chi theta}, implies 
				\begin{align*}
			\nor{\chi F_h e^{-\frac{h}{2}|D_t|^2} \theta u-e^{-\frac{h}{2}|D_t|^2} f \theta u }{L^2}
			&\leq  C e^{-\frac{ch^{-1/3}}{R}}  \nor{f}{2,R,U} \nor{ u}{L^2},
		\end{align*}
		and concludes the proof of Proposition~\ref{good conjugagte with function} for $k=0$.
		
		To obtain the estimate for $k\in \N^*$, and given~\eqref{difference facile} and~\eqref{l:1-chi theta}, it only remains to prove that
					\begin{align}
			\label{e:to-be-k}
		\nor{R_h \dt^k u}{L^2(\R;\Y)} \leq  C_k h^{-k} e^{-\frac{ch^{-1/3}}{R}}\left( \sum_{j \leq k}\nor{f^{(j)}}{2,R,U} \right) \nor{u}{L^2(\R;\X)} ,
		\end{align}
		with $R_h$ defined in~\eqref{e:def-Rh}. We can suppose without loss of generality that $k=2n$, $n\in \N$ and thus
		\begin{align}
		\label{e:ineq-rh-k}
		\nor{R_h \dt^k u}{L^2} \leq C\nor{R_h u}{L^2}+C \sum_{\ell=1}^n\nor{R_h D^{2\ell}_t u}{L^2}.
		\end{align}
		It suffices therefore to control the terms $\nor{R_h D^{\ell}_t u}{L^2}$ for $\ell\geq1$. To do so we observe that the kernel of $R_h D_t$ is given by $D_s \mathcal{K}_h$
		where $\mathcal{K}_h$ is the kernel of $R_h$. Recalling~\eqref{e:K-K1+K2}, we need consequently to control $\nor{\d_s^\ell \mathcal{K}_{j,h}(t,s)}{\mathcal{L}(\X;\Y)}$ for $j=1,2$ and prove that they satisfy the estimate of Lemma~\ref{estimate for the kernel I}. Concerning the term $\d_s^\ell \mathcal{K}_{1,h}(t,s)$ we remark that the desired bound follows from Lemma~\ref{estimate for the kernel I} applied to some derivatives of $\theta$ and $f$ instead of $\theta$ and $f$.  We need consequently to study $\d_s^\ell \mathcal{K}_{2,h}(t,s)$. 
		According to Lemma~\ref{calcul clef pour les derivees de I} below, applied to $\mathcal{K}_{2,h}=\mathcal{I}_{2,h}[\chi,\theta,f,0]$, and recalling that $\supp(\mathcal{I}_{2,h})\subset \supp(\chi)\times \supp(\theta)$ which is a compact set in  $(t,s)$ (whence $|t-s|^{k_2}$ is bounded on this set) we have 
		\begin{align*}
		\left\| \d^\ell_s \mathcal{K}_{2,h}(t,s)\right\|_{\L(\X,\Y)} & \leq C_\ell h^{-\ell}  \sum_{k_j \leq \ell} \left\| \mathcal{I}_{2,h}[\chi,\theta^{(k_3)},f^{(k_4)},k_5](t,s) \right\|_{\L(\X,\Y)}\nonumber\\ 
				& \quad + C_\ell h^{-\ell}  \sum_{k_j\leq \ell} \nor{ B[\theta^{(k_2)},k_3,k_4](t,s)}{\L(\X,\Y)} ,
		\end{align*}
		where we take $\chi_1=\chi$ in the definition of $B$.
		 Using Lemma~\ref{estimate for the kernel I} to estimate all terms involving $\mathcal{I}_{2,h}$ and proceeding as in~\eqref{bound for I21} to estimate all terms involving $B$ (where we use the localization of $\supp(\chi^\prime)$), we obtain for all $(t,s)\in \R^2$ and $h\leq 1$,
	  	\begin{align*}
		\left\| \d^\ell_s \mathcal{K}_{2,h}(t,s)\right\|_{\L(\X,\Y)} & \leq C_\ell h^{-\ell}  \left( e^{-\frac{ch^{-1/3}}{R}} + e^{-c/h} \right) \sum_{j \leq \ell}\nor{f^{(j)}}{2,R,U}  .
		\end{align*}
		Coming back to~\eqref{e:ineq-rh-k}, we have now obtained~\eqref{e:to-be-k}, which concludes the proof of Proposition~\ref{good conjugagte with function}.
	\end{proof}

		\begin{lemma}
			\label{calcul clef pour les derivees de I}
			For all $\chi_1,\theta_1 \in C^\infty_c(\R)$ and $f \in \mathcal{G}_b^{2,R}(\R;\L(\X,\Y))$,  $m,\ell \in \N$, there are coefficients $\alpha_k,\beta_k \in \R$ such that  
				\begin{align}
				\label{e:dell-I}
				\d^\ell_s \mathcal{I}_{2,h}[\chi_1,\theta,f,m] (t,s)& =\sum_{k_j \leq \ell} \alpha_k h^{-k_1}(t-s)^{k_2} \mathcal{I}_{2,h}[\chi_1,\theta^{(k_3)},f^{(k_4)},m+k_5](t,s) \nonumber\\ 
				& \quad +\sum_{k_j\leq \ell} \beta_k h^{-k_1} B[\theta^{(k_2)},k_3,m+k_4](t,s) .
			\end{align}
					where 
		\begin{align}
		B[\theta,m,k](t,s) & := \chi_1(t) \theta(s) \int_{\R \times \R}   \chi' \left(\frac{t+w}{2}\right) \eta(h^{2/3}\xi)  \tf \left(\frac{t+w}{2}+i h \xi\right)\nonumber\\
		& \quad \times  e^{i(t-w)\xi}e^{-\frac{|w-s|^2}{2h}} (w-s)^k \left(\frac{t+w}{2}+i h \xi-s \right)^m dw d \xi . \label{e:def-B-derI}
		\end{align}	
		\end{lemma}
		The proof of Lemma~\ref{calcul clef pour les derivees de I} relies on the following identities.
				\begin{lemma}
		\label{l:computation-dI}
		We have 
		\begin{align}
		\label{e:dt-I2}
		\d_t \mathcal{I}_{2,h}[\chi_1,\theta,f,m] &= \mathcal{I}_{2,h}[\chi_1',\theta,f,m]
		-h^{-1}(t-s)\mathcal{I}_{2,h}[\chi_1,\theta,f,m] \nonumber  \\
		& \quad 		+2h^{-1}\mathcal{I}_{2,h}[\chi_1,\theta,f,m+1] ,
		\end{align}
		and 
			\begin{align}
		\label{e:goal-decomp}
		(\d_t+\d_s) \mathcal{I}_{2,h}[\chi_1,\theta,f,m] & =  \mathcal{I}_{2,h}[\chi_1',\theta,f,m] +  \mathcal{I}_{2,h}[\chi_1,\theta',f,m] \nonumber\\
		 & \quad +  \mathcal{I}_{2,h}[\chi_1,\theta,f',m] + B[\theta,m,0] .
		\end{align}
		\end{lemma}
		As a direct corollary of Lemma~\ref{l:computation-dI}, decomposing 
		\begin{align*}
		\d_s \mathcal{I}_{2,h} =(\d_t+\d_s) \mathcal{I}_{2,h}  -\d_t \mathcal{I}_{2,h} ,
		\end{align*}
		we deduce the following key formula
		\begin{align}
		\label{e:forme-utile}
		\d_s \mathcal{I}_{2,h}[\chi_1,\theta,f,m] & =  \mathcal{I}_{2,h}[\chi_1,\theta',f,m] + \mathcal{I}_{2,h}[\chi_1,\theta,f',m] + h^{-1}(t-s)\mathcal{I}_{2,h}[\chi_1,\theta,f,m] \nonumber \\
		&\quad -2h^{-1}\mathcal{I}_{2,h}[\chi_1,\theta,f,m+1] + B[\theta,m,0] .
		\end{align}
		We also notice that differentiation under the integral yields
		\begin{align}
		\label{e:forme-utile-2}
		\d_s B[\theta,m,k] = B[\theta',m,k] +h^{-1}B[\theta,m,k+1] - k B[\theta,m,k-1] - m B[\theta,m-1,k] .
		\end{align}
		With these two formulas at hand, we are now prepared to prove Lemma~\ref{calcul clef pour les derivees de I}.

		\bnp[Proof of Lemma~\ref{calcul clef pour les derivees de I} from~\eqref{e:forme-utile} and~\eqref{e:forme-utile-2}]
		The proof proceeds by induction on $\ell \in \N$. For $\ell =0$, the result holds straightforwardly with $\alpha_{(0,0,0,0,0)}=1$ and $\beta_{(0,0,0,0)}=0$.
		Assume now that the result holds at range $\ell$ and prove it at range $\ell+1$. Differentiating~\eqref{e:dell-I}, we obtain
		\begin{align*}
				&\d^{\ell+1}_s \mathcal{I}_{2,h}[\chi_1,\theta,f,m] =\sum_{k_j \leq \ell} \alpha_k h^{-k_1}\Big( (t-s)^{k_2} \d_s\mathcal{I}_{2,h}[\chi_1,\theta^{(k_3)},f^{(k_4)},m+k_5]\\
				& \quad -k_2(t-s)^{k_2-1}\mathcal{I}_{2,h}[\chi_1,\theta^{(k_3)},f^{(k_4)},m+k_5]\Big)  +\sum_{k_j\leq \ell} \beta_k h^{-k_1} \d_s B[\theta^{(k_2)},k_3,m+k_4] .
			\end{align*}
					Using~\eqref{e:forme-utile}, we deduce that the first term, involving $\d_s\mathcal{I}_{2,h}$, has the form~\eqref{e:dell-I} with $\ell$ replaced by $\ell+1$. The second term, involving $(t-s)^{k_2-1}\mathcal{I}_{2,h}$, is directly under the appropriate form as well. Finally,~\eqref{e:forme-utile-2} implies that the last term, involving $\d_s B$ is also of the form~\eqref{e:dell-I} with $\ell$ replaced by $\ell+1$.
		\enp

		We conclude by proving Lemma~\ref{l:computation-dI}.
		\begin{proof}[Proof of Lemma~\ref{l:computation-dI}]
		Formula~\eqref{e:dt-I2} directly follows from rewriting $\mathcal{I}_{2,h}$ as in~\eqref{alternative form for kernel} and differentiating under the integral.
		Concerning Formula~\eqref{e:goal-decomp}, we rewrite $\mathcal{I}_{2,h}$ as 
		\begin{align}
		\mathcal{I}_{2,h}(t,s)&  = \chi_1(t) \theta(s) \mathcal{J}_2(t,s)  \quad \text{with} \label{e:I2--J2}\\
		  \mathcal{J}_2(t,s) & :=\int_{\R \times \R} F(t,w,s,\xi) e^{i(t-w)\xi}e^{-\frac{|w-s|^2}{2h}}  dw d \xi  , \nonumber\\
		F(t,w,s,\xi)& := \left(\tf^r \left(\frac{t+w}{2}+i h \xi\right)-\eta(h^{2/3}\xi)f(s)\right)\left(\frac{t+w}{2}+i h \xi-s \right)^m , \nonumber
		\end{align}
		From~\eqref{e:I2--J2} we deduce
		\begin{align}
		\label{e:I2--J2-bis}
		(\d_t+\d_s)\mathcal{I}_{2,h}(t,s)  = \chi_1'(t) \theta(s) \mathcal{J}_2(t,s)  +  \chi_1(t) \theta'(s) \mathcal{J}_2(t,s) +\chi_1(t) \theta(s) (\d_t+\d_s)\mathcal{J}_2(t,s) 
		\end{align}
		Next, we focus on $\mathcal{J}_2$. Using that $(\d_t+\d_w)(e^{i(t-w)\xi})=0$, we have on the one hand
		\begin{align*}
		\d_t \mathcal{J}_2(t,s) = \int_{\R \times \R} \d_t F(t,w,s,\xi) e^{i(t-w)\xi}e^{-\frac{|w-s|^2}{2h}}  dw d \xi -  \int_{\R \times \R} F(t,w,s,\xi)\d_w( e^{i(t-w)\xi}) e^{-\frac{|w-s|^2}{2h}}  dw d \xi  .
		\end{align*}
		Integrating by parts in $w$ in the second integral, and using $(\d_w+\d_s)(e^{-\frac{|w-s|^2}{2h}})=0$, we deduce
		\begin{align*}
		\d_t \mathcal{J}_2(t,s) = \int_{\R \times \R} (\d_t +\d_w)F(t,w,s,\xi) e^{i(t-w)\xi}e^{-\frac{|w-s|^2}{2h}}  dw d \xi - \int_{\R \times \R} F(t,w,s,\xi) e^{i(t-w)\xi}\d_s( e^{-\frac{|w-s|^2}{2h}} ) dw d \xi  .
		\end{align*}
		On the other hand, we have 
		\begin{align*}
		\d_s \mathcal{J}_2(t,s) = \int_{\R \times \R} \d_s F(t,w,s,\xi) e^{i(t-w)\xi}e^{-\frac{|w-s|^2}{2h}}  dw d \xi +  \int_{\R \times \R} F(t,w,s,\xi) e^{i(t-w)\xi} \d_s(e^{-\frac{|w-s|^2}{2h}})  dw d \xi  ,
		\end{align*}
		which, combined with the previous line yields
		\begin{align}
		\label{e:d-d-J2}
		(\d_t+\d_s) \mathcal{J}_2(t,s) = \int_{\R \times \R} (\d_t +\d_w+\d_s)F(t,w,s,\xi) e^{i(t-w)\xi}e^{-\frac{|w-s|^2}{2h}}  dw d \xi .
		\end{align}
		We next notice that $(\d_t +\d_w+\d_s) \left(\frac{t+w}{2}+i h \xi-s \right)^m=0$ and 
		\begin{align*}
		&  (\d_t +\d_w+\d_s)\left(\tf^r \left(\frac{t+w}{2}+i h \xi\right)-\eta(h^{2/3}\xi)f(s)\right)   = \d_{\Re(z)}\left(\tf^r\right) \left(\frac{t+w}{2}+i h \xi\right)-\eta(h^{2/3}\xi)f'(s) \\
		 & \quad = \eta(h^{2/3}\xi) \left(  \chi' \left(\frac{t+w}{2}\right) \tf \left(\frac{t+w}{2}+i h \xi\right)+ \chi \left(\frac{t+w}{2}\right)  \d_{\Re(z)}\tf \left(\frac{t+w}{2}+i h \xi\right)-f'(s)  \right) .
		\end{align*}
		Combining this together with~\eqref{e:d-d-J2} and~\eqref{e:I2--J2-bis} and the fact that $\d_{\Re(z)}\tf =\widetilde{(f')}$ (from~\eqref{propriete de commutation} in Lemma~\ref{extension of gevrey}) finally yields~\eqref{e:goal-decomp} and concludes the proof of the lemma.
		\end{proof}
			
	\section{The unique continuation theorems}
	\label{section the uniqueness theorem}
	
	\subsection{Adding partially Gevrey lower order terms }
	\label{section adding partially gevrey}
	
	With the results of Section~\ref{section conjugation with gevrey} at our disposal, we can now add in the Carleman estimate of Theorem~\ref{th:carlemanschrod} lower order terms with coefficients which are Gevrey $2$ with respect to $t$ and bounded with respect to $x$. Let $I \subset \R$ and $V \subset \R^d$ be open sets and define $\Omega:=I \times V$. 
	The goal of this section is to prove the following local Carleman estimate for the operator $\pg$ defined in~\eqref{laplacien-perturbe-gevrey}.
	\begin{theorem}[The Carleman estimate with Gevrey lower-order terms]
		\label{th carleman with gevrey terms}
		Let $\xg_0=(t_0,x_0) \in \Omega=I \times V \subset \R^{1+d}$ and assume that the metric $g$ is Lipschitz on $V$, with time-independent coefficients, and $\mathsf{b}^j, \pot \in \mathcal{G}^2(I; L^\infty(V;\C))$.
		Assume that $\phi$ and $f$ satisfy the assumptions of Theorem~\ref{th:carlemanschrod}.
		Then, for all $k \in \N$ and all $\mu>0$, there exist $r, \mathsf{d}, C, \tau_0>0$ such that for all $\tau \geq \tau_0$ and $w \in C^\infty_c(B(\xg_0,r))$, we have 
		\bnan
		\label{Carlemanschrod with gevrey}
		C\nor{Q_{\e,\tau}^{\phi}\pg w}{L^2}^2+ Ce^{-\mathsf{d}\tau}\nor{e^{\tau\phi}w}{H^{-k}_t H^1_x}^2 \geq  \tau \|Q_{\e,\tau}^{\phi}w\|_{\Ht{1}}^2 .
		\enan
	\end{theorem}
	Note that this Carleman estimate is still valid for $P_{\mathsf{b},\pot, \varphi}$ (defined in~\eqref{e:def-P-b-pot-phi} below) in place of $\pg$ according to Remark~\ref{r:carleman-density}.
	\begin{proof}
		We define $R:=\sum_{j=1}^d \mathsf{b}^j\d_{x_j}+\pot$ so that $\pg=i\d_t+\Delta_{g,1}+R$. We estimate $ \nor{Q_{\e,\tau}^{\phi}\pg w}{L^2}^2 \gtrsim \nor{Q_{\e,\tau}^{\phi}P w}{L^2}^2-\nor{Q_{\e,\tau}^{\phi} R w}{L^2}^2$. Application of~\eqref{Carlemanschrod} in Theorem~\ref{th:carlemanschrod} yields
		\begin{align}
			\label{estimate aux carleman gevrey}
			\nor{Q_{\e,\tau}^{\phi}\pg w}{L^2}^2+e^{-\mathsf{d}\tau}\nor{e^{\tau\phi}w}{H^{-k}_t H^1_x}^2 &\gtrsim \nor{Q_{\e,\tau}^{\phi}P w}{L^2}^2+e^{-\mathsf{d}\tau}\nor{e^{\tau\phi}w}{H^{-k}_t H^1_x}^2 -\nor{Q_{\e,\tau}^{\phi} R w}{L^2}^2  \nonumber \\
			&\gtrsim  \tau \|Q_{\e,\tau}^{\phi}w\|_{\Ht{1}}^2 -\nor{Q_{\e,\tau}^{\phi} R w}{L^2}^2.
		\end{align}
		We now estimate the last term using Proposition~\ref{good conjugagte with function}, up to reducing $r$. In order for all the setting of Section~\ref{section conjugation with gevrey} to apply, we pick $r_{0}$ small enough so that $J=(t_{0}-2r_{0},t_{0}+2r_{0})\subset I$ and $\rho>0$ is arbitrary. If $r$ is the one given by Theorem~\ref{th:carlemanschrod}, we reduce it again in order to ensure the assumption $0<r< \min(\frac{r_0}{4},\frac{\rho}{3})$. We select $\chi, \theta , \eta$ with the additional assumption that $\theta=1$ on $[t_{0}-r/2,t_{0}+r/2]$. We denote by $B_{j,h}$ the approximate conjugated operator associated to $\mathsf{b}^j$ as defined in Section~\ref{section conjugation with gevrey}, that is $B_{j,h}=F_{h}$ as defined in \eqref{def of Fh}, in the case $f=\mathsf{b}^j$ and $h$ is linked to $\tau$ via $h=\e/\tau^3$. We will keep however the $h$ notation for the conjugated operator. The function  $\mathsf{b}^j\in \mathcal{G}^2(J; L^\infty(V;\C))$ is identified with the multiplication operator in $\mathcal{G}^2(J; \mathcal{L}(L^2(V;\C)))$, that is, we make the choice $\X = \Y = L^2(V)$ and $\ban = L^\infty(V)$.

		We now assume $w \in C^\infty_c(B(\xg_0,r/2))$ so that $\theta w=w$. Applying Proposition~\ref{good conjugagte with function} with $u=e^{\tau \phi}\d_{x_j} w=\theta u$ gives 
		$$
		\nor{\chi B_{j,h} e^{-\frac{\e |D_t|^2}{2 \tau^3}}u-e^{-\frac{\e |D_t|^2}{2 \tau^3}} \mathsf{b}^j u}{L^2(\R;L^2(V))}\leq  C e^{-c \tau}\nor{u}{H^{-k}(\R;L^2(V))}.
		$$
		According to~\eqref{e:bornitude-L2-Fh}, $B_{j,h} \in \mathcal{L}(L^2(\R;L^2(V)))$ uniformly in $h\in(0,1)$, which, combined with the previous estimate gives
		\begin{align*}
			\nor{Q_{\e,\tau}^{\phi} \mathsf{b}^j \d_{x_j} w}{L^2}&=\nor{e^{-\frac{\e |D_t|^2}{2 \tau^3}} \mathsf{b}^j u}{L^2}\lesssim \nor{B_{j,h} e^{-\frac{\e |D_t|^2}{2 \tau^3}}u }{L^2}+ e^{-c \tau} \nor{u}{H^{-k}_t L^2_x}\\
			&\lesssim \nor{e^{-\frac{\e |D_t|^2}{2 \tau^3}}u }{L^2}+ e^{-c \tau} \nor{u}{H^{-k}_t L^2_x}= \nor{{Q_{\e,\tau}^{\phi} \d_{x_j} w}}{L^2}+ e^{-c \tau}\nor{e^{\tau \phi} \d_{x_j} w}{H^{-k}_t L^2_x}. 
		\end{align*}	
		Using that $e^{\tau\phi}\d_{x_j}w=\d_{x_j}(e^{\tau\phi}w) - \tau (\d_{x_j}\phi) e^{\tau\phi}w$ and $[e^{-\frac{\e |D_t|^2}{2 \tau^3}} ,\d_{x_j}]=0$, this implies
		\begin{align*}
			\nor{Q_{\e,\tau}^{\phi} \mathsf{b}^j \d_{x_j} w}{L^2} &\lesssim \tau \nor{{Q_{\e,\tau}^{\phi}  w}}{L^2}+ \nor{{Q_{\e,\tau}^{\phi} w}}{H^1_x}+ \tau e^{- c \tau} \nor{e^{\tau \phi} w}{H^{-k}_t H^1_x}.
		\end{align*}
		We proceed similarly for the potential $\pot$ to find $ \nor{Q_{\e,\tau}^{\phi} \pot  w}{L^2}\lesssim \nor{{Q_{\e,\tau}^{\phi}  w}}{L^2}+ e^{-c \tau} \nor{e^{\tau \phi} w}{H^{-k}_t L^2_x} $ and therefore adding these two estimates yields
		\begin{equation}
			\label{estimate for the rest of gevrey}
			\nor{Q_{\e,\tau}^{\phi} R w}{L^2} \lesssim  \tau \nor{{Q_{\e,\tau}^{\phi}  w}}{L^2}+ \nor{{Q_{\e,\tau}^{\phi} w}}{H^1_x}+ e^{-\frac{c}{2} \tau} \nor{e^{\tau \phi} w}{H^{-k}_t H^1_x}.
		\end{equation}
		Estimate~\eqref{estimate for the rest of gevrey} allows to absorb the last term in~\eqref{estimate aux carleman gevrey} up to taking $\tau\geq \tau_0$ with $\tau$ sufficiently large. This concludes the proof of Theorem~\ref{th carleman with gevrey terms} up to renaming the constants $r,C,c, \mathsf{d}$ and $\tau_0$.
	\end{proof}

	\subsection{Using the Carleman estimate: proof of Theorem~\ref{theorem local uniqueness intro}}
	\label{section using the carleman estimate}
	
	In this section, we prove Theorem~\ref{theorem local uniqueness intro} as a consequence of the Carleman estimate of Theorem~\ref{th carleman with gevrey terms}.
	As usual in this procedure (see e.g.~\cite[Chapter~28]{Hoermander:V4},~\cite{Lerner:book-carleman} or~\cite{LL:23notes}), we need to construct a weight function $\phi$ that
	\begin{itemize}
		\item satisfies the assumptions of Theorem~\ref{th carleman with gevrey terms}, that is the assumptions of Theorem~\ref{th:carlemanschrod};
		\item has level sets appropriately curved with respect to the level sets of $\Psi$; this is the geometric convexification part.
	\end{itemize}
	This is the content of the following lemma, in which we recall that $I \subset \R$ and $V \subset \R^d$ denote bounded open sets and we write $\xg=(t,x)$.
	\begin{lemma}
		\label{geometric convexification}
		Let $\xg_0=(t_0,x_0) \in \Omega=I \times V \subset \R^{1+d}$ and assume that the metric $g$ is Lipschitz on $V$, with time-independent coefficients, and $\mathsf{b}^j, \pot \in \mathcal{G}^2(I; L^\infty(V;\C))$.
		Let $\Psi \in C^2(\Omega ; \R)$ satisfy~\eqref{e:non-char-schro} and $\Psi(\xg_0)=0$.
		Then there exists a quadratic polynomial $\phi$ and a function $f$ satisfying the assumptions of Theorem~\ref{th:carlemanschrod} together with the following properties:  $\phi(\xg_0)=0$ and there exists $r_0$ such that for any $0<r<r_0$ there exists $\eta>0$ so that $\phi(\xg)\leq -\eta$ for $\xg \in \{\Psi\leq 0\} \cap \{r/2 \leq |\xg-\xg_0| \leq r \}$.
	\end{lemma}
	\begin{proof}
		Given $\Psi \in C^2(\Omega ; \R)$ define $\check{\phi}=G(\Psi)$ and $f$ as in~\eqref{e:choice-phi-f} with $G(s)=e^{\lambda s}-1$. Note in particular that $\check{\phi}$ and $\Psi$ have the same level sets. Then using Corollary \ref{c:explicit-convexification-exp}, one has, for $\lambda$ large enough, almost everywhere on $U$ and for every vector field $X$, 
		\begin{align}
			\label{e:cond-sous-ell-inpf}
			\B_{g, \check{\phi}, f}(X) \geq  C_0 \gln{X}, \quad \text{ and } \quad 
			\E_{g, \check{\phi}, f}  \geq C_0  \gln{\nablag \check{\phi}}>0 .
		\end{align}
		Now define $\check{\phi}_T$ by 
		$$
		\check{\phi}_T(\xg):=\sum_{|\alpha|\leq 2}\frac{1}{\alpha !} (\d^{\alpha} \check{\phi})(\xg_0)(\xg-\xg_0)^\alpha.
		$$
		Observe that both quantities $\B_{g, \check{\phi}, f}$ and $\E_{g, \check{\phi}, f}$ involve derivatives of order at most $2$ of $\check{\phi}$. Since $\Psi$ is $C^2$ and $G$ is smooth, $\check{\phi}=G(\Psi)$ is of class $C^2$ as well. Since $(\d^\alpha\check{\phi}_T)(\xg_0)=(\d^\alpha \check{\phi})(\xg_0)$ for $\alpha \leq 2$ we obtain by continuity that for any $\eps >0$,  there exists $r_1>0$ such that
		$\nor{\check{\phi}_T-\check{\phi}}{C^2(B(\xg_0,r_1))}<\eps$.
		Define finally $\phi$ by
		$$
		\phi:= \check{\phi}_T-\delta |\xg-\xg_0|^2.
		$$
		Then there is $\delta_0>0$ such that for all $\delta \in (0,\delta_0)$, $\nor{\check{\phi}_T-\phi}{C^2(B(\xg_0,r_1))}<\eps$ and hence $\nor{\phi-\check{\phi}}{C^2(B(\xg_0,r_1))}<2\eps$.
		As a consequence of~\eqref{e:cond-sous-ell-inpf}, together with the fact that $\B_{g, \phi, f}$ and $\E_{g, \phi, f}$ (defined in~\eqref{e:def-B}-\eqref{e:def-E}) are continuous with respect to $\phi$ in $C^2$ topology, we finally deduce existence of $r_1>0$ and $\delta>0$ such that for a.e. $\xg \in \overline{B}(\xg_0,r_1)$ and for all vector fields $X$, 
		\begin{align*}
			\B_{g, \phi, f}(\xg)(X) \geq  \frac{C_0}{2} \gln{X}, \quad \text{ and } \quad 
			\E_{g, \phi, f}(\xg)  \geq  \frac{C_0}{2} \gln{\nablag \phi}(\xg)>0, 
		\end{align*}
		As a consequence, $\phi$ satisfies the assumptions of Theorem~\ref{th:carlemanschrod}. The geometric statement of the lemma follows from the facts that $\check{\phi}$ and $\Psi$ have the same level sets and $\check{\phi}_T$ is the order $2$ Taylor expansion of $\check{\phi}$ (see e.g.~\cite[Proof of Theorem~2.2]{LL:23notes}).
	\end{proof}
	We are now prepared to prove Theorem~\ref{theorem local uniqueness intro}.

	\begin{proof}[Proof of Theorem~\ref{theorem local uniqueness intro}]
		Consider $u$ a solution of $\pg u=0$ such that $u=0$ in $\Omega \cap \{\Psi>0\}$. Let $\phi$ be as in Lemma~\ref{geometric convexification}. Theorem~\ref{th carleman with gevrey terms} for $k=0$ implies that there exist $r, \mathsf{d}, C, \tau_0>0$ such that for all $\tau \geq \tau_0$ and $w \in C^\infty_c(B(\xg_0,r))$, we have 
		\begin{equation}
			\label{e:carleman-appl}
			C\nor{Q_{\e,\tau}^{\phi}\pg w}{L^2}^2+ Ce^{-\mathsf{d}\tau}\nor{e^{\tau\phi}w}{L_t^2H^1_x}^2 \geq  \tau \|Q_{\e,\tau}^{\phi}w\|_{\Ht{1}}^2 .  
		\end{equation}
		According to usual approximation argument, Estimate~\ref{e:carleman-appl} still holds for functions $w \in L^2(I;H^1(V))$ such that $\pg w \in L^2$ and $\supp w \subset B(\xg_0,r)$. 
		We have moreover:
		\begin{enumerate}
			\item \label{i:propty-1} $\phi(\xg_0)=0$ and there exists $\eta>0$ so that $\phi(\xg)\leq -\eta$ for $\xg \in \{\Psi\leq 0 \}\cap \{r\geq|\xg-\xg_0|\geq r/2\}$,
			\item  \label{i:propty-2} $\phi(\xg)\leq \mathsf{d}/4$ for $|\xg-\xg_0|\leq r$.
		\end{enumerate}
		Property~\ref{i:propty-1} comes from Lemma~\ref{geometric convexification} and Property~\ref{i:propty-2} is just the continuity of $\phi$, up to reducing $r$.  Let $\chi \in C^{\infty}_c(B(\xg_0,r))$ with $\chi=1$ in $B(\xg_0,r/2)$.
		In order to apply the Carleman estimate~\eqref{e:carleman-appl} to $w=\chi u \in L^2(I;H^1(V))$, we first estimate 
		\begin{align*}
			\nor{Q_{\e,\tau}^{\phi}\pg \chi u }{L^2}&\leq \nor{Q_{\e,\tau}^{\phi}\chi\pg u }{L^2}+\nor{Q_{\e,\tau}^{\phi}[\pg, \chi] u }{L^2}=\nor{Q_{\e,\tau}^{\phi}[\pg, \chi] u }{L^2} \\
			&\leq \nor{e^{\tau \phi }[\pg, \chi] u }{L^2} \leq e^{-\eta \tau} \nor{u}{L^2_tH^1_x},
		\end{align*}
		according to the fact that $\supp(\nabla_\xg \chi)\subset  \{r\geq|\xg-\xg_0|\geq r/2\}$ and $\supp (u)\subset \{\Psi \leq 0\}$, Property~\ref{i:propty-1} and the fact that $[\pg, \chi]$ is a differential operator of order one with no derivatives in $t$. We have as well
		$$
		e^{-\mathsf{d}\tau}\nor{e^{\tau\phi}w}{L^2_tH^1_x} \leq e^{-3\mathsf{d}\tau/4} \nor{u}{L^2_tH^1_x},
		$$
		thanks to Property~\ref{i:propty-2}. Plugging the last two estimates in~\eqref{e:carleman-appl}, we finally obtain that there exists a $\delta>0$ such that 
		$$
		\nor{Q_{\e,\tau}^{\phi} \chi u }{L^2} \leq \|Q_{\e,\tau}^{\phi} \chi u\|_{\Ht{1}}^2\leq C e^{- \delta \tau}\nor{u}{L^2_tH^1_x},
		$$
		which implies that $ \nor{Q_{\e,\tau}^{\phi+\delta} \chi u }{L^2}\leq C$ uniformly in $\tau\geq \tau_{0}$. Lemma~\ref{lemme d analyse harmonique} gives $\supp(\chi u) \subset \{\phi \leq -\delta\}$. Since $\phi(\xg_0)=0$ and $\chi=1$ in $B(\xg_0,r/2)$ one has that $W=B(\xg_0,r/2) \cap \{\phi> -\delta/2\}$ is a neighborhood of $\xg_0$ in which $\chi u= u= 0$ and the proof of Theorem~\ref{theorem local uniqueness intro} is complete
	\end{proof}
	
	\subsection{Reducing the regularity of the solution: proof of Theorem~\ref{theorem local L2}}
	\label{section reducing the regularity}
	Theorem~\ref{theorem local uniqueness intro} concerns solutions $u \in L^2(I;H^1(V))$ of the  Schr{\"o}dinger equation $\pg u=0$. The  $L^2(I;H^1(V))$ regularity allows in particular not to care about the divergence form and to make sense of $\mathsf{b}^j(t,x) \d_{x_j}u(t,x)$ in the sense of distributions if $b \in L^{\infty}(I\times V)$ only. In the present section, assuming divergence form of the principal part and additional space regularity on the vectorfield $\mathsf{b}$, we generalize Theorem~\ref{theorem local uniqueness intro} to $L^2(I\times V)$ solutions to $\pg u=0$ and prove Theorem~\ref{theorem local L2}. Since the statement of Theorem~\ref{theorem local L2} is sensitive to the form of the elliptic operator involved, we prove it in the more general setting with $\pg$ replaced by 
\begin{align}
\label{e:def-P-b-pot-phi}
	P_{\mathsf{b},\pot, \varphi} = i \d_t +\Delta_{g,\varphi}+\sum_{j=1}^d \mathsf{b}^j(t,x)\d_{x_j}+\pot(t,x) ,
\end{align}
	where $\Delta_{g,\varphi}$ is defined in Section~\ref{s:rem-div}. Then we have $\pg=P_{\mathsf{b},\pot,1}$, i.e. the statement of Theorem~\ref{theorem local L2} corresponds to taking $\varphi=1$, and the application to the second part of Theorem~\ref{theorem global intro} to $\varphi = \sqrt{\det(g)}$. 
	The idea is to use the Carleman estimate of Theorem~\ref{th carleman with gevrey terms} for $k=1$ instead of $k=0$. This allows to exploit the ellipticity of $\Delta_{g,\varphi}$ via Lemma~\ref{lm:ellipticH1H-1} to gain regularity. 
	
	We first state a local regularity result for the Schr\"odinger operator $P_{\mathsf{b},\pot, \varphi}$.

	\begin{lemma}[Local regularity for $\pg$]
		\label{imroved regularity for L2 solution}
		Let $I \subset \R$ and $V \subset \R^d$ be bounded open sets and $\Omega=I \times V$.
		Assume that $g^{jk}\in W^{1,\infty}_{\loc}(V;\R)$ is symmetric and satisfies~\eqref{e:elliptic},
		that $\varphi \in W^{1,\infty}_{\loc}(V;\R)$ satisfies $\varphi >0$ on $V$,
		 that $\pot , \mathsf{b}^j \in L^{\infty}_{\loc}(\Omega;\C)$ 
		and $\sum _{j=1}^d\d_{x_j}\mathsf{b}^j \in  L^\infty_{\loc}(\Omega ;\C)$. 
Let $\chi^t \in C^\infty_c(I)$, $\chi^x\in C^\infty_c(V)$ and set $\chi_{3}(t,x)= \chi^t(t)\chi^x(x)$. Then, there is a constant $C>0$ such that for any $u \in L^2(\Omega)$ with $\chi_3 P_{\mathsf{b},\pot,\varphi}  u \in H^{-1}(\R;H^{-1}(\R^d))$, 
		we have $\chi_{3} u \in H^{-1}(\R ; H^1(V))$ with
		\begin{align}
		\label{e:reg-schrod}
		\nor{\chi_{3}u}{H^{-1}(\R ; H^1(V))} \leq C \nor{\chi_3 P_{\mathsf{b},\pot,\varphi}  u}{H^{-1}(\R;H^{-1} (\R^d))} + C\nor{u}{L^2(\Omega)}.
		\end{align}
	\end{lemma}
	
	\begin{proof}
		We prove~\eqref{e:reg-schrod} for all $u \in C^\infty_c(V)$,  and the lemma follows with a regularization argument left to the reader.
		We define the operator $R:= \sum_{j=1}^d \mathsf{b}^j(t,x) \d_{x_j}+\pot(t,x)$ so that $P_{\mathsf{b},\pot, \varphi} = i \d_t +\Delta_{g,\varphi}+R$ where $\Delta_{g,\varphi}$ is defined in Section~\ref{s:rem-div}. 
		We apply Lemma~\ref{lm:ellipticH1H-1} for any $t \in \R$ to $w=\dt^{-1}\chi^tu$ and integrate in time to obtain
		\begin{align*}
			\nor{\chi_{3}u}{H^{-1}(\R;H^1(V))}&= \nor{ \chi^x\dt^{-1} \chi^tu}{L^{2}(\R;H^1(\R^d))} \\ 
			&\leq C \left( \nor{\chi^x\dt^{-1}\chi^t \Delta_{g,\varphi} u}{L^2(\R; H^{-1}(\R^d))}+\nor{\dt^{-1}\chi^tu}{L^2(\R;L^2(\supp(\chi^x)))}\right) .   
			\end{align*}
			Using that  $\Delta_{g,\varphi} = P_{\mathsf{b},\pot,\varphi}  +D_t -R$, this implies 
			\begin{align}
			\label{mixed norm u for regularity L2}
			\nor{\chi_{3}u}{H^{-1}(\R;H^1(V))} 
			& \leq C\Big( \nor{\chi^x\dt^{-1}\chi^tD_t u}{L^2(\R;H^{-1} (\R^d))}+ \nor{\chi^x\dt^{-1}\chi^t P_{\mathsf{b},\pot,\varphi}  u}{L^2(\R;H^{-1} (\R^d))} \nonumber \\
			& \quad +	\nor{\chi^x\dt^{-1}\chi^tR u}{L^2(\R;H^{-1} (\R^d))}+ \nor{\dt^{-1}\chi^tu}{L^2(\R;L^2(\supp(\chi^x)))}\Big) .
		\end{align} 
		Now observe that  
		$\nor{\dt^{-1}\chi^tD_t}{L^2(\R) \to L^2(\R)}<+\infty,$ so that for any $\tilde{\chi}^t\in C^\infty_c(I)$ with $\tilde{\chi}^t=1$ in a neighborhood of $\chi^t$, we have 
		\begin{align}
		\label{e:intermediate-toto-1}
		 \nor{\chi^x\dt^{-1}\chi^tD_t u}{L^2(\R;H^{-1} (\R^d))} = \nor{\chi^x\dt^{-1}\chi^tD_t \tilde{\chi}^t u}{L^2(\R;H^{-1} (\R^d))} \leq  \nor{\chi^x\tilde{\chi}^t u}{L^2(\R;H^{-1} (\R^d))} .
		\end{align}
		Next remark that 
		\begin{align}
		\label{e:intermediate-toto-2}
		 \nor{\chi^x\dt^{-1}\chi^tP_{\mathsf{b},\pot,\varphi}  u}{L^2(\R;H^{-1} (\R^d))}  & =  \nor{\chi_3P_{\mathsf{b},\pot,\varphi}  u}{H^{-1}(\R;H^{-1} (\R^d))} , \quad \text{ and } \\
		 \label{e:intermediate-toto-3}
		 \nor{\dt^{-1}\chi^tu}{L^2(\R;L^2(\supp(\chi^x)))}& \leq \nor{ \chi^tu}{L^2(\R;L^2(\supp(\chi^x)))} .
		\end{align}
		 To handle the last term, we argue by duality and write
		\begin{equation}
			\label{norm by duality 2}
				\nor{\chi^x\dt^{-1}\chi^tR u}{L^2(\R;H^{-1} (\R^d))}\leq 
			\nor{\chi_3  Ru }{L^2(\R;H^{-1} (\R^d))}=\underset{\theta \in \mathcal{S}(\R^{1+d}), \nor{\theta}{L^2(\R;H^{1}(\R^d))} \leq 1}{\sup}\left|\left(\chi_3 Ru, \theta\right)_{L^2(\R^{1+d})}\right|.
		\end{equation}
		We calculate
		\begin{align*}
			\left(\chi_3 Ru, \theta\right)_{L^2(\R^{1+d})}&=\int_{\R^{1+d}} \sum_{j=1}^{d} \chi_3\mathsf{b}^j \bar{\theta} \d_{x_j} u  \, dtdx+\int_{\R^{1+d}} \chi_3 \pot u \bar{\theta} \, dtdx \\
			&=-\int_{\R^{1+d}} \sum_{j=1}^d \d_{x_j}(\chi_3\mathsf{b}^j \bar{\theta})u \, dtdx+\int_{\R^{1+d}} \chi_3 \pot u \bar{\theta} \, dtdx \\
			&=-\int_{\R^{1+d}} \sum_{j=1}^d (\d_{x_j}\chi^x)\chi^t\mathsf{b}^j \bar{\theta}u dt -\int_{\R^{1+d}} \chi_3 \div_1(\mathsf{b})u \bar{\theta} \, dtdx \\
			&\quad-\int_{\R^{1+d}} \sum_{j=1}^d \chi_3\mathsf{b}^j (\d_{x_j}\bar{\theta})u \, dtdx+\int_{\R^{1+d}} \chi_3 \pot u \bar{\theta} \, dtdx.
		\end{align*}
		Consequently, the Cauchy--Schwarz inequality yields
		\begin{align*}
			\left|\left(\chi_3 Ru, \theta\right)_{L^2(\R^{d+1})}\right| &\leq C \nor{\mathsf{b}}{L^\infty(\supp(\chi_3))} \nor{u}{L^2(\Omega)} \nor{\theta}{L^2(\R^{d+1})}+C\nor{\div_1(\mathsf{b})}{L^\infty(\supp(\chi_3))} \nor{u}{L^2(\Omega)} \nor{\theta}{L^2(\R^{d+1})}\\
			&\quad+C\nor{\mathsf{b}}{L^\infty(\supp(\chi_3))} \nor{u}{L^2(\Omega)} \nor{\theta}{L^2(\R;H^1(\R^d))}+C\nor{\pot}{L^\infty(\supp(\chi_3))} \nor{u}{L^2(\Omega)} \nor{\theta}{L^2(\R^{d+1})} \\
			&\leq C\nor{u}{L^2(\Omega)} \nor{\theta}{L^2(\R;H^1(\R^d))},
		\end{align*}
		and we obtain thanks to~\eqref{norm by duality 2} that 
		$
		\nor{\chi^x \dt^{-1}\chi^t Ru }{L^2(\R ;H^{-1}(\R^d)} \leq C\nor{u}{L^2(\Omega)}.
		$ 
		Combining this together with~\eqref{e:intermediate-toto-1}--\eqref{e:intermediate-toto-2}--\eqref{e:intermediate-toto-3} in~\eqref{mixed norm u for regularity L2} yields finally~\eqref{e:reg-schrod} for all $u \in C^\infty_c(V)$, which concludes the proof of the lemma.
	\end{proof}

	We now prove Theorem~\ref{theorem local L2}  in the more general setting of the operator $P_{\mathsf{b},\pot,\varphi}$.
	\begin{proof}[Proof of Theorem~\ref{theorem local L2}]
			The proof of Theorem~\ref{theorem local L2} proceeds as that of Theorem~\ref{theorem local uniqueness intro}. The main differences are that now we apply the Carleman estimate of Theorem~\ref{th carleman with gevrey terms} for $k=1$ and that we consider the operator $P_{\mathsf{b},\pot,\varphi}$. That Theorem~\ref{th carleman with gevrey terms} still holds for $P_{\mathsf{b},\pot,\varphi}$ in place of $\pg$ is a direct consequence of Remark~\ref{r:carleman-density}. The functions $\Psi$ and $\phi$ are the same as in the proof of Theorem~\ref{theorem local uniqueness intro}, i.e. those furnished by Lemma~\ref{geometric convexification}.  

		Recall that for $\eps,k>0$, 
		$$
		\nor{D_t^k e^{-\eps|D_t|^2}}{L^2\to L^2} = \max_{\xi_t\in \R^+}\xi_t^ke^{-\eps\xi_t^2} = \left( \frac{k}{2 e \eps}\right)^{k/2}.
		$$
		As a consequence, we have for $\tau\geq 1$ (and using $k=1$ in the above identity), 
		\begin{align*}
			\nor{Q_{\e,\tau}^{\phi}P_{\mathsf{b},\pot,\varphi}  w}{L^2}& = \nor{\dt e^{\frac{-\e |D_t|^2}{2\tau^3}} \dt^{-1} e^{\tau \phi} P_{\mathsf{b},\pot,\varphi} w }{L^2}\\
			&  \leq 2  \nor{e^{\frac{-\e |D_t|^2}{2\tau^3}} \dt^{-1} e^{\tau \phi} P_{\mathsf{b},\pot,\varphi} w }{L^2} +2 \nor{D_t e^{\frac{-\e |D_t|^2}{2\tau^3}} \dt^{-1} e^{\tau \phi} P_{\mathsf{b},\pot,\varphi} w }{L^2}  \\
			&   \leq  C \tau^{3/2} \nor{ e^{\tau \phi} P_{\mathsf{b},\pot,\varphi} w }{H^{-1}_t L^2_x} .
		\end{align*}
		This, combined with the Carleman estimate of Theorem~\ref{th carleman with gevrey terms} for $k=1$ yields
		\begin{equation}
			\label{est aux for l2 sol}
			C \tau^3 \nor{  e^{\tau \phi} P_{\mathsf{b},\pot,\varphi} w }{H^{-1}_t L^2_x}^2+Ce^{-\mathsf{d}\tau}\nor{e^{\tau\phi}w}{H^{-1}_t H^1_x}^2 \geq  \tau \|Q_{\e,\tau}^{\phi}w\|_{\Ht{1}}^2.
		\end{equation}
		We now apply Inequality~\eqref{est aux for l2 sol} to $w=\chi u$ with $\chi$ as in the proof of Theorem~\ref{theorem local uniqueness intro} and $u\in L^2(\Omega)$ solution to $P_{\mathsf{b},\pot,\varphi} u$ in $\mathcal{D}'(\Omega)$. According to Lemma~\ref{imroved regularity for L2 solution}, $\chi_3 u \in H^{-1}(\R;H^1(V))$ for all $\chi_3$ with $\supp(\chi_3)\subset I\times V$. Moreover $[P_{\mathsf{b},\pot,\varphi} , \chi]$ is a differential operator with $L^\infty$ coefficients and involving only space derivatives of order at most $1$. As a consequence, $[P_{\mathsf{b},\pot,\varphi} , \chi]u \in H^{-1}(\R;L^2(V))$ and we need to estimate
		$$
		\tau^2 \nor{ e^{\tau \phi} P_{\mathsf{b},\pot,\varphi}  (\chi u) }{H^{-1}_t L^2_x}= \tau^2 \nor{e^{\tau \phi} [P_{\mathsf{b},\pot,\varphi} , \chi]u }{H^{-1}_t L^2_x}.
		$$
		We argue by duality and write 
		\begin{equation}
			\label{norm by duality}
			\nor{e^{\tau \phi} [P_{\mathsf{b},\pot,\varphi} , \chi]u }{H^{-1}_t L^2_x}=\underset{\theta \in \mathcal{S}(\R^{1+d}), \nor{\theta}{H^1_tL^2_x} \leq 1}{\sup}\left|\left(e^{\tau \phi} [P_{\mathsf{b},\pot,\varphi} , \chi]u, \theta\right)_{L^2(\R^{1+d})}\right|.
		\end{equation}
		We choose a function $\chi_1 \in C^\infty_c(\Omega;\R)$ such that $\chi_1=1$ on the support of $\nabla_\x \chi$ and $\supp(\chi_1) \subset \{r\geq |\xg-\xg_0|\geq r/2- \varepsilon \} $ with $\varepsilon>0$ small. We consider as well $\chi_2 \in C^\infty(\Omega;\R)$ with $\chi_2=1$ on $\{\Psi<0\}$ and $\chi_2=0$ on $\{\Psi > \varepsilon\}$. Notice that this implies in particular that $\chi_2=1$ on the support of $u$. Recall that we have the property
		$$
		\phi(\xg)\leq -\eta \quad \text{for all } \xg \in \{\Psi\leq 0 \}\cap \{r\geq|\xg-\xg_0|\geq r/2\}.
		$$
		By continuity, we can then choose $\varepsilon>0$ sufficiently small such that
		\begin{align}
		\label{e:support-phi-psi-2}
		\phi(\xg)\leq -\eta/2 \quad \text{ for all }  \xg \in \{\Psi\leq \varepsilon \}\cap \{r\geq|\xg-\xg_0|\geq r/2- \varepsilon\}= \supp(\chi_1) \cap \supp(\chi_2).
		\end{align}
		We finally take $\chi^t \in C^\infty_c(I)$ and $\chi^x \in C^\infty_c(V)$ such that $\chi_{3}(t,x):= \chi^t(t)\chi^x(x)$ satisfies $\chi_{3}=1$ on $\supp(\chi)$. 		
		The operator $[P_{\mathsf{b},\pot,\varphi}, \chi]$ is a differential operator with derivatives of order at most $1$, no time derivatives, and with $L^\infty$ coefficients supported in $\supp(\nabla_\x\chi)$ where $\chi_1=1$. We then obtain
		\begin{align*}
			|(e^{\tau \phi} [P_{\mathsf{b},\pot,\varphi} , \chi]u, \theta)_{L^2(\R^{n+1})}|&=\left| \int e^{\tau \phi} [P_{\mathsf{b},\pot,\varphi} , \chi]u \overline{\theta}dtdx \right|=\left| \int  [P_{\mathsf{b},\pot,\varphi} , \chi](\chi_{3}u) e^{\tau \phi} \chi_1 \chi_2 \overline{\theta}dtdx \right| \\
			&=|([P_{\mathsf{b},\pot,\varphi}, \chi](\chi_{3}u), e^{\tau \phi}\chi_1 \chi_2 \theta)_{L^2(\R^{1+d})}|\leq \nor{[P_{\mathsf{b},\pot,\varphi}, \chi](\chi_{3}u)}{H^{-1}_tL^2}\nor{e^{\tau \phi}\chi_1 \chi_2 \theta}{H^1_tL^2_x} \\
			&\leq C \tau e^{-\eta \tau/2} \nor{\chi_{3}u}{H^{-1}_tH^1_x}\nor{\theta}{H^1_tL^2_x} \leq Ce^{-\eta \tau/4} \nor{\chi_{3}u}{H^{-1}_tH^1_x}\nor{\theta}{H^1_tL^2_x},
		\end{align*}
		where we have used~\eqref{e:support-phi-psi-2} as well as the support properties of $\nabla_\x \chi, u, \chi_1, \chi_2$. 
		Coming back to~\eqref{norm by duality} we have thus obtained the estimate
		$$
		\nor{e^{\tau \phi} [P_{\mathsf{b},\pot,\varphi} , \chi]u }{H^{-1}_t L^2_x} \leq Ce^{-\eta \tau/4} \nor{\chi_{3}u}{H^{-1}_tH^1_x}.
		$$
		Similarly, one has
		$$
		e^{-\mathsf{d}\tau}\nor{e^{\tau\phi}w}{H^{-1}_t H^1_x}\leq e^{-\mathsf{d}\tau/8} \nor{\chi_{3}u}{H^{-1}_t H^1_x}.
		$$
		Combining the last two estimates with~\eqref{est aux for l2 sol} and using Lemma~\ref{imroved regularity for L2 solution} gives the existence of some $\delta>0$ with
		$$
		\|Q_{\e,\tau}^{\phi}w\|_{L^2} \leq Ce^{-\delta \tau}\nor{\chi_{3}u}{H^{-1}_t H^1_x} \leq  Ce^{-\delta \tau} \nor{u}{L^2(\Omega)}.
		$$
		From this point forward, the conclusion of the proof of Theorem~\ref{theorem local L2} is identical to that of Theorem~\ref{theorem local uniqueness intro}. 
	\end{proof}

	In the course of the proof, we have used the following elliptic regularity lemma. It is rather classical, but we provide with a short proof for sake of completeness.
	\begin{lemma}
		\label{lm:ellipticH1H-1}
		Let $V \subset \R^d$ be an open set, assume $\mathsf{g}^{jk}\in W^{1,\infty}_{\loc}(V;\R)$ satisfies~\eqref{e:elliptic}, that $\varphi \in W^{1,\infty}_{\loc}(V;\R)$ satisfies $\varphi >0$ on $V$, and let $\chi\in C^{\infty}_{c}(V)$. 
		Then, there exists $C>0$ so that, for any $w\in L^{2}(V;\C)$ with  $\chi \Delta_{\mathsf{g},\varphi}(w)\in H^{-1}(\R^{d})$, we have 
		$$
		\nor{\chi w}{H^{1}(\R^d)} \leq C \nor{ \chi  \Delta_{\mathsf{g},\varphi}(w)}{H^{-1}(\R^d)}+C \nor{w}{L^2(\supp(\chi))}.
		$$
	\end{lemma}
	Recall (see e.g. Section~\ref{s:rem-div}) that  $ \Delta_{\mathsf{g},\varphi} =\div_\varphi \nabla_{\mathsf{g}}  = \sum_{jk}\frac{1}{\varphi} \d_{x_j} \mathsf{g}^{jk} \varphi \d_{x_k}$. 
	Note that, for any $\varphi$ and $\mathsf{g}$ as in the statement, there is a Lipschitz continuous Riemannian metric $g$ such that $\mathsf{g}\varphi = g \sqrt{\det(g)}$ (namely $g := \det(\mathsf{g}\varphi)^{-\frac{1}{d+2}}\mathsf{g}\varphi$) and for this $g$ we have $\Delta_{\mathsf{g},\varphi} = \frac{\sqrt{\det(g)}}{\varphi} \Lap = \det(\mathsf{g}\varphi)^{\frac{2}{d+2}}\Lap$. In this expression (and in the setting of Lemma~\ref{lm:ellipticH1H-1}), the prefactor is a Lipschitz nonvanishing  function. Since multiplication by a $W^{1,\infty}$ function is bounded on $H^{-1}$ (for it is on $H^1$),  it suffices to prove the result of Lemma~\ref{lm:ellipticH1H-1} for $\Lap$ (defined at the beginning of Section~\ref{s:Riemtool}) in place of $\Delta_{\mathsf{g},\varphi}$. 

	\bnp
	We may assume $w\in C^{\infty}_c(V;\R)$, the conclusion of the lemma will follow from a density argument, together with application of the result to the real and imaginary parts of the function. By integration by parts, using the notation of Section~\ref{s:Riemtool}, we have
	\begin{align*}
	\int \gln{\nablag(\chi w)}= -\int \Lap  (\chi w) \chi w=-\int  \Lap  (w) \chi^{2} w- (\Lap \chi) \chi w^2 -2\gl{\nablag\chi}{\nablag w}\chi w.
	\end{align*}
	Rewritting $\gl{\nablag\chi}{\nablag w}\chi w= \gl{\nablag\chi}{\nablag (\chi w)}  w- \gln{\nablag\chi} w^2$, we deduce  
	\begin{align*}
	\int \gln{\nablag{(\chi w)}}&=-\int \Lap (w) \chi^{2} w+( \gln{\nablag\chi}- \Lap \chi  \chi )w^2 -2\gl{\nablag\chi}{\nablag (\chi w)}  w .
	\end{align*}
	Since  $g^{jk}\in W^{1,\infty}_{\loc}(V;\R)$ and $\chi \in C^\infty_c(V)$ we have $\Lap \chi \in L^\infty(V)$. As a consequence, we have for any $\eps>0$, the existence of $C_\eps= C_\eps (\chi, g)>0$ such that
	\begin{align}
	\int \gln{\nablag{(\chi w)}}&\leq  \nor{ \chi \Lap  (w)}{H^{-1}(\R^d)}\nor{\chi w}{H^1(\R^d)}+C \nor{w}{L^2(\supp(\chi))}^{2}+C \nor{\nablag (\chi w)}{L^2(\R^d)} \nor{w}{L^2(\supp(\chi))} \nonumber\\
	&\leq C_\eps \nor{ \chi  \Lap  (w)}{H^{-1}(\R^d)}^{2}+\eps \nor{\chi w}{H^1(\R^d)}
	+ C_\eps \nor{w}{L^2(\supp(\chi))}^{2}+ \eps \nor{\nablag (\chi w)}{L^2}^{2} .
	\label{e:antipenul-elliptic} 
	\end{align}
	Using ellipticity and boundedness of $g$ on $\supp(\chi)$, we further have existence of $C_g=C_g(\chi)>1$ such that for all $w \in C^\infty_c(V)$, 
	$$
	C_g^{-1} \nor{\chi w}{H^1(\R^d)}^2 \leq  \nor{\nablag(\chi w)}{L^2}^{2} + \nor{\chi w}{L^2}^{2} \leq C_g \nor{\chi w}{H^1(\R^d)}^2 .
	$$
	Combining this with~\eqref{e:antipenul-elliptic}, we have now obtained
	$$
	C_g^{-1} \nor{\chi w}{H^1(\R^d)}^2 \leq
	C_\eps \nor{ \chi  \Lap  (w)}{H^{-1}(\R^d)}^{2}+\eps(1+C_g) \nor{\chi w}{H^1(\R^d)}
	+ (C_\eps+1) \nor{w}{L^2(\supp(\chi))}^{2} ,
	$$
	which concludes the proof of the lemma when choosing $\eps = C_g^{-1}(1+C_g)^{-1}/2$.
	\enp

	\appendix
	\section{Tools}
	\label{a:tools}
	In this appendix, we collect technical lemmata that are used along the article.
\subsection{The conclusive lemma for unique continuation}
The following is~\cite[Proposition~2.1]{Hor:97} that we state here (without proof) for the reader's convenience.
\begin{lemma}
\label{lemme d analyse harmonique}
Let $u \in L^2(\R^n)$ and let $\phi$ be a smooth real valued function. Let $(A_\tau)_{\tau >0}$ be a family of continuous bounded functions in $\R^n$, such that for any compact set $K\subset \R^n$, we have  $\| A_\tau -1\|_{L^\infty(K)} \rightarrow_{\tau \rightarrow \infty} 0$. If there exist $C,\tau_0>0$ such that
$$
\nor{A_\tau (D)e^{\tau\phi} u}{L^2}\leq C, \quad \text{ for all } \tau \geq \tau_0 , 
$$
then $\supp {u} \subset \{\phi \leq 0\}$.
\end{lemma}

	\subsection{A technical lemma on the Gaussian multiplier}
	We first recall the formula
	\begin{equation}
		\label{fourier of a gaussian}
		\mathcal{F}(e^{-\frac{|\cdot|^2}{\lambda}})(\xi)=(\pi \lambda)^{1/2}e^{-\lambda \frac{|\xi|^2}{4}}, \quad \xi \in \R ,
	\end{equation}
	used several times in the article, and its consequence
	\begin{equation}
		\label{e:gaussian-convol}
		\left( e^{-\frac{h}{2}|D_t|^2}f \right) (t)=\left(\frac{1}{2\pi h}\right)^{1/2}\int_\R f(s)e^{-\frac{|t-s|^2}{2h}}ds , \quad t \in \R .
	\end{equation}
	
	\begin{lemma}
		\label{lemma 2.4 from ll}
		Let $(\X,\|\cdot\|_\X)$ be a normed vector space, $\chi_1, \chi_2 \in C^{\infty}(\R)$ with all derivatives bounded and such that  $\dist(\supp{f_1},\supp{f_2})\geq d>0$. Then for every $k,m \in \mathbb{N}$, there exist $C,c >0$ such that for all 
		$u \in \mathcal{S}(\R;\X)$ and all $\lambda > 0$ we have 
		$$
		\nor{\chi_1 e^{-\frac{|D_t|^2}{\lambda}}(\chi_2u)}{H^k(\R;\X)} \leq Ce^{- c\lambda}\nor{u}{H^{-m}(\R;\X)}.
		$$
	\end{lemma}
	See e.g.~\cite[Lemma~2.4]{LL:15} in case $m=k=0$. 
	\begin{proof}
		We start with $k=m=0$ and recall~\eqref{e:gaussian-convol}.
		Using the support properties of $\chi_1, \chi_2$, this implies
		\begin{align*}
			\chi_1 e^{-\frac{|D_t|^2}{\lambda}}(\chi_2u)(t)&=\left(\frac{\lambda}{4\pi }\right)^{1/2}\chi_1(t)\int_{|t-s|\geq d} e^{-\frac{\lambda}{4}(s-t)^2} \chi_2(s) u(s)ds \\
			&=\left(\frac{\lambda}{4\pi }\right)^{1/2}\chi_1(t) \mathds{1}_{|\cdot|\geq d}e^{-\frac{\lambda}{4}(\cdot)^2} * \chi_2(\cdot)(t) 
		\end{align*}
		The Young  inequality thus yields
		\begin{align*}
			\nor{\chi_1 e^{-\frac{|D_t|^2}{\lambda}}(\chi_2u)}{L^2(\R;\X)}&\leq  \left(\frac{\lambda}{4\pi }\right)^{1/2} \nor{\chi_1}{L^\infty} \nor{\mathds{1}_{|\cdot|\geq d}e^{-\frac{\lambda}{4}(\cdot)^2}}{L^1(\R)} \nor{\chi_2 u}{L^2(\R;\X)} \\
			&\leq \left(\frac{\lambda}{4\pi }\right)^{1/2} \nor{\chi_1}{L^\infty} \nor{\chi_2}{L^\infty} \nor{\mathds{1}_{|\cdot|\geq d}e^{-\frac{\lambda}{4}(\cdot)^2}}{L^1(\R)} \nor{ u}{L^2(\R;\X)}.
		\end{align*}
		The result for $k=m=0$ then follows from the fact that
		\begin{align*}
			\frac{1}{2}\nor{\mathds{1}_{|\cdot|\geq d}e^{-\frac{\lambda}{4}(\cdot)^2}}{L^1(\R)}&=\int_{d}^{\infty} e^{-\frac{\lambda }{8}s^2} e^{-\frac{\lambda }{8}s^2} ds\leq e^{-\frac{\lambda }{8}d^2} \int_{0}^{\infty} e^{-\frac{\lambda }{8}s^2}ds
			\leq  \frac{Ce^{-\frac{\lambda }{8}d^2} }{\sqrt{\lambda}}\int_{0}^{\infty}e^{-s^2}ds\leq  Ce^{-c \lambda}.
		\end{align*}
		As a preparation for the general case, we prove a similar estimate if $e^{-\frac{|D_t|^2}{\lambda}}$ is replaced by $D_t^k e^{-\frac{|D_t|^2}{\lambda}}$ for $k \in \N$.
		Notice that from~\eqref{e:gaussian-convol}, we have 
		$$
		D_t^k e^{-\frac{|D_t|^2}{\lambda}}f=\left(\frac{\lambda}{4\pi }\right)^{1/2}\int_\R D_t^k e^{-\frac{\lambda}{4}(s-t)^2}f(s)ds =  \left(\frac{\lambda}{4\pi }\right)^{1/2} \sum_{0\leq k_1,k_2 \leq k} \alpha_{k_1,k_2} \int_\R  \lambda^{k_1}(s-t)^{k_2}e^{-\frac{\lambda}{4}(s-t)^2}f(s)ds ,
		$$
		where $\alpha_{k_1,k_2}\in \C$ do not depend on $\lambda$. As a consequence, proceeding as above with the Young inequality, we obtain
		\begin{align*}
			\nor{\chi_1 D_t^k e^{-\frac{|D_t|^2}{\lambda}}(\chi_2u)}{L^2(\R;\X)}&\leq  C_k \lambda^{k+1/2} \sum_{0\leq  k_2 \leq k}
			\nor{ \chi_1(t) \mathds{1}_{|\cdot|\geq d}e^{-\frac{\lambda}{4}(\cdot)^2} (\cdot)^{k_2}*(\chi_2u)(t) }{L^2(\R;\X)} \\ 
			& \leq C_k \lambda^{k+1/2}  \nor{\chi_1}{L^\infty} \nor{\chi_2}{L^\infty} \sum_{0\leq  k_2 \leq k}\nor{\mathds{1}_{|\cdot|\geq d}e^{-\frac{\lambda}{4}(\cdot)^2} (\cdot)^{k_2}}{L^1(\R)} \nor{u}{L^2(\R;\X)} .
		\end{align*}
		Using now 
		\begin{align*}
			\nor{\mathds{1}_{|\cdot|\geq d}e^{-\frac{\lambda}{4}(\cdot)^2} (\cdot)^{k_2}}{L^1(\R)}&=2 \int_{d}^{\infty} e^{-\frac{\lambda }{8}s^2} e^{-\frac{\lambda }{8}s^2} s^{k_2} ds\leq 2 e^{-\frac{\lambda }{8}d^2} \int_{0}^{\infty} s^{k_2} e^{-\frac{\lambda }{8}s^2}ds \\
			&=e^{-\frac{\lambda }{8}d^2}\left(\frac{8}{\lambda}\right)^{\frac{k_2+1}{2}} \Gamma\left(\frac{k_2+1}{2}\right)\leq C_{k_2} e^{-c_{k_2} \lambda},
		\end{align*}
		Combining these two lines, we finally deduce that for any $k \in \N$ and any $\chi_1, \chi_2 \in L^{\infty}(\R)$ such that  $\dist(\supp{f_1},\supp{f_2})\geq d>0$, there are $C_k,c_k>0$ such that for all $u \in \mathcal{S}(\R;\X)$,
		\begin{align}
			\label{e:dk-expo}
			\nor{\chi_1 D_t^k e^{-\frac{|D_t|^2}{\lambda}}(\chi_2u)}{L^2(\R;\X)}&\leq  C_{k} e^{-c_{k} \lambda}\nor{u}{L^2(\R;\X)} .
		\end{align}
		Now, we prove the following statement: for all $k,\ell,m \in \N$, for all $\chi_1, \chi_2 \in C^\infty_b(\R)$ such that  $\dist(\supp{f_1},\supp{f_2})\geq d>0$ there are $C,c>0$ such that for all $u \in \mathcal{S}(\R;\X)$,
		\begin{align}
			\label{e:dk-expo-bis}
			\nor{ D_t^\ell \chi_1D_t^k  e^{-\frac{|D_t|^2}{\lambda}}(\chi_2 D_t^m u)}{L^2(\R;\X)}&\leq  C e^{-c \lambda}\nor{u}{L^2(\R;\X)} .
		\end{align}
		To this aim, given $\ell,m \in \N$, we consider the induction assumption
		\begin{align}
			\label{e:induct-asspt}
			\tag{$A(\ell,m)$}
			\text{\eqref{e:dk-expo-bis} is satisfied for all }  k \in \N.
		\end{align} 
		We notice first that $(A(0,0))$ is~\eqref{e:dk-expo}. Then, we assume~\eqref{e:induct-asspt} and prove $(A(\ell+1,m+1))$. For this, we decompose and expand
		\begin{align*}
			D_t^{\ell+1} \chi_1D_t^k  e^{-\frac{|D_t|^2}{\lambda}} \chi_2 D_t^{m+1} & =  D_t^{\ell} \big( \chi_1D_t   +[D_t ,\chi_1] \big)D_t^k  e^{-\frac{|D_t|^2}{\lambda}} \big(D_t \chi_2 + [\chi_2,D_t] \big) D_t^{m} \\
			& = 
			D_t^{\ell}   \chi_1 D_t^{k+2}  e^{-\frac{|D_t|^2}{\lambda}}  \chi_2  D_t^{m} 
			+ i 
			D_t^{\ell}  \chi_1  D_t^{k+1}  e^{-\frac{|D_t|^2}{\lambda}}  \chi_2'  D_t^{m} \\
			& \quad 
			-
			i D_t^{\ell} \chi_1'  D_t^{k+1}  e^{-\frac{|D_t|^2}{\lambda}}  \chi_2 D_t^{m} 
			+
			D_t^{\ell} \chi_1' D_t^k  e^{-\frac{|D_t|^2}{\lambda}} \chi_2' D_t^{m} ,
		\end{align*}
		and notice that the induction assumption~\eqref{e:induct-asspt} applies to all of these four terms since $\supp \chi_j' \subset \supp\chi_j$, $j=1,2$.
		This concludes the proof of~\eqref{e:dk-expo-bis}.
		
		To conclude the proof of the lemma, we deduce from~\eqref{e:dk-expo-bis} (for $k=0$) that for $\ell,m\in \N$, and all $v \in \mathcal{S}(\R;\X)$,
		\begin{align*}
			\nor{(1+D_t^2)^\ell \chi_1 e^{-\frac{|D_t|^2}{\lambda}}(\chi_2 (1+D_t^2)^m v)}{L^2(\R;\X)}&\leq  C e^{-c \lambda}\nor{v}{L^2(\R;\X)} .
		\end{align*}
		Letting $v := (1+D_t^2)^{-m} u$ in this expression, we deduce that for all $u \in \mathcal{S}(\R;\X)$,
		\begin{align*}
			\nor{\chi_1 e^{-\frac{|D_t|^2}{\lambda}}(\chi_2  u)}{H^{2\ell}(\R;\X)}&  = 
			\nor{(1+D_t^2)^\ell \chi_1 e^{-\frac{|D_t|^2}{\lambda}}(\chi_2 (1+D_t^2)^m v)}{L^2(\R;\X)}\\
			&\leq  C e^{-c \lambda}\nor{v}{L^2(\R;\X)} = 
			Ce^{- c\lambda}\nor{u}{H^{-2m}(\R;\X)}.
		\end{align*}
		This concludes the proof of the lemma (for even integers, and thus for all integers).
\end{proof}

\subsection{A complex analysis lemma}
The following regularity lemma is used in the conjugation argument.
\begin{lemma}
\label{l:regC1}
Let $U\subset \C$ an open set containing $0$ and $h\in C^{2}(U)$ such that $|\d_{\bar z}h(z)|=o(|z|)$ as $z \to 0$. Then, the function defined by 
$$w(z) :=\frac{h(z)-h(0)}{z}, \text{ for }z \neq 0 , \quad \text{ and }\quad w(0) = \d_z h(0) ,$$ satisfies $w\in C^1(U)$.
\end{lemma}
\bnp
The only problem is close to $z=0$ and may thus assume that $U$ is a small open ball centered at $0$. We write the Taylor formula $h(z)=h(0)+z\int_0^1\partial_{z}h(sz)ds+\bar{z}\int_0^1\partial_{\bar{z}}h(sz)ds$ and obtain 
\begin{align}
\label{e:zzbarderiv-0}
w(z)=\int_0^1\partial_{z}h(sz)ds+\frac{\bar{z}}{z}\int_0^1\partial_{\bar{z}}h(sz)ds , \quad z \neq 0 .
\end{align}
The first term in the right-hand side is of class $C^{1}$ by assumption and we only need to prove that the second term $u(z) :=\frac{\bar{z}}{z}\int_0^1\partial_{\bar{z}}h(sz)ds$ can be extended as a $C^1$ function near $0$. The assumption $|\d_{\bar z}h(z)| =o(|z|)$ implies that $u(z)$ can be continuously extended by $0$ at $0$ so, we are left to consider the derivatives of $u$. Denoting by $\nabla$ any derivative, we have 
\begin{align}
\label{e:zzbarderiv}
\nabla u(z)=\nabla\left(\frac{\bar{z}}{z}\right)\int_0^1\partial_{\bar{z}}h(sz)ds+\frac{\bar{z}}{z}\int_0^1s\nabla\partial_{\bar{z}} h(sz)ds.
\end{align}
By assumption, $\d_{\bar z}h\in C^{1}$ and we may thus write (Taylor expansion with Peano form of the remainder) $\d_{\bar z}h(z) = \d_{\bar z}h(0) + z\d_z\d_{\bar z}h(0) + \zb\d_{\bar z}^2h(0) + o(|z|)$.
Since we further assume $|\d_{\bar z}h(z)| =o(|z|)$, we deduce that $\d_{\bar z}h(0)=0$, $\nabla \partial_{\bar{z}}h(0)=0$, and therefore $|\nabla \partial_{\bar{z}}h(z)| = o(1)$ as $z\to 0$. Since $\left|\nabla\left(\frac{\bar{z}}{z}\right)\right|\leq C |z|^{-1}$, we deduce from~\eqref{e:zzbarderiv} and $|\d_{\bar z}h(z)|=o(|z|)$ that
$$
|\nabla u(z)| \leq C |z|^{-1} \int_0^1 o(|sz|) ds+ \int_0^1s^2|z|ds \to 0, 
$$
as $z \to 0$ (note that in the first integral, we have used that, since $h$ is $C^2$, we have $o(z) = zm(z)$ with $m$ continuous near zero and $m(z)\to 0$ as $z\to 0$, together with the Lebesgue convergence theorem). This proves that $u$ is of class $C^1$ near zero and hence, coming back to~\eqref{e:zzbarderiv-0}, so is $w$ (with $\nabla w(0)= \nabla \partial_{z}h(0)$).
\enp
\subsection{Integration by parts formul\ae}

Given a bounded $C^1$ (or piecewise $C^1$) domain $\Omega \subset \C$ and a $C^1$ one form $\omega$ defined in a neighborhood of $\Omega$, we recall the Stokes formula
$$
\int_{\d\Omega} \omega = \int_\Omega d\omega.
$$
Here, $\d\Omega$ is given the orientation coming from the canonical orientation of $\C$. 

Now given a Banach space $\ban$ and a function $f_0 : \R^2 \simeq \C \to \ban$, and under the identification $f_0(x,y)=f(z,\bar{z})$, we apply the above formula with the one Banach-valued form $\omega(x,y)= f(z,\bar{z})dz$ to obtain
$$
\int_{\d\Omega} f(z,\bar{z})dz = \int_\Omega d \left( f(z,\bar{z})dz \right) =  \int_\Omega   \d_z f(z,\bar{z})dz\wedge dz +  \d_{\bar{z}} f(z,\bar{z})d\bar{z}\wedge dz =  \int_\Omega \d_{\bar{z}} f(z,\bar{z})d\bar{z}\wedge dz, 
$$
that is 
\begin{equation}
\label{stokes for dzb}
\int_{\d\Omega} f(z,\bar{z})dz=\int_\Omega \d_{\bar{z}} f(z,\bar{z})d\bar{z}\wedge dz.
\end{equation}
Here $\d\Omega$ is oriented so that $\Omega$ lies to the left of $\d\Omega$, see for instance \cite[Chapter 1, Section 1.2]{Hor:90}. 
Applying this to $f=gh$ with $g \in C^1(\C;\C), h \in C^1(\C;\ban)$, we deduce 
\begin{equation}
\label{stokes product}
\int_\Omega g  \d_{\bar{z}} h  d\bar{z}\wedge dz = \int_{\d\Omega} ghdz - \int_\Omega  h  \d_{\bar{z}} g d\bar{z}\wedge dz  .
\end{equation}
If now  $g \in C^1(\C)$ and $h \in C^1(\C;\ban)$ satisfy $hg \to 0$ at infinity  and $g  \d_{\bar{z}} h \in L^1(\C;\ban)$, $h \d_{\bar{z}} g \in L^1(\C;\ban)$, then we may choose $\Omega =B(0,R)$ and let $R \to +\infty$, yielding the following statement.
\begin{lemma}
\label{l:ipp with zbar at infinity}
Assume $\ban$ is a Banach space, $g \in C^1(\C;\C)$ and $h \in C^1(\C;\ban)$ satisfy $g  \d_{\bar{z}} h \in L^1(\C;\ban)$, $h \d_{\bar{z}} g \in L^1(\C;\ban)$  and $\int_{\d B(0,R)} \|hg\|_{\ban}(z) dz\to 0$ as $R \to +\infty$. Then 
\begin{equation}
	\label{ipp with zbar at infinity}
	\int_\C g  \d_{\bar{z}} h  d\bar{z}\wedge dz = - \int_\C  h  \d_{\bar{z}} g d\bar{z}\wedge dz  .
\end{equation}
\end{lemma}

Note finally that $z=x+iy$ and  $\bar{z}=x-iy$ so that 
$$ 
d\bar{z}\wedge dz = d(x-iy)\wedge d(x+iy) = 2i dx\wedge dy ,
$$
where $dx\wedge dy$ is the usual Lebesgue measure on $\R^2$ (oriented).

Note also that if $f$ is holomorphic, we recover the usual deformation of contour principle $\int_{\d\Omega} f(z)dz =0$.

\small
\bibliographystyle{alpha}
\bibliography{bibli}
\end{document}